\documentclass[oneside,11pt,reqno]{amsart}
\usepackage{amssymb,amsmath,amsthm,bbm,enumerate,mdwlist,url,multirow,hyperref,amsthm,paralist,mathrsfs,bm,tabularx,mathtools}
\usepackage[shortlabels]{enumitem}
\usepackage[normalem]{ulem}
\allowdisplaybreaks

\def\EE{\mathbb{E}}
\def\DD{\mathcal{D}}
\def\XX{\mathcal{X}}
\def\YY{\mathcal{Y}}
\def\FF{\mathcal{F}}
\def\GG{\mathcal{G}}

\def\BB{\mathcal{B}}
\def\dd{\mathrm{d}}
\def\PP{\mathbb{P}}
\def\QQ{\mathbb{Q}}
\def\ee{\mathsf{e}}

\def\pkLp{\mathsf{pk\text{-}Lp}}
\def\pkspLp{\mathsf{pk\text{-}spLp}}
\def\pksubo{\mathsf{pk\text{-}subo}}
\def\spLpnot{\mathsf{spLp}_{\neg\uparrow}}
\newcommand{\cm}{\mathsf{cm}_{ [0,\infty]}}

\theoremstyle{definition}
\newtheorem{definition}{Definition}
\theoremstyle{theorem}
\newtheorem{theostar}{Theorem}

\newtheorem{propositionstar}[theostar]{Proposition}

\newtheorem{proposition}[definition]{Proposition}

\newtheorem{theorem}[definition]{Theorem}
\AtBeginEnvironment{definition}{%
  \pushQED{\qed}%
}
\AtEndEnvironment{definition}{\popQED\enddefinition}
\newtheorem{corollary}[definition]{Corollary}
\numberwithin{equation}{section}
\numberwithin{definition}{section}
\theoremstyle{remark}
\newtheorem{remark}[definition]{Remark}
\AtBeginEnvironment{remark}{%
  \pushQED{\qed}%
}
\AtEndEnvironment{remark}{\popQED\endremark}

\newtheorem{question}[definition]{Question}
\AtBeginEnvironment{question}{%
  \pushQED{\qed}%
}
\AtEndEnvironment{question}{\popQED\endquestion}

\newtheorem{application}[definition]{Application}
\AtBeginEnvironment{application}{%
  \pushQED{\qed}%
}
\AtEndEnvironment{application}{\popQED\endapplication}
\newtheorem{example}[definition]{Example}
\AtBeginEnvironment{example}{%
  \pushQED{\qed}%
}
\AtEndEnvironment{example}{\popQED\endexample}
\makeatletter
\@namedef{subjclassname@2020}{%
  \textup{2020} Mathematics Subject Classification}
\makeatother
\begin{document}
\title[Positive Markov processes in Laplace duality]{Positive Markov processes in Laplace duality}

\author{Cl\'ement Foucart}
\address{CMAP, Ecole Polytechnique and   Laboratoire Analyse, G\'eom\'etrie \& Applications, Universit\'e Sorbonne Paris Nord}
\email{foucart@math.univ-paris13.fr}

\author{Matija Vidmar}
\address{Department of Mathematics, Faculty of Mathematics and Physics, University of Ljubljana and  Institute of Mathematics, Physics and Mechanics, Ljubljana, Slovenia}
\email{matija.vidmar@fmf.uni-lj.si}

\begin{abstract} 
This article develops a general framework for Laplace duality between positive Markov processes in which the one-dimensional Laplace transform of one process can be represented through that of another. We show that a process admits a Laplace dual if and only if it satisfies a certain complete monotonicity condition. Moreover, we analyse how the conventions adopted for the values of $0 \cdot \infty$ and $\infty \cdot 0$ are reflected in the weak continuity/absorptivity properties of the processes in duality at the boundaries $0$ and $\infty$. A broad class of generators admitting Laplace duals is identified, and we provide sufficient conditions under which the associated martingale problems are well-posed with the solutions being in  duality at the level of their semigroups. Laplace duality is shown to furnish a unifying structure for several generalizations of continuous-state branching processes, e.g. those with immigration or evolving in random environments. Along the way, a theorem originally due to Ethier and Kurtz -- connecting duality of generators to that of the associated semigroups -- is refined, and we provide a concise proof of the Courr\`ege form for the  pointwise infinitesimal generator of a positive Markov process whose  domain includes the exponential functions. The latter leads naturally to the notion of a Laplace symbol, which is a parsimonious encoding of the infinitesimal dynamics of the process.
\end{abstract}



\date{\today}

\maketitle
\vspace{-1cm}
 \setcounter{tocdepth}{1} 
\footnotesize
\tableofcontents
\normalsize
\vspace{-1.5cm}

\section{Introduction}
\subsection{Motivation and agenda}

Duality constitutes a central technique in the analysis of Markov processes finding application in such diverse areas as interacting particle systems, non-equilibrium statistical mechanics, stochastic population models and mathematical finance. \label{added:reference-quick} Different notions of duality appear in the literature, ranging from classical duality with respect to measures, which is closely related to time-reversal, Blumenthal and Getoor \cite{blumenthal-getoor}, to dual representations of Laplace transforms that exchange arguments and time indices of the processes, recently studied by Kuznetsov and Wang \cite{kuznetsov-wang}.

The notion considered in this work is sometimes referred to as stochastic duality. Two stochastic processes $X$ and $Y$, with state-spaces $E$ and $F$ respectively, are said to be in duality with respect to a given function $H$ defined on $E\times F$ when, for all times~$t$,
\begin{equation}\label{eq:introHduality}\mathbb{E}_x[H(X_t,y)]=\mathbb{E}^y[H(x,Y_t)],\end{equation}
 the subscript $x$ and superscript $y$ indicating the starting values of $X$ and $Y$. 

Liggett \cite{zbMATH03892344} and Ethier-Kurtz \cite{EthierKurtz} provide foundational treatments of how duality techniques (in the sense of \eqref{eq:introHduality}) are employed in the study of interacting particle systems and martingale problems, see also \cite{zbMATH08044135}. We refer the reader  to Jansen-Kurt's survey \cite{duality} and to Swart \cite{Swart} for general accounts  of duality and  its applications in the setting of Markov chains. In recent years, the theory of stochastic duality has seen substantial development, enriching both its structural foundations and range of applications. Notable advances include algebraic approaches, see Giardin{\`a} and Redig \cite{dualityliebook} and Sturm et al. \cite{zbMATH07279967}, and new duality relationships in discrete and continuous population models, see e.g. Berzunza and Pardo \cite{arXiv:2009.11820}, Cordero et al. \cite{zbMATH07581733}, Foucart et al. \cite{FoMaMa2019,FoRivWi2024} and Gonz{\'a}lez Casanova et al. \cite{zbMATH07458586}.

A well-established setting is the so-called Siegmund duality in which the $H$ of \eqref{eq:introHduality} is defined on $[0,\infty]^2$ and given by the indicator function $H(x,y)=\mathbbm{1}_{\Delta}(x,y)$ with $\Delta:=\{(x,y)\in [0,\infty]^2: x\geq y\}$. For this $H$, it is a result of 
 Siegmund \cite{MR0431386} that there exists a $[0,\infty]$-valued Markov process $Y$ in $H$-duality with a given Markov process $X$ provided  the map $[0,\infty]\ni x\mapsto \mathbb{P}_x(X_t\geq y)$ is right-continuous and nondecreasing for all $y\in [0,\infty]$. When this last condition holds the process $X$ is said to be  \textit{stochastically monotone}.  

Besides Siegmund duality, the framework of \textit{moment duality} has received a great deal of attention. The duality function $H$ is in this case  defined on $[0,1]\times \mathbb{N}_0$ and given by $H(x,n)=x^n$. It is a  setting that arises commonly in theoretical population genetics. A $[0,1]$-valued Wright-Fisher diffusion is for instance  in moment duality with  the block-counting process of a Kingman coalescent. We may refer here to M\"ohle \cite{zbMATH01396203}, Carinci et al. \cite{zbMATH06403246} and Foucart and Zhou \cite{zbMATH07734715}.

In this article we wish to consider Laplace duality of continuous-time and continuous-space positive Markov processes $X$ and $Y$, namely, the property that for all times $t$,  \eqref{eq:introHduality} holds with $H(x,y)=e^{-xy}$, i.e. for all starting points $x$ of $X$ and $y$ of $Y$, one has
\begin{equation}\label{intro:ld}
\EE_x[e^{-X_t y}]=\EE^y[e^{-xY_t}].
\end{equation} 
  Informally speaking, at the level of the generators $\XX$ of $X$ and $\YY$ of $Y$, the semigroup relation \eqref{intro:ld} corresponds
to 
\begin{equation}\label{intro:gen}
\XX_x e^{-xy}=\YY^y e^{-xy},
\end{equation}
the subscript $x$ and superscript $y$   indicating the variable in which $\XX$ and $\YY$ are acting. 

To the best of our knowledge the literature offers no specific treatment of Laplace duality in the sense of \eqref{intro:ld}-\eqref{intro:gen}.   The most well-known  class of processes for which a Laplace duality holds is the one for which the dual  $Y$ in \eqref{intro:ld} is non-random. We may then represent $Y$ under $\mathbb{P}^{y}$ as a deterministic flow $(u_t(y))_{t\in[ 0,\infty)}$ with $u_0(y)=y$, and  \eqref{intro:ld} becomes
\begin{equation*}
\EE_x[e^{-X_t y}]=e^{-xu_t(y)}.
\end{equation*}
The semigroup of the process $X$ admits the representation of the preceding display if and only if the following property holds: 
\begin{equation*}
\EE_{x+x'}[e^{-X_t y}]=\EE_{x}[e^{-X_t y}]\EE_{x'}[e^{-X_t y}].
\end{equation*}
This convolution identity 
is known as the \textit{branching property}; a Markov process satisfying it is referred to in the literature as a continuous-state branching process (CB process, for short). A classical result, due to Lamperti \cite{Lamperti2} and Silverstein \cite[Section 4]{zbMATH03294035}, states that under mild regularity assumptions on the semigroup of $X$ the function $(u_t(y))_{t\in[ 0,\infty)}$ verifies the ordinary differential equation:
\[\frac{\dd}{\dd t}u_t(y)=-\Psi(u_t(y)),\ u_0(y)=y, \quad y\in (0,\infty),\]
with $\Psi$ a L\'evy-Khintchine function of the form
\begin{equation}\label{levy-kintchine-Psi-intro}
    \Psi(y)=\int\big(e^{-uy}-1+uy\mathbbm{1}_{(0,1]}(u)\big)\nu(\dd u)+ay^2-by-c,\quad y\in [0,\infty),
\end{equation}
where $b\in \mathbb{R}$, $a\in [0,\infty)$, $\nu$ is a L\'evy measure on $(0,\infty)$ (so, $\int 1\wedge u^2 \nu(\dd u)<\infty$) and $c \in [0,\infty)$. The function $\Psi$, called the branching mechanism, gives the action of the generator $\XX$ of $X$ on the exponential functions: for all $x,y\in (0,\infty)$,
\begin{equation}
\XX_x e^{-xy}=\lim_{t\downarrow 0}\frac{\EE_x[e^{-X_t y}]-e^{-xy}}{t}
 =x\Psi(y)e^{-xy}=-\Psi(y)\frac{\dd}{\dd y}e^{-xy}=\YY^ye^{-xy},\label{basic-example-cb-determ}
\end{equation}
$\YY$ being the generator (here, a first-order  differential operator) associated to the deterministic process $Y$ with drift (as a function of position) $-\Psi$.\label{duality-for-csbp} 

If a process $X$ admits a Laplace dual $Y$ that is non-deterministic, then the branching property 
no longer holds, and new dynamics -- characterized by different generators $\mathcal{X}$ and $\mathcal{Y}$ -- emerge. Examples of such processes, which generalize CB processes to incorporate phenomena such as competition or collisions have appeared in various contexts throughout the literature. We refer to Alkemper and Hutzenthaler \cite{zbMATH05212987}, Hermann and Pfaffelhuber \cite{zbMATH07188057}, Horridge and Tribe \cite{zbMATH02119574}, Hutzenthaler and Wakolbinger \cite{MR2308333}, Foucart \cite{MR3940763, foucart2021local}, Foucart and Vidmar \cite{foucartvidmar}, Greven et al. \cite{greven2015multitypespatialbranchingmodels}. Further instances of Laplace dualities, for multivariate processes, appear (though not by name) e.g. in Watanabe \cite[Theorem~1]{watanabe-bivariate} for bivariate  branching processes  and in \cite[Definition~2.1]{duffie-etal} for affine processes, but we shall study here only the  univariate setting.\label{added-multivariate-passage}

This paper has three principal objectives.
\begin{enumerate}[(i)]
\item \label{(i)intro} To delineate precisely the notion of Laplace duality \eqref{intro:ld} and to consider the structural properties of processes in such duality, in particular the existence of Laplace duals and the role of the boundary behavior. 
\item \label{(ii)intro} To study the implications of \eqref{intro:ld} for \eqref{intro:gen} and to find a wide explicit class of operators -- candidates for generators of positive Markov processes -- satisfying~\eqref{intro:gen}.
\item \label{(iii)intro} To establish a complex of conditions ensuring that \eqref{intro:gen} can be ``integrated'' into \eqref{intro:ld} and to apply this to a range of new examples.
\end{enumerate}
Our exposition of \ref{(i)intro} can be considered quite definitive, making  no sacrifices on generality. For \ref{(ii)intro}-\ref{(iii)intro} we shall be less exhaustive,  focusing predominantly, but not exclusively, on processes for which $\infty$ is absorbing and, more often than not, for which there is some continuity at $0$.

\subsection{Main results and article structure}\label{subsection:main-results}
Before presenting the highlights of this paper let us outline a key difficulty that we will encounter. 

 We may consider the (potentially) Laplace dual processes $X$ and $Y$ with the end-points $0$, $\infty$ either included or excluded from their respective state-spaces. In the former case we may then have to agree on (natural) conventions for each of the two expressions $e^{-0\cdot \infty}$, $e^{-\infty\cdot 0}$. Now, on the hand, there is no extension of the  map $(0,\infty)^2\ni (x,y)\mapsto e^{-xy}$ to either of the two points $(0,\infty)$, $(\infty,0)$ that is  separately continuous in both variables (since $\lim_{\epsilon\downarrow 0}\epsilon\cdot\infty=\infty \ne 0=\lim_{n\to\infty} 0\cdot n$), which makes it seem safest to restrict to pairs of state-spaces for $X$ and $Y$ that avoid the issue of specifying $e^{-0\cdot\infty}$ and $e^{-\infty\cdot 0}$ altogether. But, as we will see, one  cannot always  work  with a $(0,\infty)$-valued dual process $Y$ (or $X$). We are thus led to taking the full state-space $[0,\infty]$ as the generic choice for both $X$ and $Y$,  with due attention having to be paid to the two natural conventions for $e^{-0\cdot\infty}$ (either  $e^{-0\cdot \infty}=e^{-0}=1$ or $e^{-0\cdot \infty}=e^{-\infty}=0$) and (analogously) for~$e^{-\infty\cdot 0}$.

After the preparatory   Section~\ref{sec:preliminaries}, that recalls some relevant concepts and fixes notation, we explore the consequences of Laplace duality for the associated semigroups in Section~\ref{sec:Laplacedualitysemigroups}. We find that  common to all  processes in Laplace duality is the property of \emph{complete monotonicity}, which means, roughly speaking, that the semigroup of the process  leaves invariant the class of completely monotone functions. Moreover, we establish in  Theorem~\ref{thm:completemonotonicity} that
 \begin{align}\nonumber X  &\text{ admits a Laplace dual under a given set of conventions for $e^{-0 \cdot \infty}$ and $e^{-\infty \cdot 0}$} \\ \nonumber
 \Leftrightarrow &\text{ $X$ is completely monotone and certain weak continuity/absorptivity properties}\\*
 & \text{ hold for it at the boundaries $0$, $\infty$}.\nonumber 
 \end{align}
The weak continuity properties at the boundaries and whether or not these are absorbing will be seen to -- and we shall make it explicit how they -- depend crucially on the different combinations of conventions for $e^{-0 \cdot \infty}$ and $e^{-\infty \cdot 0}$. 

For processes $X$ and $Y$ in Laplace duality a correspondence is established in Corollary~\ref{cor:absorbing-nonsticky}  between the behavior of $X$ at $0$ and $\infty$ (absorbing/non-sticky) and the behavior of $Y$ at $\infty$ and $0$ (non-sticky/absorbing), respectively. The complete monotonicity of processes $X$ and $Y$ in Laplace duality entails also relations between  excessive (resp. invariant) functions of $X$ and  excessive (resp. invariant) measures of $Y$, which is explained in Proposition~\ref{thm:excessivecm}. Classical links between the long-term limiting laws of $X$ and $Y$ are described in Proposition \ref{prop:LToflimitingdistribution}. This achieves our aim \ref{(i)intro}.

In pursuing the second objective \ref{(ii)intro}, to which we turn in Section~\ref{sec:LDgen}, we are confronted with a further obstacle, namely that we start with a general positive Markov process without any a priori information about the structure of its infinitesimal generator. This is particularly problematic when attempting to identify operators $\XX$ and $\YY$ such that \eqref{intro:gen} holds. We therefore establish in Subsection~\ref{sec:courrage}, Theorem~\ref{thm:courrege}, that the  pointwise infinitesimal generator $\XX$ of any Markov process $X$ with state-space $[0,\infty]$, whose domain includes the exponential functions, admits the so-called Courr\`ege form when restricted to them.  On compactly supported twice differentiable functions such a  form is given for the strong infinitesimal generator of a real-valued Feller process in Courrège \cite{zbMATH03249220}, see also Jacob \cite[p.~360, Theorem 4.5.21]{jacob2001pseudo}, but for our level of generality the pointwise generator $\XX$ is more appropriate.\label{add-stuff} As in the Feller case, the action of $\XX$ on the exponentials is composed of two parts: a local one, corresponding to diffusion and drift,  and a non-local one, encoding the jumps -- not necessarily positive -- and governed by a family of L\'evy measures, all encapsulated in what we call the \emph{Laplace symbol} $\psi_X$ of~$X$:
$$\psi_X(x, y) := e^{xy} \mathcal{X}_x e^{-x y} = \lim_{t \downarrow 0} \frac{\mathbb{E}_x[e^{-(X_t - x)y}] - 1}{t}.$$
For instance, the Laplace symbol of a $\mathrm{CB}$ process $X$ with branching mechanism $\Psi$  satisfies, cf. \eqref{basic-example-cb-determ}, $\psi_X(x,y)=x\Psi(y)$.

 Having clarified the possible issues at the boundaries with the help of the conventions for $e^{-0\cdot \infty}$, $e^{-\infty\cdot 0}$ and having identified the form of the  pointwise  generator for any positive Markov process whose domain includes the exponential functions, we are  positioned well  to explore the relationship between \eqref{intro:ld} and \eqref{intro:gen}.

In Subsection~\ref{subsec:geninLaplaceduality} we examine the basic structure of the pointwise generators $\XX$ and $\YY$ of two processes $X$ and $Y$ with associated Laplace symbols $\psi_X$ and $\psi_Y$ that are in Laplace duality. With no surprise, we shall prove in Theorem~\ref{thm:generator-duality-semigroup} that, modulo technical reservations,  
\begin{align*}
&\text{$X$ in Laplace duality \eqref{intro:ld} with $Y$}  \Rightarrow \text{$\XX$ in Laplace duality \eqref{intro:gen} with $\YY$}\\
\shortintertext{and}
&\text{$\XX$  in Laplace duality  \eqref{intro:gen}  with $\YY$} \Leftrightarrow \text{$\psi_X(x,y)=\psi_Y(y,x)$}.
\end{align*}
When $\psi_X(x,y)=\psi_Y(y,x)$ we say that the Laplace symbol $\psi_X$ of $X$ (or, for that matter, $\psi_Y$ of $Y$) is a \textit{Laplace dual symbol}. In Theorem~\ref{theorem:tensor-form} we characterize the Laplace dual symbols $\psi$ that take the separable form $\psi(x, y) = R(x)\Psi(y)$ for some $\Psi$ as in \eqref{levy-kintchine-Psi-intro}  and some function $R$. It will emerge that  $R$ must itself be of L\'evy–Khintchine type. We recover in Theorem~\ref{thm:localgeninduality} a result, first obtained in \cite{foucartvidmar}, characterizing positive Markov processes with no negative jumps that admit for their Laplace dual a diffusion. They are the so-called continuous-state branching processes with collisions. 
It is finally observed in Subsection~\ref{subsec:geninLaplaceduality}  that the symbol of a multiplicative L\'evy process is also Laplace dual.

Building on the Laplace dual symbols identified in Subsection~\ref{subsec:geninLaplaceduality}  we introduce in Definition~\ref{def:LDS} of Subsection~\ref{sec:wideclass} a wide class, denoted $ \mathsf{LDS}$, of Laplace dual symbols, constructed from functions of the L\'evy–Khintchine type, and forming a convex cone, which secures our second goal \ref{(ii)intro}.

Lastly, we address our third aim \ref{(iii)intro} in Sections \ref{sec:fromdualsymboltosemigroups}-\ref{sec:examples}. They contain directly applicable results allowing one to deal with new examples.

We start by designing a set of conditions allowing one to transfer the Laplace duality of two operators $\XX, \YY$ to the level of semigroups, provided that processes exist as solutions to the martingale problems associated to $\XX$ and $\YY$. These conditions are obtained  by first refining  Ethier-Kurtz's theorems on duality \cite[Chapter 4]{EthierKurtz} for a general function in  Theorem~\ref{theorem:duality-from-mtgs} and then by applying it to the specific setting of the exponential duality map in Theorem \ref{thm:generators}. Subsequently, processes whose Laplace symbols belong to a certain subclass of $\mathsf{LDS}$ are  constructed as solutions to well-posed martingale problems in Theorem~\ref{thm:LaplacedualitysemigroupLDS}. The existence of solutions follows from the positive maximum principle and a well-known general result. Laplace duality between  solutions of the martingale problem and its dual is  established using Theorem \ref{thm:generators}.  The duality in turn guarantees uniqueness of the solutions,
as well as their Markov property. 

In Section~\ref{sec:examples} we study specific examples via stochastic equations with jumps. The CB processes and subordinators are for instance revisited through the lens of our duality framework. We  observe that several generalizations of CB processes that have been defined in the literature -- adding phenomena such as collision \cite{foucartvidmar}, immigration \cite{KAW}, or L\'evy random environment \cite{zbMATH06247275,zbMATH06963718,zbMATH06836271} --  admit Laplace duals. We then focus on a specific subclass of completely monotone Markov processes, whose Laplace dual symbols, belonging to $\mathsf{LDS}$, are baptized \textit{decomposable}, providing sufficient conditions for uniqueness and non-explosion of solutions to the associated  stochastic equations.  Theorem \ref{thm:generators}  applies and ensures   duality at the level of semigroups, leading to a new generic family of stochastic equations in Laplace duality.

\section{Preliminaries}\label{sec:preliminaries}
In this section, the reader interested only in the results (not their proofs) concerning Laplace duality need only consult the definitions and pick up the notation. 
\subsection{Miscellaneous general notation and  conventions}\label{miscellaneous} 
We understand $e^{-\infty}:=0$, also $x\cdot \infty:=\infty=:\infty\cdot x$ for $x\in (0,\infty)$. The symbol $0^+\cdot\infty$ shall refer to the use of the convention $0\cdot\infty=\lim_{\epsilon\downarrow 0}\epsilon\cdot\infty=\infty$ \emph{in interpreting  $e^{-0\cdot\infty}$}, the superscript indicating then in what limiting sense the product is to be understood. Analogously we introduce $0\cdot \infty^-$, $\infty^-\cdot 0$ and $\infty\cdot 0^+$. We stress that we can have  $e^{-0\cdot \infty}\ne e^{-\infty\cdot 0}$, e.g. under $0^+\cdot\infty$, $\infty^-\cdot 0$.  As a mnemonic device, notice that  it is the number that is \emph{not} ``adorned'' with a $\pm$ that results from the product, e.g. under $0\cdot \infty^-$, we interpret $e^{-0\cdot \infty}=e^{-0}=1$. Henceforth, whenever referring to conventions for $0\cdot\infty$ and $\infty\cdot 0$, we always implicitly mean them only with reference to interpreting $e^{-0\cdot\infty}$ and $e^{-\infty\cdot 0}$. The standard measure-theoretic convention $0\cdot \infty=0=\infty\cdot 0$  shall continue to prevail in all other respects.

The notation $\uparrow$ (resp. $\uparrow\uparrow$) means nondecreasing, more precisely ``increasing, equalities allowed'' (resp. strictly increasing, i.e. ``increasing, equalities not allowed''); in a similar vein we interpret $\downarrow$ (resp. $\downarrow\downarrow$). When we use e.g.  $\text{$\uparrow$-}\lim$ it means that the limit in question is nondecreasing. The domain  of a function $f$ is denoted $D_f$. For an extended real-valued map $\phi$ we put  $\phi_+:=\max(\phi,0)$ and~$\phi_-:=\max(-\phi,0)$.   We understand $\inf\emptyset:=\infty$.  

For a topological space $E$: 
\begin{align}\nonumber
\BB_E&:=  \text{the Borel $\sigma$-field on $E$},\\\nonumber
\mathsf{M}(E)&:=\{\text{Borel measurable real functions on $E$}\},\\\nonumber
\mathsf{B}(E)&:=\{\text{bounded measurable real functions on $E$}\},\\\nonumber
\mathsf{C}_0(E)&:=\{\text{continuous real maps on $E$ vanishing at infinity}\}  \text{ in case $E$ is locally compact},\\\nonumber
\mathsf{C}(E)&:=\{\text{continuous real maps on $E$}\}  \text{ in case $E$ is compact, and}\\
\mathbb{D}_E&:=\{\omega\in E^{[0,\infty)}:\omega\text{ c\`adl\`ag}\}\text{ endowed with the $\sigma$-field of evaluation maps.}\label{various-E-spaces}
\end{align}  

The Laplace transform  of a  measure $\mu$ on $\BB_{A}$, $A\subset [0,\infty]$, is the map $\widehat \mu:(0,\infty)\to [0,\infty]$ given by \[\widehat{\mu}(x):=\int e^{-xv}\mu(\dd v),\quad x\in (0,\infty). \]   A measure $\mu$ on a measurable space $(S,\Sigma)$ is said to be carried by a $T\in \Sigma$ if $\mu(S\backslash T)=0$.

\subsection{Markov processes, their semigroups, transition kernels and generators}\label{sec:positivemarkov}
Given a measurable space $(E,\mathcal{E})$,  a Markov process $X$ under the probabilities $\PP_\cdot=(\PP_x)_{x\in E}$ (being all of them defined on a common measurable space) valued in $(E,\mathcal{E})$, is for us an $E$-valued process indexed by $[0,\infty)$ (infinite lifetime) such that: 
\begin{enumerate}[(I)]
    \item measurability: the map $(E\ni x\mapsto \PP_x(X_t\in A))$ is measurable for all $t\in [0,\infty)$ and all $A\in\mathcal{E}$;
    \item normality: for all $x\in E$, $\PP_x(X_0=x)=1$;  and 
    \item Markov property: for all real $t\geq s\geq 0$,   $A\in\mathcal{E}$,  $x\in E$, we have
\begin{equation}\label{markov-property}
\PP_x\left(X_t\in A\vert \sigma(X_u:u\in [0,s])\right)=\PP_{X_s}(X_{t-s}\in A) \text{ a.s.-$\PP_x$.}
\end{equation}
\end{enumerate}
For $t\in [0,\infty)$ and $x\in E$ we shall then write the law of $X_t$ under $\mathbb{P}_x$ as
$$p_t(x,\cdot):={X_t}_\star\PP_x.$$
The collection $p:=(p_t(x,\cdot))_{(t,x)\in [0,\infty)\times E}$ is a Markovian family of transition kernels:  the map $(E\ni x\mapsto p_t(x,A))$ is measurable for all $A\in \mathcal{E}$ and $t\in [0,\infty)$ (measurability); $p_0(x,\cdot)=\delta_x$ for all $x\in E$ (normality);  $$p_{t+s}(x,\cdot)=\int p_s(u,\cdot)p_t(x,\dd u),\quad \{t,s\}\subset [0,\infty),\, x\in E \text{ (semigroup property)}.$$ It carries the complete information of the finite-dimensional laws of the process $X$ and its properties exactly parallel the ones listed above for the family of laws $\PP_\cdot$ of $X$. We also define, for $t\in [0,\infty)$ and $f:E\to \mathbb{R}$ bounded, measurable,
\begin{equation}\label{eq:semigroup}
P_tf(x):=\int f(u)p_t(x,\dd u),\quad  x\in E.
\end{equation} The family $P:=(P_t)_{t\in [0,\infty)}$ is  the so-called (transition) semigroup of $X$. Finally, under the pointwise infinitesimal generator  -- henceforth, just generator, except where we will want to emphasize its pointwise nature -- $\XX$ of $X$ we shall understand the linear map with domain $D_\XX$ comprising those bounded measurable $f:E\to \mathbb{R}$ for which the limit (thereafter specified as the value of $\XX$ at $f$)  
\begin{equation}\label{eq:generator}
\XX f:=\lim_{t\downarrow 0}\frac{P_tf-f}{t}
\end{equation} exists pointwisely in $\mathbb{R}$ on $E$. 

Below, when referring to a Markov process $X$ we implicitly take it for granted that it comes equipped with the notation $(\PP_\cdot,p,P,\XX)$ (unless some of these are given, explicitly, a separate meaning, in which case that one supersedes), though we will remind the reader of parts of this notation from time to time. Often $X$ will be accompanied by another Markov process $Y$, which shall then be assumed to be furnished with the analogous pieces of data $(\PP^\cdot,q,Q,\YY)$. The expectations of $X$ (resp. $Y$) we  then write as $(\EE_x)_{x\in [0,\infty]}$ (resp. $(\EE^y)_{y\in [0,\infty]}$).

While the above concepts are very general, we will, for the most part, be concerned only with the case when the state-space $E$ is $[0,\infty]$:
\begin{definition}\label{def:positive-Markov-process}
A positive Markov process is a  Markov process valued in $([0,\infty],\BB_{[0,\infty]})$. If a positive Markov process $X$ has c\`adl\`ag paths, then in saying that it is non-explosive we shall mean that $\zeta_\infty:=\inf \{t\in [0,\infty):X_{t-}\text{ or }X_t=\infty\}=\infty$ a.s.-$\PP_x$ for all $x\in [0,\infty)$.
\end{definition}

\subsection{Duality of Markov processes}\label{subsection:processes-in-duality}
Let $(E, \mathcal{E})$ and $(F, \mathcal{F})$ be measurable spaces and let $H : E \times F \to [-\infty,\infty]$ be a measurable function, bounded from below or from above. The following $H$-duality relationship was  discussed, albeit only informally, already in the Introduction, cf. \eqref{eq:introHduality}.

\begin{definition}\label{def:H-duality}
An $(E,\mathcal{E})$-valued Markov process $X$  is in $H$-duality with an $(F,\mathcal{F})$-valued Markov process $Y$   when $$\EE_x[H(X_t,y)]=\EE^y[H(x,Y_t)]$$ holds for all $t\in [0,\infty)$, $x\in E$,  $y\in F$.
\end{definition}
We submit for consideration the three blanket assumptions (it would perhaps be more natural to have them be symmetric in $E$ and $F$, but technically, for the result that we wish to present, we do not need it):
\begin{enumerate}[(A)]
\item\label{dual:A} $H$ is law-determining on $F$: for probabilities $Q_1$ and $Q_2$ on $F$, if $\int H(x,y)Q_1(\dd y)=\int H(x,y)Q_2(\dd y)$ for all $x\in E$, then $Q_1=Q_2$.
\item\label{dual:B} $H$ is kernel-measurability-determining on $F$: for $q:F\times \mathcal{F}\to [0,1]$ with $q(y,\cdot)$ a probability on $(F,\mathcal{F})$ for all $y\in F$, if the map $F\ni y\mapsto \int H(x,v)q(y,\dd v)$ is measurable for all $x\in E$, then $q$ is a probability kernel, i.e.  (in addition to the assumed property of $q$) $q(\cdot, B):F\to [0,1]$ is measurable for all $B\in \mathcal{F}$.
\item\label{dual:C} $F$ has the Kolmogorov-extension property:  for any index set $\Lambda$ and any consistent family $\rho$  of probabilities on $F^T$, as $T$ ranges over the finite subset of $\Lambda$, there exists a (automatically unique)  probability on $F^\Lambda$ extending $\rho$.
\end{enumerate}
Under \ref{dual:A}, if $X$ admits $H$-duals $Y_1$ and $Y_2$, then the law of $Y_1$ is the same as the law of $Y_2$. 
The following proposition ensures that kernels in duality with Markovian ones are themselves Markovian provided conditions \ref{dual:A}-\ref{dual:C} are satisfied. See also \cite[Proposition~3.6]{duality} and \cite[Theorem 2.1]{zbMATH08044135} for related results.
\begin{propositionstar}\label{proposition:general-dual}
Assume \ref{dual:A}-\ref{dual:C}. Let $X$ be an $(E,\mathcal{E})$-valued Markov process and let, for each $y\in F$ and $t\in [0,\infty)$, $q_t(y,\cdot)$ be a probability on $(F,\mathcal{F})$ such that for all $x\in E$, 
\begin{equation}\label{would-be-dual}
\int H(u,y)p_t(x,\dd u)=\int H(x,v)q_t(y,\dd v).
\end{equation}
Then there exists an $(F,\FF)$-valued Markov process $Y$, unique in law, such that $X$ is in $H$-duality with $Y$; this $Y$ has $q:=(q_t(y,\cdot))_{(t,y)\in [0,\infty)\times F}$ for its family of transition kernels.
\end{propositionstar}
\begin{proof}
We have already remarked that \ref{dual:A} ensures uniqueness in law of $Y$. 
Let us prove existence. As $F$ has the Kolmogorov extension property it will suffice to demonstrate that $q$ is a Markovian family of  kernels. For $y\in F$, by normality of $X$ and~\eqref{would-be-dual},   $$\int H(x,v)\delta_y(\dd v)=H(x,y)=\int H(u,y)p_0(x,\dd u)=\int H(x,v)q_0(y,\dd v),\quad x\in E;$$  since $H$ is law-determining on $F$ we infer that $q_0(y,\cdot)=\delta_y$, which is the normality of $q$. Measurability of $q$ is a direct consequence of the fact that $H$ is kernel-measurability-determining on  $F$ and of \eqref{would-be-dual}. Finally, let us check the semigroup property of $q$.  For $\{t,s\}\subset [0,\infty)$ and $x\in E$,  $y\in F$, on using \eqref{would-be-dual}, Tonelli-Fubini and the semigroup property of $p$:
\begin{align*}
 \int q_{t}(y,\dd v) \int q_{s}(v,\dd v')H(x,v')&= \int q_{t}(y,\dd v) \int p_{s}(x,\dd u)H(u,v)\\
 &=  \int p_{s}(x,\dd u)\int q_{t}(y,\dd v) H(u,v)\\
  &= \int p_{s}(x,\dd u)\int p_{t}(u,\dd u')H(u',y)\\
  &=\int p_{s+t}(x,\dd u)H(u,y)
  =\int q_{s+t}(y,\dd v)H(x,v).
\end{align*}
 Since $H$ is law-determining on $F$ this completes the proof.
\end{proof}

\subsection{L\'evy-Khintchine forms}\label{subsection:levy-khintchin}
We are obliged to provide some background on L\'evy-Khintchine functions. The notation required for the understanding of the results of the continuation of the text can be read off from Table~\ref{table.3}. 

 Let $x\in [0,\infty)$. Given a quadruplet $l=(\nu,a,b,c)$ consisting of a measure $\nu$  on  $\BB_{[-x,\infty)}$ satisfying $\int (1\land u^2)\nu(\dd u)<\infty$ and $\nu(\{0\})=0$, $a\in [0,\infty)$, $b\in \mathbb{R}$ and $c\in [0,\infty)$ (henceforth, an $x$-L\'evy quadruplet), the convex map $\kappa:[0,\infty)\to \mathbb{R}$ specified according to
 \begin{equation}\label{eq:exponents:pkLp}
 \kappa(y):=\int  \left(e^{-uy}-1+uy\mathbbm{1}_{[-1,1]}(u)\right)\nu(\dd u)+ay^2-by-c,\quad y\in [0,\infty),
 \end{equation}
  represents the Laplace exponent of a possibly killed (cemetery: $\infty$) L\'evy process $\xi$ with no jumps of size $<-x$: $\EE[e^{-y\xi_t};\xi_t<\infty]=e^{t\kappa(y)}$ for all $t\in [0,\infty)$ and $y\in [0,\infty)$ \cite[Theorems~1.6 and 3.6, Eq.~(3.11)]{Kyprianoubook}\footnote{The addition of killing at an independent exponential random time of rate $c$ is trivial.}. We denote the class of all $\kappa$ of this form by $\pkLp_{x}$. 
  The $\kappa$ of \eqref{eq:exponents:pkLp} determines the quadruplet $(\nu,a,b,c)$ uniquely: 
  
  \begin{proposition}\label{lemma:laplace-symbol-unique}
Let $x\in [0,\infty)$. If $l_1=(\nu_1,a_1,b_1,c_1)$ and $l_2=(\nu_2,a_2,b_2,c_2)$ are $x$-L\'evy  quadruplets with associated $\kappa_1$ and $\kappa_2$ according to \eqref{eq:exponents:pkLp}, then $\kappa_1=\kappa_2$ implies $l_1=l_2$.
\end{proposition}
\begin{proof}
 Writing $l=(\nu,a,b,c)$ generically for $l_1$ or $l_2$ and $\kappa$ generically for $\kappa_1$ and $\kappa_2$ in parallel, differentiation under the integral sign (via dominated convergence) yields $$\kappa'''(y)=-\int u^3e^{-uy}\nu(\dd u)\in \mathbb{R},\quad y\in (0,\infty).$$ Considering the measure $\gamma:=(\BB_{[0,\infty)}\ni A \mapsto  \int_{A-x} u^3\nu(\dd u))$, noting that finite Laplace transforms  determine signed measures on $\BB_{[0,\infty)}$ and that  $\nu(\{0\})=0$, we infer from $\kappa_{1}=\kappa_{2}$ that $\gamma_1=\gamma_2$, therefore $\nu_1=\nu_2$, and finally $(a_1,b_1,c_1)=(a_2,b_2,c_2)$. 
\end{proof}

Given a  triplet $(\nu,d,c)$ consisting of a  measure $\nu$  on   $\BB_{(0,\infty)}$ satisfying $\int (1\land u)\nu(\dd u)<\infty$, $d\in [0,\infty)$ and $c\in [0,\infty)$ (henceforth, a L\'evy triplet), the concave map  $\Phi:[0,\infty)\to \mathbb{R}$ specified according to \begin{equation}\label{Phi-form}
\Phi(y):=\int  \left(1-e^{-uy}\right)\nu(\dd u)+dy+c,\quad y\in [0,\infty),
\end{equation}  represents the Laplace exponent of a possibly killed (cemetery: $\infty$) subordinator $\xi$: $\EE[e^{-y\xi_t};\xi_t<\infty]=e^{-t\Phi(y)}$ for all $t\in [0,\infty)$ and $y\in [0,\infty)$. It is well-known and easy to verify that $\Phi$  satisfies $\lim_{y\to \infty}\frac{\Phi(y)}{y}=d$ and determines $(\nu,d,c)$ uniquely \cite[Theorem~1.2 and the paragraph just after it]{subordinators}. We denote the class of all such $\Phi$ by $\pksubo$.

We write $\pkspLp:=\pkLp_0$ for the class of Laplace exponents of possibly killed  (cemetery: $\infty$) spectrally positive L\'evy processes. In the representation \eqref{eq:exponents:pkLp} above we may then just as well and do insist that $\nu$ is a  L\'evy measure on $\BB_{(0,\infty)}$, and any $\Psi\in \pkspLp$ takes the form
\begin{equation}\label{eq:LKpsi}
    \Psi(y)=\int\big(e^{-uy}-1+uy\mathbbm{1}_{(0,1]}(u)\big)\nu(\dd u)+ay^2-by-c,\quad y\in [0,\infty),
\end{equation}
$(\nu,a,b,c)$ being the corresponding L\'evy quadruplet. Furthermore, since a two-fold integration by parts yields $  \Psi(y)=y^2\int_0^1e ^{-uy}\int _u^1\nu((v,1])\dd v\dd u-y\int_0^\infty e^{-uy}\nu((u\lor 1,\infty))\dd u+ay^2-by-c$ for $y\in (0,\infty)$, we see by dominated convergence that  such a $\Psi$ satisfies $\lim_{y\to\infty}\frac{\Psi(y)}{y^2}=a\in [0,\infty)$. Also, $\Psi(\infty-)=\infty$ unless $-\Psi\in \pksubo$ \cite[Section~VII.1, penultimate paragraph on p.~188]{Ber96} and therefore $\pkspLp\cap (-\infty,0]^{[0,\infty)}=-\pksubo$.

We further denote by $\spLpnot:=\pkspLp\cap [0,\infty)^{[0,\infty)}$ the family of Laplace exponents of spectrally positive L\'evy processes that are not drifting to $\infty$. So, every $\Sigma\in\spLpnot$ admits the representation 
 \begin{equation}\label{class-sigma}
 \Sigma(y)=\int  \left(e^{-uy}-1+uy\right)\nu(\dd u)+ay^2+dy,\quad y\in [0,\infty),
 \end{equation}
for some measure $\nu$ on $\BB_{(0,\infty)}$ satisfying $\int u\land u^2\nu(\dd u)<\infty$, $a\in [0,\infty)$ and $d\in [0,\infty)$. Via integration by parts $  \Sigma(y)=y\int_0^\infty (1-e^{-uy})\nu((u,\infty))\dd u+ay^2+dy $ for $y\in [0,\infty)$, so that any such  $\Sigma$ satisfies $\Sigma(0)=0$ and, by dominated convergence, $\lim_{y\to 0}\frac{\Sigma(y)}{y}=d\in [0,\infty)$.  Since  only affine maps are convex and concave, it follows that $\pkspLp\cap \pksubo=\spLpnot\cap \pksubo=\{\gamma\mathrm{id}_{[0,\infty)}:\gamma\in [0,\infty)\}$. 

For easier reference we summarize in Table~\ref{table.3}  the various classes of L\'evy-Khintchine forms introduced above.
\begin{table}[htb!]
\begin{center}
    \begin{tabular}{ccc}
\multicolumn{1}{c}{class} &
\multicolumn{1}{c}{generic representative of the class} &
\multicolumn{1}{c}{condition on $\nu$} \\
\hline
    $\pkLp_x$&  $\kappa(y)=\int_{[-x,\infty)}  \left(e^{-uy}-1+uy\mathbbm{1}_{[-1,1]}(u)\right)\nu(\dd u)+ay^2-by-c$ &  $ \int 1\land u^2\nu(\dd u)<\infty$\\
            $\pksubo$ & $\Phi(y)=\int_{(0,\infty)}  \left(1-e^{-uy}\right)\nu(\dd u)+dy+c$ &   $ \int 1\land u\nu(\dd u)<\infty$\\
    $\pkspLp$    &    $ \Psi(y)=\int_{(0,\infty)}\big(e^{-uy}-1+uy\mathbbm{1}_{(0,1]}(u)\big)\nu(\dd u)+ay^2-by-c$  &   $ \int 1\land u^2\nu(\dd u)<\infty$\\
    $\spLpnot$ & $\Sigma(y)=\int_{(0,\infty)}  \left(e^{-uy}-1+uy\right)\nu(\dd u)+ay^2+dy$ &   $\int u\land u^2 \nu(\dd u)<\infty$
    \end{tabular}
    \end{center}
     \caption{The classes of Laplace exponents \eqref{eq:exponents:pkLp}-\eqref{class-sigma}. The domain of definition of the measure $\nu$ is given by the delimiters in the integrals for the expressions of the generic representative of each class. In all cases the parameters satisfy $\{a,d,c\}\subset [0,\infty)$, $b\in \mathbb{R}$. Additionally, $\nu$ does not charge $0$ in the classes $\pkLp_x$, $x\in [0,\infty)$.}\label{table.3}
 \end{table}
 
\begin{remark}\label{rem:decompositionjumpmeasure}
   In \eqref{eq:LKpsi}, instead of working with a single  L\'evy measure $\nu$ and the truncating function $\mathbbm{1}_{(0,1]}$,   it will sometimes be more convenient to decompose $\nu=\nu_1+\nu_2$ into the sum of two  measures $\nu_1$ and $\nu_2$  on $\BB_{(0,\infty)}$ satisfying  $\int 1\wedge  u \nu_1(\dd u)<\infty$, $ \int u \wedge u^2 \nu_2(\dd u)<\infty$, 
also the drift parameter $b=b_1-b_2$  as a difference with $\{b_1,b_2\}\subset [0,\infty)$, allowing us to write instead, for $y\in [0,\infty)$,
 \begin{align*}
 \Psi(y)&=\int  \left(e^{-uy}-1+uy\right)\nu_2(\dd u)+ay^2+b_2y-\left(\int  (1-e^{-uy})\nu_1(\dd u)+b_1y+c\right);
 \end{align*}
 such a decomposition is always possible (for instance by taking for $\nu_1$ the part of $\nu$ outside of $(0,1]$ and  $b_1:=b_+$), though not unique. Note that $\Psi$ can thus be written as $\Sigma-\Phi$ with $\Sigma\in \spLpnot$ and $\Phi\in \pksubo$.
    \end{remark}

We will have an opportunity to use up
\begin{proposition}\label{proposition:closure-under-limits}
For each $x\in [0,\infty)$ the class $\pkLp_x$ is closed for pointwise convergence (of sequences) towards a function $[0,\infty)\to \mathbb{R}$ continuous at zero; moreover, $\pkLp_x\vert_{(0,\infty)}$ is closed for pointwise limits.
\end{proposition}
The pointwise limit of convex maps is convex and the limit of nonpositive numbers is nonpositive; each member of $\pkLp_x$ being convex and nonpositive at zero, we see that the assumption of continuity at zero in the preceding is redundant if the limiting function is nonnegative on $(0,\infty)$. In particular, $\spLpnot$ is closed for pointwise convergence of sequences. More generally, given a sequence $(\kappa_n)_{n\in \mathbb{N}}$  in $\pkLp_x$ converging pointwise in $\mathbb{R }$ to some $\kappa:[0,\infty)\to \mathbb{R}$, the convexity ensures that  automatically $\kappa(0+)\leq \kappa(0)$. The possibility of $\kappa(0+)<\kappa(0)$ reflects that we may have extra mass ``escaping to $\infty$'' in the limit, explains why we need the restriction $\vert_{(0,\infty)}$ in the second statement of the proposition and dovetails with  $\pksubo\vert_{(0,\infty)}$, the class of Bernstein functions, being closed for pointwise convergence \cite[Corollary~3.8(ii)]{bernstein}. \label{cf-what-is-meant}
\begin{proof}[Proof of Proposition~\ref{proposition:closure-under-limits}]
Let  $(\kappa_n)_{n\in \mathbb{N}}$ be a sequence in $\pkLp_x$ converging pointwise in $\mathbb{R}$ on $(0,\infty)$ to some  $\kappa: (0,\infty)\to \mathbb{R}$. For each $n\in \mathbb{N}$ consider the  infinitely divisible distribution $Q_n$ on $\BB_\mathbb{R}$ associated to $\kappa_n-\kappa_n(0)$ satisfying $$\int e^{-uy}Q_n(\dd u)=e^{\kappa_n(y)-\kappa_n(0)},\quad y\in [0,\infty)$$ (i.e. $Q_n$ is  the law at time $1$ of the L\'evy process having the Laplace exponent $\kappa_n-\kappa(0)$).
Fix $\theta\in (0,\infty)$ and introduce the exponentially tilted infinitely divisible \cite[Theorem~3.9]{Kyprianoubook} probabilities $Q^\theta_n:=(\mathbb{R}\ni u\mapsto e^{-u\theta-(\kappa_n(\theta)-\kappa_n(0))})\cdot Q_n$, $n\in \mathbb{N}$, whose representing L\'evy measures are still carried by $[-x,\infty)$. We see that
\begin{align*}
\lim_{n\to\infty}\int e^{-uy}Q^\theta_n(\dd u)&=\lim_{n\to\infty}\int e^{-u(y+\theta)-(\kappa_n(\theta)-\kappa_n(0))}Q_n(\dd u)=\lim_{n\to\infty}e^{\kappa_n(y+\theta)-\kappa_n(\theta)}\\
&=e^{\kappa(y+\theta)-\kappa(\theta)},\quad y\in (-\theta,\infty).
\end{align*}
So, the moment generating functions of the $Q^\theta_n$ are convergent towards the specified limit as $n\to \infty$ on the neighborhood $(-\infty,\theta)$ of zero. By a result of Curtiss \cite[Theorem~3]{curtiss} we deduce existence of a probability $Q^\theta_\infty$ on $\BB_\mathbb{R}$ such that $\lim_{n\to\infty} Q_n^\theta=Q^\theta_\infty$ weakly (strictly speaking \cite[Theorem~3]{curtiss} gives the pointwise convergence, as $n\to\infty$, of the  distribution functions of the $Q^\theta_n$ towards the distribution function of $Q^\theta_\infty$ at every continuity point of the latter, but this is equivalent to weak convergence \cite[Theorem~25.8, Eq.~(25.1)]{billingsley}).  In particular \cite[p.~34, Lemma~7.8]{MR3185174}, $Q^\theta_\infty$ is infinitely divisible and, by  \cite[p.~41, Theorem~8.7(1)]{MR3185174} applied with arbitrary $f$ compactly supported in $(-\infty,-x)$, its representing L\'evy measure is carried by $[-x,\infty)$. Besides, the same result \cite[Theorem~3]{curtiss}  of Curtiss also ensures that $$\int e^{-uy}Q^\theta_\infty(\dd u)=\lim_{n\to\infty}\int e^{-uy}Q^\theta_n(\dd u)=e^{\kappa(y+\theta)-\kappa(\theta)},\quad y\in (-\theta,\infty).$$
By weak convergence, for each $m\in \mathbb{N}$, $\int( e^{u\theta}\land m)Q^\theta_\infty(\dd u)=\lim_{n\to\infty}\int( e^{u\theta}\land m)Q^\theta_n(\dd u)\leq \limsup_{n\to\infty}\int e^{u\theta}Q^\theta_n(\dd u)=\limsup_{n\to\infty}e^{-(\kappa_n(\theta)-\kappa_n(0))}\leq \lim_{n\to\infty}e^{-\kappa_n(\theta)}=e^{-\kappa(\theta)}<\infty$; by monotone convergence, on letting $m\to\infty$, we find that $\int e^{u\theta}Q^\theta_\infty(\dd u)<\infty$. In the preceding display we may then pass to the limit $y\downarrow -\theta$ by monotone -- on $[0,\infty)$ -- and bounded -- on $(-\infty,0)$ -- convergence, to find that $\int e^{u\theta}Q^\theta_\infty(\dd u)=e^{\kappa(0+)-\kappa(\theta)}$ with necessarily $\kappa(0+)\in \mathbb{R}$. 
Performing another exponential tilting we get the infinitely divisible probability $Q_\infty:=(\mathbb{R}\ni u\mapsto e^{u \theta-(\kappa(0+)-\kappa(\theta))})\cdot Q^\theta_\infty$, whose representing L\'evy measure continues to be carried by $[-x,\infty)$, satisfying $$\int e^{-uy}Q_\infty(\dd u)=e^{\kappa(y)-\kappa(0+)},\quad y\in (0,\infty).$$ This can only be if there is a $\tilde\kappa\in \pkLp_x$ vanishing at zero for which $\tilde \kappa=\kappa-\kappa(0+)$ on $(0,\infty)$.  The conclusions of the proposition are now immediate.
\end{proof}

 \section{Laplace duality: semigroups}\label{sec:Laplacedualitysemigroups}

\subsection{Completely monotone processes}\label{subsec:cmprocesses}
 The property common to all positive Markov processes in Laplace duality, irrespective of the conventions used for $0\cdot \infty$, $\infty\cdot 0$ is -- as we shall show -- that of 
\begin{definition}\label{definition:cm-main}
A positive Markov process $X$  is said to be completely monotone when $$P_t\cm \subset \cm\text{ for all }t\in [0,\infty),$$ where $\cm$ is the class of $\downarrow$ maps $f:[0,\infty]\rightarrow [0,\infty)$ for which the restriction $f\lvert_{(0,\infty)}$ is completely monotone. 
\end{definition}
Recall here \cite[Definition~1.3]{bernstein} that a function $f:(0,\infty)\to \mathbb{R}$ is completely monotone when it is $C^\infty$ with $(-1)^nf^{(n)}\geq 0$ for all $n\in \mathbb{N}_0$, and remark that the class $\cm$ is closed under pointwise limits (in $[0,\infty)$), this property holding separately true of completely monotone maps \cite[Corollary~1.6]{bernstein} and $\downarrow$ functions.  
\begin{remark}\label{remark:connecton-cm-to-existing}
A Markov process is completely monotone  if and only if $\cm$ is a \textit{cone} for $X$ in the sense of M\"ohle \cite[Definition 2.3]{zbMATH06225378}. Such a property of complete monotonicity is mentioned in Jansen and Kurt \cite[p.~74, Table 1, line 5]{duality}. We stress however that the state-space considered in \cite{duality} is $[0,\infty)$, cf. the forthcoming Remark~\ref{rem:thereasonofconvention}.
\end{remark}
Prime representatives of $\cm$ are the maps of 
\begin{example}\label{example:exponential-maps}
For $y\in [0,\infty]$, no matter what the choice between the conventions $0^+\cdot\infty$, $0\cdot \infty^-$ (when $y=\infty$) or between $\infty\cdot 0^+$, $\infty^-\cdot 0$ (when $y=0$), $$\ee^y:=([0,\infty]\ni u\mapsto e^{-uy})\in \cm.$$  For $y\in (0,\infty)$, for $y=0$ under $\infty^-\cdot 0$, and for $y=\infty$ under $0^+\cdot\infty$, $\ee^y$ is also continuous. However, $\ee^0=\mathbbm{1}_{[0,\infty)}$ under $\infty\cdot 0^+$, while $\ee^\infty=\mathbbm{1}_{\{0\}}$ under $0\cdot \infty^-$, neither of which is continuous everywhere on $[0,\infty]$. Analogous considerations hold true for the maps $$\ee_x:=([0,\infty]\ni v\mapsto e^{-xv}),\quad x\in [0,\infty].$$ It is worth pointing out that $\ee_x(y)=\ee^y(x)$ for $(x,y)\in [0,\infty]^2$ as long as the same conventions are applied on the l.h.s. and on the r.h.s..
\end{example}
\begin{definition}
We retain in what follows the notation $\ee_x$, $x\in [0,\infty]$, and $\ee^y$, $y\in [0,\infty]$, for the exponential maps, cautioning the reader that the interpretation of $\ee_0$, $\ee_\infty$ and $\ee^0$, $\ee^\infty$ depends on the conventions for $0\cdot\infty$ and $\infty \cdot 0$. 
 \end{definition}
 In our first significant result we wish to establish that it suffices to check the property of Definition~\ref{definition:cm-main} on the exponentials $\ee^y$, $y\in (0,\infty)$. This will be the content of Proposition~\ref{prop:cmequiv}. In preparation for the proof thereof recall
\begin{theostar}[Bernstein-Widder]\label{thm:bernstein} 
Any  completely monotone function $f:(0,\infty)\to\mathbb{R}$ is the Laplace transform $\hat\mu$ of a unique measure $\mu$ on  $\BB_{[0,\infty)}$, in which case
\begin{enumerate}[(i)]
\item $\mu(\{0\})=f(\infty-)$, whence $\mu$ is carried by $(0,\infty)$ if and only if $f(\infty-)=0$;
\item $\mu([0,\infty))=f(0+)$, whence $\mu$ is finite if and only if $f$ is bounded. 
\end{enumerate}
When, moreover, $f\leq 1$,  then there exists a unique probability $\nu$ on $\BB_{[0,\infty]}$ such that \[f(x)=\widehat{\nu}(x)
,\quad x\in (0,\infty).\]  Conversely, if $\mu$ is a measure on $\BB_{[0,\infty)}$ with $\hat\mu<\infty$, then $\hat\mu$ is completely monotone. 
\end{theostar}
\begin{proof}
This is basically  \cite[Theorem~1.4]{bernstein}. The only minor point perhaps worth mentioning is that in case $f\leq 1$, then  extending $\mu$ by putting the mass $1-f(0+)$ onto $\infty$, we get the probability $\nu$. 
\end{proof}

The analogue of Bernstein-Widder's theorem pertaining to $\cm$ is the content of
\begin{proposition}\label{cm-represent}
Let $f\in \cm$. Then there exists a unique finite measure $\mu$ on  $\BB_{(0,\infty)}$ such that 
\begin{align}\nonumber
f(x)&=(f(0)-f(0+))\mathbbm{1}_{\{0\}}(x)\\
&\qquad +\int e^{-xv}\mu(\dd v)+(f(\infty-)-f(\infty))\mathbbm{1}_{[0,\infty)}(x)+f(\infty),\quad x\in [0,\infty].\label{eq:cm-represent}
\end{align}
\end{proposition}
\begin{proof}
Uniqueness follows from $\widehat{\mu}=f-f(\infty-)$ on $(0,\infty)$ and from the fact that finite Laplace transforms determine measures on $\BB_{(0,\infty)}$. We turn to existence. Applying Theorem~\ref{thm:bernstein} to $\big(f-f(\infty-)\big)\vert_{(0,\infty)}$ gives a finite measure $\mu$ on  $\BB_{(0,\infty)}$ for which $$f(x)-f(\infty-)=\int e^{-xv}\mu(\dd v),\quad x\in (0,\infty),$$ yielding at once \eqref{eq:cm-represent} for $x\in (0,\infty)$.  Since  $\int e^{-\infty\cdot v}\mu(\dd v)=0$, \eqref{eq:cm-represent} holds true for $x=\infty$, whereas  $f(0+)-f(\infty-)=\int e^{-0\cdot v}\mu(\dd v)$ ensures that it is also verified for $x=0$. 
\end{proof}

\begin{example}\label{example:oh-fuck}
If $f\in \cm$, then, under the conventions $0\cdot\infty^-$, $\infty\cdot 0^+$, so that $\ee^0=\mathbbm{1}_{[0,\infty)}$ and $\ee^\infty=\mathbbm{1}_{\{0\}}$, we can rewrite \eqref{eq:cm-represent} as 
\[f=f(\infty)+\int \ee^v\tilde{\mu}(\dd v),\]
for the measure $\tilde{\mu}$ on $\BB_{[0,\infty]}$ defined by
$$\tilde{\mu}:=\big(f(\infty-)-f(\infty)\big)\delta_0+\iota_\star \mu+\big(f(0)-f(0+)\big)\delta_\infty,$$
$\iota:(0,\infty)\to [0,\infty]$ being the identity inclusion.
\end{example}
Theorem~\ref{thm:bernstein}  also yields

\begin{proposition}\label{lemma:kernels}
Let $X$ be a positive Markov process and let $t\in [0,\infty)$. Assume that  \[P_t\{\ee^y:y\in (0,\infty)\} \subset \cm.\]
Then, as functions of $x\in [0,\infty]$, the expressions $\mathbb{P}_x\left(X_t<\infty\right)$, $\mathbb{E}_x\left[e^{-X_ty}\right]$ for $y\in (0,\infty)$, and $\mathbb{P}_x\left(X_t=0\right)$ all belong to $\cm$; in particular, there exist unique probabilities $q_t(0+,\cdot)$, $q_t(y,\cdot)$ for  $y\in (0,\infty)$, and $q_t(\infty-,\cdot)$ on $\BB_{[0,\infty]}$ such that for all $x\in (0,\infty)$,
\begin{equation}\label{eq-lemma:laplacedualpreform2}
\mathbb{P}_x\left(X_t<\infty\right)=\int e^{-xv}q_t(0+,\dd v),
\end{equation}
\begin{equation}\label{eq-lemma:laplacedualpreform1}
\mathbb{E}_x\left[e^{-X_ty}\right]=\int e^{-xv}q_t(y,\dd v)\text{ for } y\in (0,\infty),
\end{equation}
and
\begin{equation}\label{eq-lemma:laplacedualpreform3}
\mathbb{P}_x\left(X_t=0\right)=\int e^{-xv}q_t(\infty-,\dd v).
\end{equation}
The map $(0,\infty)\ni y \mapsto q_t(y,A)$ is measurable for all  $A\in \BB_{[0,\infty]}$.
\end{proposition}
\begin{proof}
Notice that we can write $\PP_x(X_t<\infty)=\lim_{y\downarrow 0}\mathbb{E}_x\left[e^{-X_ty}\right]$ and $\PP_x\left(X_t=0\right)=\lim_{y\to \infty}\mathbb{E}_x\left[e^{-X_ty}\right]$ for all $x\in  [0,\infty]$. As $\cm$ is closed for pointwise limits and because by assumption $P_t\ee^y\in \cm$ for all $y\in (0,\infty)$, we deduce that all the maps in question  belong to $\cm$.  Since for each $y\in (0,\infty)$, the map $(0,\infty)\ni x\mapsto \mathbb{E}_x[e^{-X_ty}]$ is completely monotone and bounded by $1$, the existence of the probability measure $q_t(y,\cdot)$ on $[0,\infty]$ follows from Theorem~\ref{thm:bernstein}. Similarly   we get  $q_t(0+,\cdot)$ and $q_t(\infty-,\cdot)$.  Uniqueness of the measures follows from the fact that Laplace transforms determine probabilities on $\BB_{[0,\infty]}$. Finally, for all $x\in (0,\infty)$, the expression $$
\mathbb{E}_x\left[e^{-X_ty}\right]=\int e^{-xv}q_t(y,\dd v)$$
is even completely monotone, but certainly measurable, in  $y\in (0,\infty)$. Since the maps $\ee_x$, $x\in (0,\infty)$, form a multiplicative class of bounded functions that generate $\BB_{[0,\infty]}$,  functional monotone class \cite[Theorem~2.45]{pollard} applies and ensures that  $(0,\infty)\ni y\mapsto q_t(y,A)$ is measurable for all $A\in \BB_{[0,\infty]}$.
\end{proof}

Combining Propositions~\ref{cm-represent} and~\ref{lemma:kernels} we arrive at the announced
\begin{proposition}\label{prop:cmequiv}
A positive Markov process $X$ is completely monotone if and only if  $$P_t\{\ee^y:y\in (0,\infty)\} \subset \cm,\quad t\in [0,\infty).$$
\end{proposition}
\begin{proof} Necessity is clear. Take now any $f\in \cm$. It admits the representation of Proposition~\ref{cm-represent}, in particular the finite measure $\mu$ on $\BB_{(0,\infty)}$ is specified. By, and in the notation of  Proposition~\ref{lemma:kernels}, for all $t\in [0,\infty)$, we may introduce the finite measure $\mu_t$ on $\BB_{[0,\infty]}$, $$\mu_t(A):=\int q_t(y,A)\mu(\dd y),\quad A\in \BB_{[0,\infty]},$$ and then further by \eqref{eq:cm-represent} and Tonelli we have,  for all $x\in (0,\infty)$,
\begin{align*}
\mathbb{E}_x[f(X_t)]&=(f(0)-f(0+))\mathbb{P}_x(X_t=0)\\
&\qquad +\int\mathbb{E}_x\left[e^{-X_ty}\right]\mu(\dd y)+(f(\infty-)-f(\infty))\mathbb{P}_x(X_t<\infty)+f(\infty)\\
&=(f(0)-f(0+))\int e^{-xv}q_t(\infty-,\dd v)+\int\int e^{-xv}q_t(y,\dd v)\mu(\dd y)\\
& \qquad+(f(\infty-)-f(\infty))\int e^{-xv}q_t(0+,\dd v)+f(\infty)\\
&=(f(0)-f(0+))\int e^{-xv}q_t(\infty-,\dd v)+\int e^{-xv}\mu_t(\dd v)\\
& \qquad+(f(\infty-)-f(\infty))\int e^{-xv}q_t(0+,\dd v)+f(\infty)\\
&=\text{\footnotesize $\int e^{-xv}\left[\underbrace{(f(0)-f(0+))q_t(\infty-,\cdot)+\mu_t+(f(\infty-)-f(\infty))q_t(0+,\cdot)+f(\infty)\delta_0}_{=:\rho_t}\right](\dd v)$\normalsize}.
\end{align*}
The measure $\rho_t$ being positive (since $f$ is positive $\downarrow$), we deduce at once that $(P_tf)\vert_{(0,\infty)}$ is completely monotone. Finally, as the first equality of the preceding display is  valid for all $x\in [0,\infty]$, not just $x\in (0,\infty)$, and since each term separately on the r.h.s. of this equality is $\downarrow$ in $x\in [0,\infty]$, as was noted in Proposition~\ref{lemma:kernels}, we get also that $P_tf$ is $\downarrow$.
\end{proof}

\begin{corollary}\label{corollary:cm-inherited}
A weak limit (in the sense of one-dimensional distributions) of completely monotone positive Markov processes is completely monotone. 
\end{corollary}
\begin{proof}
Combine Proposition~\ref{prop:cmequiv} with the fact that $\cm$ is closed under pointwise limits and the continuity of the maps $\ee^y$, $y\in (0,\infty)$.
\end{proof}

\subsection{Processes in Laplace duality}\label{subsec:processesinLaplaceduality}
We define now the notions of Laplace duality relationships associated to the conventions for $0\cdot \infty$ and $\infty\cdot 0$ that were specified in Subsection~\ref{miscellaneous}. 
\begin{definition}\label{def:LD-main}
Let $\star_1\in \{0^+\cdot\infty,0\cdot \infty^-\}$ be a convention for $0\cdot \infty$ and let $\star_2\in \{\infty\cdot 0^+,\infty^-\cdot 0\}$ be a convention for $\infty\cdot 0$. A positive Markov process $X$ is said to be in $(\star_1,\star_2)$-Laplace duality with a positive Markov process $Y$ provided
\begin{equation}\label{LD-main}
\EE_x[e^{-X_ty}]=\EE^y[e^{-xY_t}],\quad t\in [0,\infty),\, (x,y)\in [0,\infty]^2,
\end{equation}
the products in the exponentials being interpreted under the conventions $\star_1$ and $\star_2$.
\end{definition}
The concept of Definition~\ref{def:LD-main} is  the  $H$-duality of $X$ and $Y$ of Definition~\ref{def:H-duality} for $H=\mathsf{e}_{\star_1,\star_2}:[0,\infty]\times [0,\infty]\to [0,1]$, $$\mathsf{e}_{\star_1,\star_2}(x,y):=e^{-xy},\quad (x,y)\in [0,\infty]^2,$$ under the conventions $\star_1$ and $\star_2$. Conditions \ref{dual:A}-\ref{dual:B} of Subsection~\ref{subsection:processes-in-duality} are met for this $H$ because $ \{\mathsf{e}_{\star_1,\star_2}(x,\cdot):x\in (0,\infty)\}$ is a multiplicative class of bounded real functions generating $\BB_{[0,\infty]}$ so that functional monotone class applies; \ref{dual:C} follows because $[0,\infty]$ is endowed with the Borel $\sigma$-field of a Polish space, see e.g. \cite[Theorem~4.5 \& Exercise~2.13]{pollard}. In particular, if $X$ admits Laplace duals $Y_1$ and $Y_2$ of the same kind, then the law of $Y_1$ is the same as the law of $Y_2$.  The same conditions  \ref{dual:A}-\ref{dual:C} are met for $ H=\widehat{\mathsf{e}_{\star_1,\star_2}}$, which results from $\mathsf{e}_{\star_1,\star_2}$ on transposing its first and second coordinate, the fact that $\mathsf{e}_{\star_1,\star_2}$ need not be symmetric notwithstanding.\label{what-was-example-2.2}

We note also that $X$ is in $(0^+\cdot\infty,\infty\cdot 0^+)$-Laplace duality with $Y$ if and only if $Y$ is in $(0^+\cdot\infty,\infty\cdot 0^+)$-Laplace duality with $X$. The same for $(0\cdot \infty^-,\infty^-\cdot 0)$ in lieu of $(0^+\cdot\infty,\infty\cdot 0^+)$. On the other hand, $X$ is in $(0^+\cdot\infty,\infty^-\cdot 0)$-Laplace duality with $Y$ if and only if $Y$ is in $(0\cdot \infty^-,\infty\cdot 0^+)$-Laplace duality with $X$. We shall refer to the facts noted in this paragraph indiscriminately as transposition.\label{page--transposition}

A simple consequence of Proposition~\ref{prop:cmequiv} is
\begin{corollary}\label{corollary:LDimpliesCM} If a positive Markov process $X$ admits a Laplace dual (of any kind), then $X$ is completely monotone.
\end{corollary}
\begin{proof}
Let $Y$ be the Laplace dual. Apply Proposition~\ref{prop:cmequiv} together with the observation that for all $y\in (0,\infty)$, the expression $ P_t \ee^y(x)=\EE_x[e^{-X_ty}]= \EE^y[e^{-xY_t}]=\int e^{-xv}q_t(y,\dd v)$ is certainly completely monotone in $x\in (0,\infty)$, and also it is $\downarrow$ in $x\in [0,\infty]$ because $e^{-xv}$ is $\downarrow$ in $x\in  [0,\infty]$ for all $v\in [0,\infty]$. This is so no matter what the choice of the conventions for $0\cdot\infty$ and $\infty \cdot 0$.
\end{proof}
Complete monotonicity will be combined with the relevant properties at the boundaries to ensure existence of Laplace duals.

\begin{definition}\label{definition:varying-ct-bodr}
Let $X$ be a positive Markov process. As previously agreed upon, $p=(p_t(x,\cdot))_{(t,x)\in [0,\infty)\times [0,\infty]}$ is its family of transition kernels and  $P=(P_t)_{t\in [0,\infty)}$ is its semigroup. 
\begin{enumerate}[(A)]
\item A boundary point $x\in \{0,\infty\}$  is absorbing (resp. non-sticky) for $X$ when, for all $t\in [0,\infty)$, $p_t(x,\cdot)=\delta_x$ (resp. $p_t(u,[0,\infty]\backslash \{x\})=1$ for all $u\in [0,\infty]\backslash \{x\}$).
\item\label{weak-} $X$ is weakly continuous at $x\in [0,\infty]$ if $P_tf$ is continuous at $x$ for all continuous $f:[0,\infty]\to \mathbb{R}$  for all $t\in [0,\infty)$.
\item\label{weak0} By weak-$[0,\infty)$ continuity at $0$ of $X$ we mean that  for all $t\in [0,\infty)$, $\EE_x[f(X_t);X_t<\infty]$ converges  to $\EE_0[f(X_t);X_t<\infty]$ as $x\downarrow 0$ for all  continuous $f:[0,\infty]\to \mathbb{R}$  (equivalently, all  continuous $f:[0,\infty)\to \mathbb{R}$ having a finite limit at $\infty$). 
\item\label{weakinfty} By weak-$(0,\infty]$ continuity at $\infty$ of  $X$ we mean that for all $t\in [0,\infty)$, $\EE_x[f(X_t);X_t>0]$ converges  to $\EE_\infty[f(X_t);X_t>0]$  as $x\uparrow \infty$ for all  continuous $f:[0,\infty]\to \mathbb{R}$ (equivalently, all  continuous $f:(0,\infty]\to \mathbb{R}$ having a finite limit at $0$). 
\end{enumerate}
We reserve the right to apply the preceding qualifications to any (not necessarily Markovian) $[0,\infty]$-valued process $X$ defined under a family of probabilities $(\PP_x)_{x\in E}$, $E\subset [0,\infty]$, on a common measurable space, as long as the resulting notions are significant. 
\end{definition}

The linear span of $\{\ee^y:y\in (0,\infty)\}\cup \{1\}$ is an algebra containing the constants and separating the points of $[0,\infty]$.\label{page-stone-weierstrass} By Stone-Weierstrass \cite[Theorem~2.4.11]{dudley} for $[0,\infty]$, if a positive Markov process $X$ is completely monotone, then $P_tf$ is continuous on $(0,\infty)$ for all continuous $f:[0,\infty]\to \mathbb{R}$, for all $t\in [0,\infty)$. Put differently, $X$ is weakly continuous at every point of $(0,\infty)$. It is only at the boundaries that the  weak continuity of a completely monotone process can fail. It is also worth emphasizing that the weak continuity notion of \ref{weak0}, resp. \ref{weakinfty}, is  a stronger condition on the process $X$ than would have been mere weak continuity at $0$, resp. at $\infty$, of $X$ in the sense of \ref{weak-}. We further caution the reader regarding the term \textit{absorbing} which  does not, by itself, imply that the process ever actually reaches the boundary in question with strictly positive probability. 
\begin{theorem}\label{thm:completemonotonicity}
The various types of Laplace dualities corresponding to the four possible combinations of conventions are characterized as follows.
\begin{enumerate}[(i)]
    \item \label{thmcompletemonotonicity1} 
A given positive Markov process $X$ is in $(0^+\cdot\infty,\infty\cdot 0^+)$-Laplace duality with some positive Markov process $Y$ if and only if $X$ is completely monotone, weakly-$[0,\infty)$ continuous at $0$ and has $\infty$  absorbing. When so, then $Y$ is also completely monotone, weakly-$[0,\infty)$ continuous at $0$ and has $\infty$  absorbing; besides, under $ 0^+\cdot\infty$, $\infty\cdot 0^+$,
\begin{equation}\label{eq:laplacedualinftyabsorbing0}
\EE_0[e^{-X_ty}]=\PP^y(Y_t<\infty)\text{ and }\PP_x(X_t<\infty)=\EE^0[e^{-xY_t}],\quad (x,y)\in [0,\infty]^2,\, t\in [0,\infty).
\end{equation}
\item \label{thmcompletemonotonicity2} A given positive Markov process $X$ is in $(0\cdot\infty^-,\infty^-\cdot 0)$-Laplace duality with some positive Markov process $Y$ if and only if $X$ is completely monotone, weakly-$(0,\infty]$ continuous at $\infty$ and has $0$  absorbing. When so, then $Y$ is also completely monotone, weakly-$(0,\infty]$ continuous at $\infty$ and has $0$  absorbing; besides, under $0\cdot \infty^-$, $\infty^-\cdot 0$, 
\begin{equation}\label{eq:laplacedualinftyabsorbing1}
\EE_\infty[e^{-X_ty}]=\PP^y(Y_t=0) \text{ and }\PP_x(X_t=0)=\EE^\infty[e^{-xY_t}],\quad (x,y)\in [0,\infty]^2,\, t\in [0,\infty).
\end{equation}
\item \label{thmcompletemonotonicity3} A given positive Markov process $X$ is in $(0\cdot\infty^-,\infty\cdot 0^+)$-Laplace duality with some positive Markov process $Y$ if and only if $X$ is completely monotone and has $0$ and $\infty$ both absorbing. When so, then  $Y$ is completely monotone, weakly continuous at $0$ and $\infty$, and, under $0\cdot \infty^-$, $\infty\cdot 0^+$, 
\begin{equation}\label{eq:laplacedualinftyabsorbing2}
\PP_x(X_t<\infty)=\EE^0[e^{-xY_t}]\text{ and }\PP_x(X_t=0)=\EE^\infty[e^{-xY_t}],\,  x\in [0,\infty],\, t\in [0,\infty).
\end{equation}
\item \label{thmcompletemonotonicity4} A given positive Markov process $X$ is in $(0^+\cdot\infty,\infty^-\cdot 0)$-Laplace duality with some positive Markov process $Y$ if and only if $X$ is completely monotone and weakly continuous at both $0$ and $\infty$. When so, then  $Y$ is completely monotone, has $0$ and $\infty$ both absorbing, and, under $0^+\cdot\infty$, $\infty^-\cdot 0$, 
\begin{equation}\label{eq:laplacedualinftyabsorbing3}
\EE_0[e^{-X_ty}]=\PP^y(Y_t<\infty)\text{ and }\EE_\infty[e^{-X_ty}]=\PP^y(Y_t=0),\,\, y\in [0,\infty],\, t\in [0,\infty).
\end{equation}
\end{enumerate}
\end{theorem} 
 We  prove  Theorem~\ref{thm:completemonotonicity} on p.~\pageref{proof:big-one} after some discussion and illustrating examples.  Tables~\ref{table.1} and~\ref{table.2}  summarize the roles of the conventions. 
Several duality relationships may hold for a given completely monotone positive Markov process, which will be seen in the examples to follow presently. 
\begin{table}[htb!]
\begin{center}
    \begin{tabular}{ccc}
    $(\star_1,\star_2)$ & continuity of $X$ & absorption points of $X$\\\hline
         $(\dashuline{0^{+}\cdot \infty},\underline{\infty\cdot 0^+})$ & weakly-$[0,\infty)$ \dashuline{at $0$} & \underline{$\infty$}  \\
         $(\dotuline{0\cdot\infty^-},\underline{\underline{\infty^{-}\cdot 0}})$ & weakly-$(0,\infty]$ \underline{\underline{at  $\infty$}} & \dotuline{ $0$ } \\
          $(\dotuline{0\cdot\infty^-},\underline{\infty\cdot 0^+})$ &  & \dotuline{ $0 $ }, \underline{$\infty$}\\
          $(\dashuline{0^+\cdot\infty},\underline{\underline{\infty^-\cdot 0}})$ & weakly \dashuline{at $0$} and \underline{\underline{at $\infty$}} & 
    \end{tabular}
    \end{center}
     \caption{Weak continuity/absorption properties of $X$ in $(\star_1,\star_2)$-Laplace duality with $Y$ according to the various conventions. Different  underlinings indicate what convention co-occurs with which property. We see that the properties of $0$ (resp. of $\infty$) for $X$ can be read off from the first (resp. second) entry of the convention, the presence (resp. absence) of $\pm$ at $0$ (resp. $\infty$) indicating weak continuity (resp. absorption) at said point.}\label{table.1}
 \end{table}
 \begin{table}[htb!]
\begin{center} 
       \begin{tabular}{ccc}
    $(\star_1,\star_2)$ & continuity of $Y$ & absorption points of $Y$\\\hline
  $(\underline{0^{+}\cdot \infty},\dashuline{\infty\cdot 0^+})$ & weakly-$[0,\infty)$ \dashuline{at $0$} & \underline{$\infty$}  \\
         $(\underline{\underline{0\cdot\infty^-}},\dotuline{\infty^{-}\cdot 0})$ & weakly-$(0,\infty]$ \underline{\underline{at  $\infty$}} & \dotuline{ $0$ } \\
     $(\underline{\underline{0\cdot\infty^-}},\dashuline{\infty\cdot 0^+})$    & weakly \dashuline{at $0$} and \underline{\underline{at $\infty$}}   & \\
          $(\underline{0^+\cdot\infty},\dotuline{\infty^-\cdot 0})$ &  & \dotuline{ $0 $ }, \underline{$\infty$}\\
    \end{tabular}
\end{center}
     \caption{Same as Table~\ref{table.1} but for the process $Y$ (still $X$ is in $(\star_1,\star_2)$-Laplace duality with $Y$). Now the first (resp. second) entry of the convention has to do with properties of $\infty$ (resp. of $0$) for $Y$; and  then again the presence (resp. absence) of $\pm$ at $\infty$ (resp. $0$) indicates weak continuity (resp. absorption) at said point. }\label{table.2}
 \end{table}

An immediate consequence of  Theorem~\ref{thm:completemonotonicity}, in particular of  the identities \eqref{eq:laplacedualinftyabsorbing0}-\eqref{eq:laplacedualinftyabsorbing3}, is a  correspondence between absorptivity and non-stickiness of processes in Laplace duality.

\begin{corollary}\label{cor:absorbing-nonsticky}
Let $\star_2\in \{\infty\cdot 0^+,\infty^-\cdot 0\}$.      If $X$ is in $(0^+\cdot \infty, \star_2)$-Laplace duality with $Y$, then 
         $0$ is absorbing for $X$ if and only if $\infty$ is non-sticky for $Y$.
If $X$ is in $( 0\cdot\infty^-,\star_2)$-Laplace duality with $Y$,  then 
         $\infty$ is absorbing for $Y$ if and only if $0$ is non-sticky for $X$.\qed
 \end{corollary}
The parallel statements  in Corollary \ref{cor:absorbing-nonsticky} with the roles of $X$ and $Y$ interchanged also hold true; this follows by transposition. 
\begin{remark}\label{rem:thereasonofconvention}
Let us  explain the necessity of introducing conventions for $0\cdot \infty, \infty\cdot 0$. As we have pointed out in the Introduction, at first sight it may seem best to only deal with pairs of state-spaces for $X$ and $Y$ that avoid the issue of specifying $0\cdot\infty$ and $\infty\cdot 0$. The problem with this is that in considering the existence of a Laplace dual $Y$ for a given $X$, whilst by Bernstein-Widder's theorem~\ref{thm:bernstein} the relation \eqref{intro:ld} surely implies the complete monotonicity of $\EE_x[e^{-X_t y}]$ in $x\in (0,\infty)$, which in turn ensures the existence of a probability $\nu_t(y,\cdot)$ for which 
\begin{equation*}
\EE_x[e^{-X_t y}]=\int e^{-xv}\nu_t(y,\dd v),\quad x\in (0,\infty),
\end{equation*}
in general $\nu_t(y,\cdot)$ is a probability only on $[0,\infty]$, a superset of $(0,\infty)$ with both end-points $0$, $\infty$ included. For this reason we consider the natural domain of Laplace duality to be $[0,\infty]$ with due consideration having to be given to the conventions on $0\cdot\infty$ and $\infty\cdot 0$. Fortunately, this approach is also without loss of generality because Laplace duality on restricted state-spaces can always be prolongated to $[0,\infty]$, as we shall see in Remark~\ref{extended:wlog}. 
\end{remark}

\begin{example}[Subordinators]\label{example:subo}
Let $X$ be a possibly killed subordinator (coffin: $\infty$) under the probabilities $(\PP_x)_{x\in [0,\infty]}$ with Laplace exponent $\Phi\in\pksubo$: $$\EE_x[e^{-X_ty};X_t<\infty]=e^{-xy-t\Phi(y)},\quad \{x,y\}\subset [0,\infty),\, t\in [0,\infty);
$$ under $\PP_\infty$, $X$ is  the constant $\infty$ a.s. We consider below two variants of a Laplace dual $Y$ for $X$. In both cases, for $y\in (0,\infty)$, under $\PP^y$, $Y$ a.s. starts at $y$ and is killed, sent to the coffin $\infty$, at rate $\Phi(y)$, while under $\PP^\infty$ it is the constant $\infty$ a.s.
\begin{itemize}
    \item In the first variant, under $\PP^0$, $Y$ a.s. starts at $0$ and is killed, sent to the coffin $\infty$, at rate $\Phi(0)$. Then
    $X$ is in $(0^+\cdot\infty, \infty\cdot 0^+)$-Laplace duality with  $Y$.

\item In the second variant, under  $\PP^0$, $Y$ is the constant $0$ a.s. Then $X$ is in $(0^+\cdot\infty, \infty^-\cdot 0)$-Laplace duality with $Y$.
\end{itemize}
So, when $\Phi(0)>0$, the subordinator $X$ is in Laplace duality with different $Y$, depending on the choice of the conventions for the products $0\cdot \infty$ and $\infty\cdot 0$. However, except in the trivial case when $\Phi\equiv 0$ (the constant process, which is in Laplace duality of whichever kind with itself), the other two types of Laplace dualities are precluded  ($\because$ $X$ does not have $0$ absorbing if $\Phi\not\equiv 0$). 
\end{example}
In the context of the preceding example the reader may well wonder what changes if, instead of a subordinator, one considers a possibly killed spectrally positive L\'evy process $X$ (that is not a subordinator). Its  Laplace transform then takes the form
\[\EE_x[e^{-X_ty};X_t<\infty]=e^{-xy+t\Psi(y)}, \quad (x,y)\in \mathbb{R}\times[0,\infty),\ t\in [0,\infty),\]
where $\Psi\in \pkspLp\backslash (-\pksubo)$. In this case, for no $t\in (0,\infty)$ can there exist a family of probabilities $q=(q_t(y,\cdot))_{y\in (0,\infty)}$ on $[0,\infty]$ satisfying
\[ \EE_x[e^{-X_ty}]= \int e^{-xv} q_t(y, \dd v),\quad \{x,y\}\subset (0,\infty),
\]
for if there were, letting $x\downarrow 0$,  we would get 
$q_t([0, \infty))\geq  e^{t\Psi(y)} > 1$ for some (all large enough) $y \in (0, \infty)$, which is a contradiction.
\begin{example}[CB processes]\label{example:cb}
Let $X$ be a CB process under the probabilities $(\PP_x)_{x\in [0,\infty]}$  with the branching mechanism $\Psi\in \pkspLp$. According to \cite[Chapter 12]{Kyprianoubook} it means:  $$\EE_x[e^{-X_ty}]=e^{-xu_t(y)},\quad x\in [0,\infty),\, y\in (0,\infty),\, t\in [0,\infty),
$$
where for each $y\in (0,\infty)$, $[0,\infty)\ni t\mapsto u_t(y)\in (0,\infty)$ is the unique $C^1$-solution to 
\begin{equation}\label{ut}
\frac{\dd}{\dd t} u_t(y)=-\Psi(u_t(y)),\ u_0(y)=y;\end{equation} 
under $\PP_\infty$, $X$ is the constant $\infty$ a.s. We consider below three variants of a Laplace dual $Y$ for $X$. In all cases, for $y\in (0,\infty)$, under $\PP^y$, $Y$ a.s. coincides with $(u_t(y))_{t\in [0,\infty)}$.
\begin{itemize}
    \item In the first variant, under $\PP^\infty$ (resp. $\PP^0$)  $Y$ is constant and equal to $\infty$ (resp. $0$) a.s. Then 
$X$ is in $(0^+\cdot\infty,\infty^-\cdot 0)$-Laplace duality with $Y$.

\item In the second variant, for all $t\in [0,\infty)$, a.s.,
\begin{itemize}
    \item under $\PP^\infty$, $Y_t=u_t(\infty-):=\text{$\uparrow$-}\lim_{y\uparrow\infty}u_t(y) \in (0,\infty]$,
    \item under $\PP^0$, $Y_t=u_t(0+):=\text{$\downarrow$-}\lim_{y\downarrow0} u_t(y) \in [0,\infty)$.
\end{itemize}
Then $X$ is in $(0\cdot\infty^-, \infty\cdot 0^+)$-Laplace duality with $Y$.
\item For the third variant  take $0$ absorbing for $Y$ as in the first instance and take $\PP^\infty$  as in the second  variant. Then $X$ is in $(0\cdot\infty^-, \infty^-\cdot 0)$-Laplace duality with $Y$.
\end{itemize}
Thus, the CB process $X$ may be in Laplace duality with different $Y$ depending on the choice of the conventions for the products $0\cdot \infty$ and $\infty\cdot 0$. We return to CB processes in Paragraph~\ref{para:cb-processes} by studying the generators, there with $(0^+\cdot\infty,\infty\cdot 0^+)$-Laplace duality, so that in fact a CB process admits all four types of Laplace duals.
\end{example}
\begin{example}[CBI processes]\label{example:cbi}
Let $X$ be a continuous-state branching process with immigration (CBI process) under the probabilities $(\PP_x)_{x\in [0,\infty]}$  with branching mechanism $\Psi\in\pkspLp$ and immigration mechanism $\Phi\in\pksubo$. One has, see Kawazu and Watanabe \cite[Eq.~(1.1) and Theorem~1.1]{KAW}: $$\EE_x[e^{-X_ty}]=e^{-xu_t(y)-\int_0^t\Phi(u_s(y))\dd s},\quad x \in [0,\infty),\, y\in (0,\infty),\, t\in [0,\infty),
$$
where for each $y\in (0,\infty)$, $[0,\infty)\ni t\mapsto u_t(y)\in (0,\infty)$ satisfies \eqref{ut}, while under $\PP_\infty$, $X$ is the constant $\infty$ a.s. We interpret $u(0):=([0,\infty)\ni t\mapsto u_t(0+))$ in the limiting sense. Under $\PP^\infty$, let $Y$  be the constant $\infty$ a.s. , while for $y\in [0,\infty)$, under $\PP^y$, let  $Y$ a.s. satisfy $Y_t=u_t(y)$ for $t\in [0,\zeta)$, where $\zeta$ is a random time at which $Y$ is killed, sent to $\infty$, and whose law is given by $$\mathbb{P}^y(\zeta>t)=e^{-\int_0^{t}\Phi(u_s(y))\dd s}, \quad t\in[ 0,\infty).$$Then  $X$ is in $(0^+\cdot\infty, \infty\cdot 0^+)$-Laplace duality with~$Y$. 
\end{example}
\begin{proof}[Proof of Theorem~\ref{thm:completemonotonicity}]\label{proof:big-one}
The proofs of the parts are all quite similar, but for clarity we write them all out separately at the cost of some repetition.

 \textbf{Part }\ref{thmcompletemonotonicity1}. We work throughout this part under the conventions $ 0^+\cdot\infty$, $\infty\cdot 0^+$. 

The linear span of $\{\ee^y\vert_{[0,\infty)}:y\in (0,\infty)\}$ is an algebra of maps vanishing at infinity, separating the points of, and vanishing nowhere on $[0,\infty)$. Thanks to the convention  $\infty \cdot 0^+$ and by Stone-Weierstrass' theorem \cite[Corollary~V.8.3]{conway} for $[0,\infty)$, 
\begin{equation}
\label{continuityat0LT}
\begin{split}
\text{$X$ is weakly-$[0,\infty)$ continuous at $0$}
\;\Leftrightarrow\;
\Big(\lim_{x\downarrow 0}\mathbb{E}_x\!\left[e^{-X_t y}\right]
= \mathbb{E}_0\!\left[e^{-X_t y}\right]\\
\hfill
\text{for all } t\in[0,\infty),\ y\in[0,\infty)\Big).
\end{split}
\end{equation}
Suppose first that $X$ is in $(0^+\cdot\infty,\infty\cdot 0^+)$-Laplace duality with $Y$. \eqref{continuityat0LT}, \eqref{LD-main}, $0^+\cdot\infty$ show that $X$ is weakly-$[0,\infty)$ continuous at $0$. \eqref{LD-main}, $\infty\cdot 0^+$ ensure that $X$ has $\infty$ absorbing.  Corollary~\ref{corollary:LDimpliesCM} guarantees the complete monotonicity of $X$. By transposition we deduce that $Y$ has these same three properties.

Conversely, assume $X$ is completely monotone, weakly-$[0,\infty)$ continuous at $0$ and has $\infty$  absorbing.  From, and in the notation of Proposition~\ref{lemma:kernels}\eqref{eq-lemma:laplacedualpreform2}-\eqref{eq-lemma:laplacedualpreform1}, and by $\infty \cdot 0^+$: for  $t\in [0,\infty)$,  writing $q_t(0,\cdot):=q_t(0+,\cdot)$, then for $y\in [0,\infty)$, $x\in (0,\infty)$,  
\begin{equation}\label{fund1}
\int e^{-uy}p_t(x,\dd u)=\int  e^{-xv}q_t(y,\dd v);
\end{equation} but, on passing to the limit $x\downarrow 0$, also for $x=0$ by the assumption of weak-$[0,\infty)$ continuity at $0$ of $X$, \eqref{continuityat0LT} and $0^+\cdot\infty $. Moreover,  \eqref{fund1} holds also, first, for $x=\infty$ due to the absorptivity of $\infty$ for $X$ and the convention $\infty\cdot 0^+$; second, then for $y=\infty$ on specifying $q_t(\infty,\cdot):=\delta_\infty$ and noting this time $ 0^+\cdot\infty$. Thus \eqref{fund1} persists for arbitrary $(x,y)\in [0,\infty]^2$. In order to deduce that $X$ is in $(0^+\cdot\infty,\infty\cdot 0^+)$-Laplace duality with a positive Markov process it remains to apply Proposition~\ref{proposition:general-dual}.

\eqref{eq:laplacedualinftyabsorbing0} is a mere specification of \eqref{LD-main} to $x=0$ using $0^+\cdot\infty$, and separately to $y=0$ using $\infty\cdot 0^+$ (or we apply transposition).

\textbf{Part} \ref{thmcompletemonotonicity2}. We work throughout this part under the conventions $ 0\cdot \infty^-$, $\infty^-\cdot 0$. 

The linear span of $\{(1-\ee^y)\vert_{(0,\infty]}:y\in (0,\infty)\}$ is an algebra of maps vanishing at infinity (which is to say: having limit zero at $0+$), separating the points of,  and vanishing nowhere on $(0,\infty]$.  Thanks to the convention  $0\cdot\infty ^-$ and  Stone-Weierstrass for $(0,\infty]$, 
\begin{equation}
\label{continuityatinftyLT}
\begin{aligned}
\text{$X$ is weakly-$(0,\infty]$ continuous at $\infty$}
&\Leftrightarrow
\Big(\lim_{x\uparrow \infty}\mathbb{E}_x\!\left[e^{-X_t y}\right]
= \mathbb{E}_\infty\!\left[e^{-X_t y}\right]\\
&\qquad\qquad 
\text{for all } t\in[0,\infty),\ y\in(0,\infty]\Big).
\end{aligned}
\end{equation}

Suppose first that $X$ is in $(0\cdot \infty^-,\infty^-\cdot 0)$-Laplace duality with $Y$. \eqref{continuityatinftyLT}, \eqref{LD-main}, $\infty^-\cdot 0$ show that $X$ is weakly-$(0,\infty]$ continuous at $\infty$. \eqref{LD-main}, $0\cdot \infty^-$ ensure that $X$ has $0$ absorbing. Corollary~\ref{corollary:LDimpliesCM} guarantees the complete monotonicity of $X$. By transposition we deduce that $Y$ has these same three properties.

Conversely, assume $X$ is completely monotone, weakly-$(0,\infty]$ continuous at $\infty$ and has $0$  absorbing.  From,  and in the notation of Proposition~\ref{lemma:kernels}\eqref{eq-lemma:laplacedualpreform1}-\eqref{eq-lemma:laplacedualpreform3}, and by $0\cdot \infty^-$: for  $t\in [0,\infty)$,  writing $q_t(\infty,\cdot):=q_t(\infty-,\cdot)$, then for $y\in (0,\infty]$, $x\in (0,\infty)$,  
\begin{equation}\label{fund2}
\int e^{-uy}p_t(x,\dd u)=\int e^{-xv}q_t(y,\dd v);
\end{equation} but, on passing to the limit $x\uparrow \infty$, also for $x=\infty$ by the assumption of weak-$(0,\infty]$ continuity at $\infty$ of $X$, \eqref{continuityatinftyLT} and $\infty^-\cdot 0$. Moreover,  \eqref{fund2} holds also, first, for $x=0$ due to the absorptivity of $0$ for $X$ and $0\cdot \infty^-$; second, then for $y=0$ on specifying $q_t(0,\cdot):=\delta_0$ and noting this time $ \infty^-\cdot 0$. Thus \eqref{fund2} persists for arbitrary $(x,y)\in [0,\infty]^2$. In order to deduce that $X$ is in $(0\cdot \infty^-,\infty^-\cdot 0)$-Laplace duality with a positive Markov process it remains to apply Proposition~\ref{proposition:general-dual}.

\eqref{eq:laplacedualinftyabsorbing1} is a mere specification of \eqref{LD-main} to $x=\infty$ using $\infty^-\cdot 0$, and separately to $y=\infty$ using $0\cdot \infty^-$ (or we apply transposition).

\textbf{Part} \ref{thmcompletemonotonicity3}. We work throughout this part under the conventions $ 0\cdot \infty^-$, $\infty\cdot 0^+$.

The linear span of $\{\ee_x:x\in (0,\infty)\}\cup \{1\}$ is an algebra containing the constants and separating the points of $[0,\infty]$. Thanks to Stone-Weierstrass for $[0,\infty]$ we therefore have: 
\begin{equation}\label{continuityatinftyLTY}
\begin{aligned}
\text{$Y$ is weakly continuous at $\infty$}
&\Leftrightarrow\Big(
\lim_{y\uparrow \infty}\mathbb{E}^y\!\left[e^{-xY_t}\right]
= \mathbb{E}^\infty\!\left[e^{-xY_t}\right]\\
&\qquad \qquad \ \, \, \text{for all } t\in[0,\infty),\ x\in(0,\infty)\Big);
\end{aligned}
\end{equation}
\begin{equation}\label{continuityat0LTY}
\begin{aligned}
\text{$Y$ is weakly continuous at $0$}
&\Leftrightarrow \Big(
\lim_{y\downarrow 0}\mathbb{E}^y\!\left[e^{-xY_t}\right]
= \mathbb{E}^0\!\left[e^{-xY_t}\right]\\
&\qquad \qquad \text{for all } t\in[0,\infty),\ x\in(0,\infty)\Big).
\end{aligned}
\end{equation}
Suppose first that $X$ is in $(0\cdot \infty^-,\infty\cdot 0^+)$-Laplace duality with $Y$. \eqref{LD-main}, $\infty\cdot 0^+$ (resp. \eqref{LD-main}, $0\cdot \infty^-$) ensure that $X$ has $\infty$ (resp. $0$) absorbing. \eqref{LD-main}, \eqref{continuityatinftyLTY}, $0\cdot \infty^-$ (resp. \eqref{LD-main}, \eqref{continuityat0LTY}, $\infty\cdot 0^+$) yield that $Y$ is weakly continuous at $\infty$ (resp. $0$). 
 Corollary~\ref{corollary:LDimpliesCM} guarantees the complete monotonicity of $X$ and $Y$.  

Conversely, assume $X$ is completely monotone, and has $0$, $\infty$ both absorbing. 

 From,  and in the notation of Proposition~\ref{lemma:kernels}\eqref{eq-lemma:laplacedualpreform2}-\eqref{eq-lemma:laplacedualpreform3}, and by $0\cdot \infty^-$, $\infty\cdot 0^+$: for  $t\in [0,\infty)$,  writing $q_t(\infty,\cdot):=q_t(\infty-,\cdot)$ and $q_t(0,\cdot):=q_t(0+,\cdot)$, then for $y\in [0,\infty]$, $x\in (0,\infty)$,  
\begin{equation}\label{fund3.}
\int e^{-uy}p_t(x,\dd u)=\int e^{-xv}q_t(y,\dd v).
\end{equation}
Moreover,  \eqref{fund3.} holds also, first, for $x=0$ due to the absorptivity of $0$ for $X$ and $0\cdot \infty^-$; second,  for $x= \infty$ by the absorptivity of $\infty$ for $X$ and $\infty\cdot 0^+$. Thus \eqref{fund3.} persists for arbitrary $(x,y)\in [0,\infty]^2$. In order to deduce that $X$ is in $(0\cdot \infty^-,\infty\cdot 0^+)$-Laplace duality with a positive Markov process it remains to apply Proposition~\ref{proposition:general-dual}.

\eqref{eq:laplacedualinftyabsorbing2} is a mere specification of \eqref{LD-main} to $y=0$ using $\infty\cdot 0^+$, and separately to $y=\infty$ using $0\cdot \infty^-$.

\textbf{Part} \ref{thmcompletemonotonicity4}. We work throughout this part under the conventions $ 0^+\cdot\infty$, $\infty^-\cdot 0$. 

The necessity of the conditions of the theorem follows from Item~\ref{thmcompletemonotonicity3} on transposing. As for sufficiency, assume then that $X$ is completely monotone and weakly continuous at both $0$ and $\infty$. 

From, and in the notation of Proposition~\ref{lemma:kernels}\eqref{eq-lemma:laplacedualpreform1}: for  $t\in [0,\infty)$,  $y\in (0,\infty)$, $x\in (0,\infty)$,  
\begin{equation}\label{fund3}
\int e^{-uy}p_t(x,\dd u)=\int e^{-xv}q_t(y,\dd v);
\end{equation}
but, on passing to the limit $x\uparrow\infty$ also for $x=\infty$ by the weak continuity of $X$ at $\infty$ and $\infty^-\cdot 0$; likewise, on passing to the limit $x\downarrow 0$ also for $x=0$ by the weak continuity of $X$ at $0$ and $0^+\cdot\infty$. Moreover,  setting $q_t(0,\cdot):=\delta_0$ and $q_t(\infty,\cdot):=\delta_\infty$, \eqref{fund3} continues to hold, first, for $y=0$ due to $\infty^-\cdot 0$; second, for $y=\infty$ due to  $ 0^+\cdot\infty$. Thus \eqref{fund3} persists for arbitrary $(x,y)\in [0,\infty]^2$. In order to deduce that $X$ is in $(0^+\cdot\infty,\infty^-\cdot 0)$-Laplace duality with a positive Markov process it remains to apply Proposition~\ref{proposition:general-dual}.
\end{proof}

In applying Theorem \ref{thm:completemonotonicity}, in the direction of starting with a process $X$ and attempting to conclude the existence of the dual $Y$, one may be led to performing operations of restriction followed by/or just of extension of state-space. The issue of restriction is considered in
\begin{remark}\label{rem:restriction}
Let $(0,\infty)\subset E'\subset [0,\infty]$. 
Suppose the $[0,\infty]$-valued Markov process $X$ has the boundary points of  $[0,\infty]\backslash E'$ non-sticky. Then we may introduce the positive Markov process $X'$ having state-space  $E'$ and the probabilities $\PP_x':=\PP_x\vert_{E'}$ for $x\in E'$. It is, as far as the finite-dimensional distributions are concerned,  the same process when looked at only on points of issue from $E'$: ${X'_t}_\star \PP_x={X'_t}_\star\PP'_x$ for all $t\in [0,\infty)$ and all $x\in E'$. Thus, if $X$ is completely monotone to begin with, then the resulting process $X'$ verifies complete monotonicity on $E'$: for all $t\in [0,\infty)$, for all $\downarrow$  maps $f:E'\to [0,\infty)$ that are completely monotone on $(0,\infty)$, we have that $\EE_x'[f(X'_t)]$ is $\downarrow$ in $x\in E'$, completely monotone in $x\in ( 0,\infty)$. (But $X'$ can have this property without $X$ having been completely monotone to begin with.) Such restriction of the process can also be iterative,  first removing $0$ and then $\infty$ or vice versa.
\end{remark}
As for the problem of extension, let $(0,\infty)\subset E'\subset [0,\infty]$ and suppose $X'$ is, under the probabilities $(\PP_x')_{x\in E'}$, a Markov process valued in $E'$ that is completely monotone on $E'$. Then,  for $E'\subset E\subset [0,\infty]$, the $E$-valued  Markov process $X$ having the probabilities of $X'$ and the points $E\backslash E'$ absorbing  is completely monotone. The next proposition further shows that given complete monotonicity on $(0,\infty)$ of a process $X$, one can  extend $X$ continuously at the boundaries and find a Laplace dual.
\begin{proposition}\label{prop:continuousextensionat0}
Let $X'$ be a $(0,\infty)$-valued Markov process under the probabilities $(\PP'_x)_{x\in (0,\infty)}$, such that  $(0,\infty)\ni x\mapsto \EE_x'[e^{-X_t'y}]$ is completely monotone for all $y\in (0,\infty)$.  Then, for all $t\in[ 0,\infty)$, the probabilities $p_t(x,\cdot):=\PP_x'(X'_t\in \cdot)$  admit a weak limit as $x\downarrow 0$ (resp.  $x\uparrow \infty$) and the latter, call it $p_t(0,\cdot)$ (resp. $p_t(\infty,\cdot)$, is a probability on $[0,\infty)$ (resp. $[0,\infty]$). Moreover, we have the following two assertions.
\begin{enumerate}[(i)]
\item\label{extensionat0}The    family  $p:=(p_t(x,\cdot))_{(t,x)\in [0,\infty)\times [0,\infty)}$ of probabilities on $[0,\infty)$ has  the measurablity, normality  and semigroup properties. The associated $[0,\infty]$-valued Markov process $X$ with transition kernels $p$ on $[0,\infty)$ and the boundary $\infty$ absorbing is weakly-$[0,\infty)$ continuous at $0$ and also completely monotone. 
In particular, $X$ admits a $(0^+\cdot\infty,\infty\cdot 0^+)$-Laplace dual process.
\item \label{extensionatinfinity} The   family  $p:=(p_t(x,\cdot))_{(t,x)\in [0,\infty)\times [0,\infty]}$ of probabilities on $[0,\infty]$ has also the measurability, normality and semigroup properties. The associated Markov process $X$ with transition kernels $p$ is weakly continuous at $0$ and $\infty$ and also completely monotone. In particular, $X$ admits a $(0^+\cdot\infty,\infty^-\cdot 0)$-Laplace dual process.
\end{enumerate}
\end{proposition}
It will be clear from the proof that we have also the following version. 

\emph{Let $X'$ be a $[0,\infty)$-valued Markov process under the probabilities $(\PP'_x)_{x\in[0,\infty)}$, weakly continuous at zero and such that,  for all $y\in (0,\infty)$,  $ \EE_x'[e^{-X_t'y}]$ is completely monotone in $x\in (0,\infty)$, continuous in $x\in  [0,\infty)$. Then, for all $t\in[ 0,\infty)$, as  $x\uparrow \infty$, the probabilities $p_t(x,\cdot):=\PP_x'(X'_t\in \cdot)$  admit a weak limit as a probability $p_t(\infty,\cdot)$ on  $[0,\infty]$ and \ref{extensionatinfinity}  holds.}

The same proof works because the extension to \ref{extensionatinfinity}  is done independently after \ref{extensionat0}. 
\begin{proof}[Proof of Proposition~\ref{prop:continuousextensionat0}]
 By Prokhorov's theorem, the fact that  Laplace transforms  determine probabilities on $\BB_{[0,\infty]}$, and the complete monotonicity ($\therefore$ $\downarrow$) of $(0,\infty)\ni x\mapsto \EE_x'[e^{-X_t'y}]$ for all $y\in (0,\infty)$, we see that for each $t\in [0,\infty)$ the weak limit $p_t(0,\cdot)$ of the probabilities $p_t(x,\cdot)$ on $[0,\infty]$ exists as $x\downarrow 0$. With $x\in (0,\infty)$ arbitrary, since for all $y\in (0,\infty)$, $$\int e^{-uy}p_t(x,\dd u)\leq \int e^{-uy}p_t(0,\dd u)\leq p_t(0,[0,\infty)),$$ on letting $y\downarrow 0$, we deduce that $p_t(0,[0,\infty))=1$. 

By the same set of tokens we see that the weak limit $p_t(\infty,\cdot)$ of the probabilities $p_t(x,\cdot)$ on $[0,\infty]$ as $x\uparrow \infty$ exists as a probability on $[0,\infty]$.  (We stress that here there is possibly a mass at $\infty$.)

\ref{extensionat0}. By what we have just established we may 
consider $p:=(p_t(x,\cdot))_{(t,x)\in [0,\infty)\times [0,\infty)}$ as a family of probabilities on $[0,\infty)$. It is clear that $p_0(0,\{0\})=1$, so $p$ has the normality property. Let us check the semigroup property of $p$ (for point of issue $0$, it being clear for points of issue out of $(0,\infty)$). For $\{s,t\}\subset [0,\infty)$:
\begin{align*}
\int e^{-uy}p_{t+s}(0,\dd u)&=\lim_{x\downarrow 0}\int e^{-uy}p_{t+s}(x,\dd u)=\lim_{x\downarrow 0}\int \overbrace{\int e^{-u'y}p_s(u,\dd u')}^{\text{bdd cts in }u\in [0,\infty)}p_{t}(x,\dd u)\\
&=\int \int e^{-u'y}p_s(u,\dd u')p_{t}(0,\dd u),\quad y\in (0,\infty);
\end{align*}
since Laplace transforms determine probabilities on $\BB_{[0,\infty]}$ the semigroup property follows. Thus we can extend $X'$ to $X$ by asking that $X$ has the transition kernels $p$ on $[0,\infty)$ and $\infty$ absorbing. It is plain that the resulting positive Markov processes is weakly-$[0,\infty)$ continuous at $0$, also completely monotone. Consequently, Theorem~\ref{thm:completemonotonicity}\ref{thmcompletemonotonicity1}  applies and ensures existence of a positive Markov process that is in $(0^+\cdot\infty,\infty\cdot 0^+)$-Laplace duality therewith.

\ref{extensionatinfinity}. Concerning $p$, by the findings of \ref{extensionat0} it remains to check the semigroup property for point of issue $\infty$, but this follows essentially verbatim as it was done in \ref{extensionat0} for $0$. Thus  we end up with a normal measurable Markovian family of kernels $p$. The associated $[0,\infty]$-valued Markov process is weakly continuous both at $0$ and $\infty$ and is completely monotone. In turn, Theorem~\ref{thm:completemonotonicity}\ref{thmcompletemonotonicity4} applies and ensures existence of a positive Markov process that is in $(0^+\cdot\infty,\infty^{-}\cdot 0)$-Laplace duality therewith.
\end{proof}
Remark~\ref{rem:restriction} and Proposition~\ref{prop:continuousextensionat0} can be applied one after the other. We have left in them unaddressed the issue of whether or not the resulting processes have  modifications with nice, say c\`adl\`ag paths: the discussion concerned only the procurement of a Markovian family of kernels restricting/extending the given one. Though, this problem is  trivial at least when extending a given c\`adl\`ag process by absorbing points.

We close this subsection by making precise how the insistance on the state-spaces $[0,\infty]$ in Laplace duality represents no loss of generality relative to the situation in which we would have also allowed $(0,\infty)$, $[0,\infty)$, etc.
\begin{remark}\label{extended:wlog}
Let $X$ and $Y$ be Markov processes with respective state-spaces $E$ and $F$, subsets of $[0,\infty]$, supersets of $(0,\infty)$, that are in Laplace duality on $E\times F$: 
\begin{equation}\label{eq:restricted}
\EE_x[e^{-X_ty}]=\EE^y[e^{-xY_t}],\quad t\in [0,\infty),\, (x,y)\in E\times F,
\end{equation}
under some conventions on $0\cdot\infty$, $\infty\cdot 0$, when, respectively, $(0,\infty)\in E\times F$, $(\infty,0)\in E\times F$. We claim that there exist extensions of $X$ and $Y$ to positive Markov processes, with the outstanding points absorbing, that are in Laplace duality under a suitable choice of the conventions for $0\cdot\infty$, $\infty\cdot 0$ extending the given ones. Indeed, if $\infty$ is not in the state-space of $Y$, then we extend $Y$ to $\infty$ by making the latter absorbing and, if $0$ was in the state-space of $X$, agreeing on the convention $0^+\cdot\infty$. Likewise, if $0$ is not in the state-space of $Y$, then we extend $Y$ to $0$ by making the latter absorbing and, if $\infty$ was in the state-space of $X$, agreeing on the convention $\infty^-\cdot 0$. We repeat this in an analogous manner to extend also $X$ to $[0,\infty]$. 
\end{remark}
\subsection{Feller property, excessive/invariant functions and measures, long-term limiting laws}\label{subsec:Fellerproperty}
Let us provide some technical background on Feller processes, referring  the reader to Kallenberg \cite[Chapter~19]{Kallenberg} and to B\"ottcher et al. \cite[Chapter I]{zbMATH06256582} for  detailed accounts. 

Under a $\mathsf{C}_0(E)$-Feller process with values in a locally compact separable metric space $E$ we shall understand then an $E$-valued Markov process $X$, whose semigroup $P=(P_t)_{t\in [0,\infty)}$ satisfies the two properties:
\begin{align}
&P_t\mathsf{C}_0(E)\subset \mathsf{C}_0(E),\quad t\in [0,\infty);\label{FellerI}\\
&\lim_{t\downarrow 0}P_tf(x)=f(x),\quad x\in E,\, f\in  \mathsf{C}_0(E).\label{FellerII}
\end{align}
When $E$ is compact, we shall also say that $X$ is $\mathsf{C}(E)$-Feller if \eqref{FellerI}-\eqref{FellerII} hold with $\mathsf{C}(E)$ instead of $\mathsf{C}_0(E)$. For a compact $E$, $\mathsf{C}_0(E)=\mathsf{C}(E)$; we make the notational distinction only for emphasis.

Property \eqref{FellerII} is automatic if $X$ is right-continuous and we may allow in the preceding $X$ to be non-conservative. \label{desc.non-cons} A possibly non-conservative $E$-valued Markov process $X=(X_t)_{t\in [0,\zeta)}$ with lifetime $\zeta$  under the probabilities $(\PP_x)_{x\in E}$ obeys \eqref{markov-property} only ``a.s.-$\PP_x$ on $\{s<\zeta\}$'', but has otherwise the same properties of  measurability, and normality ($\PP_x(X_0=x,0<\zeta)=1$ for all $x\in E$). Furthermore, with $p_t(x,\cdot):={X_t}_\star\PP_x\vert_{\{t<\zeta\}}$  for $x\in E$,   $P_t$ is defined out of the family $(p_t(x,\cdot))_{x\in E}$, exactly as in \eqref{eq:semigroup}, this for all $t\in [0,\infty)$. By one-point compactification or addition of an isolated point, according as to the whether $E$ is not or is compact, the non-conservative case reduces to the conservative one \cite[Lemma~19.13]{Kallenberg}. The Feller property of a family of (sub)Markovian transition kernels $p$, or of their associated semigroup $P$, is defined in the clear way without reference to $X$. For a Feller semigroup a process realization thereof having  sample paths in $\mathbb{D}_E$ always exists \cite[Theorem~19.15]{Kallenberg} and the strong Markov and quasi left-continuity properties are automatic \cite[Theorem~19.17, Proposition~25.20]{Kallenberg}.

Recall also the notion of the strong (infinitesimal) generator for a Feller process $X$ with semigroup $P$, state-space $E$: it is the linear operator $\XX_{\mathrm{s}}$ defined on 
\begin{equation*}
D_{\XX_{\mathrm{s}}}:=\left\{f\in \mathsf{C}_0(E): \underset{t\downarrow 0}{\lim} \frac{P_tf-f}{t} \text{ exists as a uniform limit}\right\}
\end{equation*}
and given by
\begin{equation*}
\XX_{\mathrm{s}}f:=\underset{t\downarrow 0}{\lim} \frac{P_tf-f}{t}\text{ for } f\in D_{\XX_{\mathrm{s}}}.
\end{equation*}
The Feller property ensures that it coincides with the pointwise version on restriction to $\mathsf{C}_0(E)$ \cite[Theorem 1.33]{zbMATH06256582}, more precisely
\begin{equation}\label{well-known-generators}
\XX_{\mathrm{s}}=\mathcal{X}\cap (\mathsf{C}_0(E)\times \mathsf{C}_0(E)),
\end{equation} where $\mathcal{X}$ is given by \eqref{eq:generator}. The generator of a Feller process provides the fundamental Dynkin's martingales, see for instance \cite[Theorem~1.36 \& sentence just after Definition~1.35]{zbMATH06256582}: for any $f\in D_{\XX_{\mathrm{s}}}$, the process
\begin{equation}\label{page-dynkin-mtgs}
\left(f(X_t)-\int_0^t\XX_{\mathrm{s}}f(X_s)\dd s\right)_{t\in [0,\infty)} \text{ is a martingale}.
\end{equation} A core for $\XX_{\mathrm{s}}$ is any linear $D\subset D_{\XX_{\mathrm{s}}}$ such that $\XX_{\mathrm{s}}\vert_D$ has closure $\XX_{\mathrm{s}}$ (in $\mathsf{C}_0(E)\times \mathsf{C}_0(E)$, where we endow $\mathsf{C}_0(E)$ with the supremum  norm $\Vert\cdot\Vert_\infty$, i.e. the topology of uniform convergence).

With these reminders out of the way, we first address the right-continuity \eqref{FellerII}:
\begin{equation}\label{felller-duality}
P_tf(x)\underset{t\downarrow 0}{\longrightarrow} f(x),  \quad \ x\in [0,\infty],\, f\in \mathsf{C}([0,\infty]),
\end{equation} for the semigroup  $P$ of a positive Markov process $X$ in Laplace duality. 
\begin{proposition}\label{prop:rightcontinuitysemigroup} 
 Let $X$ (semigroup $P$) be in $(\star_1,\star_2)$-Laplace duality with $Y$ (semigroup $Q$). 
Then $P$ satisfies \eqref{felller-duality} if and only if, respectively depending on the conventions $(\star_1,\star_2$), for all $y\in (0,\infty)$, \begin{equation}\label{cond:rightcontinuityPt}
Q_t\ee_x(y)\underset{t\downarrow 0}{\longrightarrow} \ee_x(y), \quad x\in  (0,\infty),
\end{equation} and
\begin{enumerate}[(i)]
\item when $(\star_1,\star_2)=(0^+\cdot\infty, \infty \cdot 0^+)$: $ q_t(y,[0,\infty))\underset{t\downarrow 0}{\longrightarrow} 1$. 
\item  when $(\star_1,\star_2)=(0\cdot\infty^-, \infty^{-}\cdot 0)$:  $ q_t(y,\{0\})\underset{t\downarrow 0}{\longrightarrow} 0$. 
 \item when $(\star_1,\star_2)=(0\cdot\infty^-, \infty\cdot 0^+)$: no further conditions. 
\item \label{feller-duality:iv}  when $(\star_1,\star_2)=(0^+\cdot\infty, \infty^{-}\cdot 0)$: $q_t(y,\{0\})\underset{t\downarrow 0}{\longrightarrow} 0$ and $q_t\big(y,[0,\infty)\big)\underset{t\downarrow 0}{\longrightarrow} 1 $. 
\end{enumerate}
These conditions are met if $Y$ admits a version that is right-continuous at zero (in particular they are met if $Y$ has a $\mathbb{D}_{[0,\infty]}$-valued version).
\end{proposition}
\begin{proof}
We write it out for $(\star_1,\star_2)=(0^+\cdot\infty, \infty^{-}\cdot 0)$-duality, the other cases being similar (and easier). Recall first that 
by Stone-Weierstrass the linear span of $\{\ee^{y}: y\in (0,\infty)\}\cup \{1\}$ is dense in $(\mathsf{C}([0,\infty]),\Vert\cdot\Vert_\infty)$. Assume that \ref{feller-duality:iv} holds true. By the duality relationship, for all $t\in [0,\infty)$, for all $y\in (0,\infty)$,
\begin{equation}\label{rightcontinuityintermsofq}
\begin{aligned}
&P_t\ee^y(x)= Q_t\ee_x(y)\underset{t\downarrow 0}{\longrightarrow} \ee_x(y)=\ee^y(x),\quad x\in (0,\infty), 
\\
&P_t\ee^y(\infty)=q_t(y,\{0\})\underset{t\downarrow 0}{\longrightarrow}0=\ee^y(\infty) \text{ and } \\
&P_t\ee^y(0)=q_t\big(y,[0,\infty)\big)\underset{t\downarrow 0}{\longrightarrow}1=\ee^y(0).
\end{aligned}
\end{equation}
Now let $f\in \mathsf{C}([0,\infty])$ be arbitrary and pick  $(f_n)_{n\in \mathbb{N}}$, a sequence of functions belonging to the linear span of $\{\ee^{y}: y\in (0,\infty)\}\cup \{1\}$, such that $\lim_{n\to\infty}\Vert f_n-f \Vert_\infty=0$. For all $t\in [0,\infty)$, using the fact that $\Vert P_tf \Vert_{\infty} \leq \Vert f \Vert_\infty$ we see that, for all $x\in [0,\infty]$,
\begin{align*}
|P_tf(x)-f(x)|&\leq \Vert P_tf-P_tf_n\Vert_\infty+|P_tf_n(x)-f_n(x)|+\Vert f_n-f\Vert_\infty  \\
&\leq 2\Vert f_n-f\Vert_\infty+ |P_tf_n(x)-f_n(x)|.
\end{align*}
By \eqref{rightcontinuityintermsofq}, $\lim_{t\downarrow 0}|P_tf_n(x)-f_n(x)|=0$. Thus, by letting $t\downarrow 0$ and thereafter $n\to\infty$ in the preceding display, we get
$\underset{t\rightarrow 0}{\lim} |P_tf(x)-f(x)|=0$, which is  \eqref{felller-duality}.

The fact that \ref{feller-duality:iv} is necessary for \eqref{felller-duality} follows readily on exchanging  the l.h.s. and r.h.s. of the equalities in \eqref{rightcontinuityintermsofq}. 

If now $Y$ is right-continuous at zero then, for all $y\in (0,\infty)$, we have that \eqref{cond:rightcontinuityPt}, $q_t(y,\{0\})=\mathbb{P}^y(Y_t=0)\underset{t\downarrow 0}{\longrightarrow} 0$ and $q_t(y,\{\infty\})=\mathbb{P}^y(Y_t=\infty)\underset{t\downarrow 0}{\longrightarrow} 0$ hold for all $y\in (0,\infty)$ by bounded convergence (since $Y_{0+}=Y_0=y\in (0,\infty)$ a.s.-$\PP^y$), which is \ref{feller-duality:iv}.
\end{proof}

We now study the more restrictive Feller property \eqref{FellerI}.
\begin{proposition}\label{propositon:Feller-property}
 Let $X$ and $Y$ be positive Markov processes, $X$ satisfying \eqref{felller-duality}. 
  We have the following assertions.
\begin{enumerate}[(i)]
\item\label{propositon:Feller-property:i} If $X$ is in $(0^+\cdot\infty,\infty\cdot 0^+)$-Laplace duality with  $Y$, then $X$ is $\mathsf{C}([0,\infty])$-Feller if and only if  $Y$ has $0$ non-sticky, in which case (the possibly submarkovian) $p$ restricted to $[0,\infty)$ has the $\mathsf{C}_0([0,\infty))$-Feller property too. 
\item\label{propositon:Feller-property:ii} If $X$ is in $(0\cdot \infty^-,\infty^-\cdot 0)$-Laplace duality with  $Y$, then $X$ is $\mathsf{C}([0,\infty])$-Feller if and only if  $Y$ has $\infty$ non-sticky, in which case (the possibly submarkovian) $p$ restricted to $(0,\infty]$ has the $\mathsf{C}_0((0,\infty])$-Feller property too. 
\item\label{propositon:Feller-property:iv} If  $X$ is in $(0\cdot \infty^-,\infty\cdot 0^+)$-Laplace duality with $Y$, then $X$ is $\mathsf{C}([0,\infty])$-Feller if and only if $Y$ has $0$ and $\infty$ both non-sticky, in which case (the possibly submarkovian) $p$ restricted to $(0,\infty)$ has the $\mathsf{C}_0((0,\infty))$-Feller property too. 
\item\label{propositon:Feller-property:iii} If  $X$ is in $(0^+\cdot\infty,\infty^-\cdot 0)$-Laplace duality with $Y$, then $X$ is $\mathsf{C}([0,\infty])$-Feller.
\end{enumerate}
\end{proposition}
\begin{proof}
\ref{propositon:Feller-property:i}. If $Y$ has $0$ non-sticky, then under $0^+\cdot\infty$, $e^{-xY_t}$ is continuous in $x\in [0,\infty]$ a.s.-$\PP^y$ for all $y\in (0,\infty)$ and $t\in [0,\infty)$. Combining this with the Laplace duality relation \eqref{LD-main} and bounded convergence we find that, for all $t\in [0,\infty)$, $P_t\{\ee^y:y\in (0,\infty)\}\subset \mathsf{C}([0,\infty])$. But the linear span of $\{\ee^y:y\in (0,\infty)\}\cup \{1\}$ is dense in $(\mathsf{C}([0,\infty]),\Vert\cdot\Vert_\infty)$ and we deduce that $P_t\mathsf{C}([0,\infty])\subset   \mathsf{C}([0,\infty])$. Together with  \eqref{felller-duality} it means that $X$ is  $\mathsf{C}([0,\infty])$-Feller. Conversely, if $X$ is  $\mathsf{C}([0,\infty])$-Feller, then  for all $t\in [0,\infty)$ and  $y\in (0,\infty)$, \eqref{LD-main} under $\infty\cdot 0^+$ and the continuity of $([0,\infty]\ni x\mapsto \EE_x[e^{-X_ty}])$ at $\infty-$ require $\PP^y(Y_t=0)=0$; in addition, the complete monotonicity of $Y$ together with Proposition~\ref{lemma:kernels} ensures that $([0,\infty]\ni y\mapsto \PP^y(Y_t=0))$ is $\downarrow$, so that also $\PP^\infty(Y_t=0)=0$, and thus  $Y$ has $0$ non-sticky. The ``in which case'' statement 
follows on noting the density of the linear span of $\{\ee^y\vert_{[0,\infty)}:y\in (0,\infty)\}$ in $(\mathsf{C}_0([0,\infty)),\Vert\cdot\Vert_\infty)$, the restriction of $p$ to $[0,\infty)$ being in fact a (sub)Markovian family of transition kernels, since $X$ has $\infty$ absorbing. 

\ref{propositon:Feller-property:ii}. Similar to \ref{propositon:Feller-property:i} except that we consider $\{(1-\ee^y)\vert_{(0,\infty]}:y\in (0,\infty)\}$ in lieu of $\{\ee^y\vert_{[0,\infty)}:y\in (0,\infty)\}$, and the $\downarrow$ character of $([0,\infty]\ni y\mapsto \PP^y(Y_t<\infty))$ in lieu of that of $([0,\infty]\ni y\mapsto \PP^y(Y_t=0))$.

\ref{propositon:Feller-property:iv}. For the ``in which case'' part observe that  the linear span of $\{(\ee^{y_1}-\ee^{y_2})\vert_{(0,\infty)}:(y_1,y_2)\in (0,\infty)^2\}$ is an algebra of functions separating the points of $(0,\infty)$ and vanishing nowhere on $(0,\infty)$, therefore is dense in $(\mathsf{C}_0((0,\infty)),\Vert\cdot\Vert_\infty)$ by Stone-Weierstrass. The remainder of the straightforward argument employing \eqref{LD-main} is left to the reader.

\ref{propositon:Feller-property:iii}.  Under $0^+\cdot\infty$, $\infty^-\cdot 0$, $e^{-xY_t}$ is continuous in $x\in [0,\infty]$ for all $y\in (0,\infty)$ and $t\in [0,\infty)$.  We conclude as in \ref{propositon:Feller-property:i}.
\end{proof}

The next proposition details invariant sets  for the semigroup and cores for the generator of a process in Laplace duality.
Set:
\begin{align*}
\mathsf{cm}_{[0,\infty)}^\pm&:=\left\{\int \ee ^y\nu(\dd y):\nu\text{ a finite signed measure on $\BB_{(0,\infty)}$}\right\}\subset \mathsf{C}_0([0,\infty)),\\
\mathsf{cm}_{(0,\infty]}^\pm&:=\left\{\int (1-\ee^y)\nu(\dd y):\nu\text{ a finite signed measure on $\BB_{(0,\infty)}$}\right\}\subset \mathsf{C}_0((0,\infty]),\\
\mathsf{cm}_{(0,\infty)}^\pm&:=\left\{\int (\ee^{y_1}-\ee^{y_2})\gamma(\dd y):\gamma\text{ a finite signed measure on $\BB_{(0,\infty)^2}$}\right\}\subset \mathsf{C}_0((0,\infty)), \text{ and }\\
\mathsf{cm}_{[0,\infty]}^\pm&:=\left\{c+\int \ee^y\nu(\dd y): \ c\in \mathbb{R}, \ \nu\text{ a finite signed measure on $\BB_{(0,\infty)}$}\right\}\subset \mathsf{C}([0,\infty]).
\end{align*}
Notice that, under the conventions $0^+\cdot \infty$, $\infty^{-}\cdot 0$, we may identify
\begin{equation}\label{eq:identify-cm[0,infty]}
\mathsf{cm}_{[0,\infty]}^\pm=\left\{\int \ee^y\nu(\dd y):\nu\text{ a finite signed measure on $\BB_{[0,\infty]}$}\right\}.
\end{equation}

\begin{proposition}\label{prop:cores}
Let $X$ be a positive Markov process. We denote  by $P$ the possibly submarkovian semigroup and by $\XX_{\mathrm{s}}$ the strong generator  of each of the  Feller processes referred to in \ref{prop:cores:i}-\ref{prop:cores:iv} below, i.e. of $X$, but with the indicated restricted state-space (in each case the assumed duality ensures the absorptivity property that renders the restriction meaningful, cf. Proposition~\ref{propositon:Feller-property}). 
\begin{enumerate}[(i)]
    \item\label{prop:cores:i} If $X$ admits a $(0^+\cdot\infty,\infty\cdot 0^+)$-Laplace dual and is $\mathsf{C}_0([0,\infty))$-Feller on restriction to $[0,\infty)$, then  $P_t\mathsf{cm}_{[0,\infty)}^\pm\subset \mathsf{cm}_{[0,\infty)}^\pm$ for all $t\in [0,\infty)$. If further $\mathsf{cm}_{[0,\infty)}^\pm\subset D_{\XX_{\mathrm{s}}}$, then $\mathsf{cm}_{[0,\infty)}^\pm$ is a core for $\XX_{\mathrm{s}}$.
    \item \label{prop:cores:ii} If $X$ 
    admits a $(0\cdot \infty^-,\infty^-\cdot 0)$-Laplace dual and is $\mathsf{C}_0((0,\infty])$-Feller on restriction to $(0,\infty]$, then $P_t\mathsf{cm}_{(0,\infty]}^\pm\subset \mathsf{cm}_{(0,\infty]}^\pm$ for all $t\in [0,\infty)$. If further $\mathsf{cm}_{(0,\infty]}^\pm\subset D_{\XX_{\mathrm{s}}}$, then $\mathsf{cm}_{(0,\infty]}^\pm$ is a core for $\XX_{\mathrm{s}}$.
    \item \label{prop:cores:iii} If $X$ admits  a $(0\cdot \infty^-,\infty\cdot 0^+)$-Laplace dual and is $\mathsf{C}_0((0,\infty))$-Feller on restriction to $(0,\infty)$, then $P_t\mathsf{cm}_{(0,\infty)}^\pm\subset \mathsf{cm}_{(0,\infty)}^\pm$  for all $t\in [0,\infty)$. If further $\mathsf{cm}_{(0,\infty)}^\pm\subset D_{\XX_{\mathrm{s}}}$, then $\mathsf{cm}_{(0,\infty)}^\pm$ is a core for $\XX_{\mathrm{s}}$.
        \item \label{prop:cores:iv} If $X$ 
        admits a $(0^+\cdot \infty,\infty^-\cdot 0)$-Laplace dual, then (it is  $\mathsf{C}([0,\infty])$-Feller by Proposition~\ref{propositon:Feller-property}\ref{propositon:Feller-property:iii} and) $P_t\mathsf{cm}^{\pm}_{[0,\infty]}\subset \mathsf{cm}^{\pm}_{[0,\infty]}$  for all $t\in [0,\infty)$. If further $ \mathsf{cm}^{\pm}_{[0,\infty]}\subset D_{\XX_{\mathrm{s}}}$, then $\mathsf{cm}_{[0,\infty]}^\pm$ is a core for $\XX_{\mathrm{s}}$.
\end{enumerate}
\end{proposition}
\begin{proof}
In each item, the core property of the invariant set follows from the Feller property assumed on the process, the fact that it is contained in the domain  of the strong generator, Stone-Weierstrass and a well-known sufficient condition for cores \cite[Proposition~19.9]{Kallenberg}. We focus therefore on the invariance property, providing full details only for the first case and giving the gist of the argument for the remaining ones. Let $t\in[0,\infty)$ and let $\nu$ be a finite signed measure on $\BB_{(0,\infty)}$. In each of the cases call $Y$ the process for which $X$ is in Laplace duality with $Y$ of the type specified in that case.

\ref{prop:cores:i}.  Work under $0^+\cdot\infty,\infty\cdot 0^+$. Set $f_\nu:=\int \ee^y\nu(\dd y)$. By Tonelli, duality and $0^+\cdot\infty$, we have, for all $ x\in [0,\infty)$,\small
\begin{align*}
    P_tf_\nu(x)&=\int\EE_x[e^{-X_ty}]\nu(\dd y)=\int\EE^y[e^{-xY_t}]\nu(\dd y)=\int e^{-xv}\int \mathbb{P}^y(Y_t\in \dd v,Y_t<\infty)\nu(\dd y).
\end{align*}\normalsize
By assumption $X$ is Feller and so Proposition~\ref{propositon:Feller-property}\ref{propositon:Feller-property:i} ensures that $0$ is non-sticky for $Y$, whence  the finite signed measure $\int \mathbb{P}^y(Y_t\in \cdot,Y_t<\infty)\nu(\dd y)$ is carried by $(0,\infty)$. It follows that $P_tf_\nu\in \mathsf{cm}_{[0,\infty)}^\pm$, and we conclude by Stone-Weierstrass.

\ref{prop:cores:ii}. By Proposition~\ref{propositon:Feller-property}\ref{propositon:Feller-property:ii}, $\infty$ is non-sticky for $Y$, whence  the finite signed measure $\int \mathbb{P}^y(Y_t\in \cdot,Y_t>0)\nu(\dd y)$ is carried by $(0,\infty)$.

\ref{prop:cores:iii}. For a finite signed measure $\gamma$ on $\BB_{(0,\infty)^2}$ and $x\in (0, \infty)$ we may write 
\begin{align*}
&\int e^{-x v}\int \PP^{y_1}(Y_t\in \dd v)\gamma(\dd y)-\int e^{-x v}\int \PP^{y_2}(Y_t\in \dd v)\gamma(\dd y)\\
&\qquad\qquad\qquad=\int e^{-xv_1}-e^{-xv_2}\left[\int \PP^{y_1}(Y_t\in \dd v_1)\PP^{y_2}(Y_t\in \dd v_2)\gamma(\dd y)\right].
\end{align*}
 By Proposition~\ref{propositon:Feller-property}\ref{propositon:Feller-property:iv}, $0$ and $\infty$ are both non-sticky for $Y$, whence  the finite signed measure $\int ({Y_t}_\star \PP^{y_1})\times ({Y_t}_\star \PP^{y_2})\gamma(\dd y)$ is carried by $(0,\infty)^2$.

\ref{prop:cores:iv}. Work under $0^+\cdot \infty,\infty^-\cdot 0$. For a finite signed measure $\nu$ on $[0,\infty]$, by Laplace duality and Tonelli, $P_t\int \ee^y\nu(\dd y)=\int  \ee^v\int \PP^y(Y_t\in \dd v)\nu(\dd y)$. Recalling \eqref{eq:identify-cm[0,infty]}  we find that $P_t\mathsf{cm}^{\pm}_{[0,\infty]}\subset \mathsf{cm}^{\pm}_{[0,\infty]}$. 
\end{proof}

The difficulty in practice is checking that the linear subspaces $\mathsf{cm}$ delineated above belong to the domain of the strong generator. Though, granted the Feller properties, we know from \eqref{well-known-generators} that it sufficies to check that they belong to the domain of the  pointwise generator and that the result of the action of the  pointwise generator on them is contained in the relevant  $\mathsf{C}_0$-space. We illustrate this in
\begin{corollary}\label{prop:cmindomain}  Let the positive Markov process $X$ (semigroup $P$, generator $\XX$) be $\mathsf{C}([0,\infty])$-Feller and let it admit a $(0^+\cdot \infty,\infty^-\cdot 0)$-Laplace dual. If $\{\ee^y: y\in (0,\infty)\}\subset D_{\XX}$, 
$$\XX \ee^y\in \mathsf{C}([0,\infty]) \text{ for all $y\in (0,\infty)$ and }\sup_{(x,y)\in [0,\infty]\times (0,\infty)}\vert\XX \ee^y(x)\vert<\infty,$$
then $\mathsf{cm}^{\pm}_{[0,\infty]}$ is a core for $\XX_{\mathrm{s}}$ (in particular $\ee^y\in D_{\XX_{\mathrm{s}}}$ for all $y\in (0,\infty)$).
\end{corollary}
\begin{proof}
By the Feller property, the fact that $\ee^y\in D_{\XX}$ and the assumption $\XX \ee^y\in \mathsf{C}([0,\infty]) $, we see that $\ee^y\in D_{\XX_{\mathrm{s}}}$ for all $y\in (0,\infty)$. Therefore, from the Dynkin martingales \eqref{page-dynkin-mtgs} for the  equality and by bounded convergence for the limit,
\begin{equation}\label{bplimit}\frac{1}{t}\left(P_t\ee^y(x)-\ee^y(x)\right)=\frac{1}{t}\mathbb{E}_x\left[\int_{0}^{t}\XX \ee^y(X_s)\dd s\right]\underset{t\downarrow 0}{\longrightarrow} \XX\ee^y(x),\quad x\in [0,\infty],\, y\in (0,\infty).
\end{equation}
Let $f\in \mathsf{cm}^{\pm}_{[0,\infty]}$. By definition of $\mathsf{cm}^{\pm}_{[0,\infty]}$ and Tonelli  there is a finite signed measure $\nu$ on $\BB_{(0,\infty)}$ such that, for all $x\in [0,\infty]$,
\begin{align}\label{prelimitfinD}
\frac{1}{t}(P_tf(x)-f(x))&=\int \frac{1}{t}(P_t\ee^y(x)-\ee^y(x))\nu(\dd y).
\end{align}
Since the limit in \eqref{bplimit} is pointwise and bounded uniformly with respect to the variable $y$, by bounded convergence the quantity \eqref{prelimitfinD} converges  towards
$\int\XX \ee^y(x)\nu(\dd y)$ as $t\downarrow 0$, for all $x\in [0,\infty]$. Furthermore, $\big([0,\infty]\ni x\mapsto \int \XX \ee^y(x)\nu(\dd y)\big)$ is bounded and continuous (by bounded convergence again). This ensures that $\mathsf{cm}^{\pm}_{[0,\infty]}\subset D_{\XX_{\mathrm{s}}}$.  The fact that $\mathsf{cm}^{\pm}_{[0,\infty]}$ is a core for $\XX_{\mathrm{s}}$ follows now from Proposition \ref{prop:cores}\ref{prop:cores:iv}.
\end{proof}

Our next result  links excessive (resp. invariant) measures and excessive (resp. invariant) functions of processes in Laplace duality. Recall that for a Markov process $X$ with values in a measurable space $(E,\mathcal{E})$ and for $\theta\in [0,\infty)$ we call $\theta$-excessive (resp. $\theta$-invariant) for $X$ any  $\sigma$-finite measure $\mu$ on $(E,\mathcal{E})$ such that for all $t\in [0,\infty)$, $\mu P_t\leq e^{\theta t}\mu$ (resp. $\mu P_t= e^{\theta t}\mu$), i.e. for all measurable $\varphi:E\to [0,\infty]$,  $\int P_t\varphi(x)\mu(\dd x)\leq e^{\theta t}\int\varphi(x)\mu(\dd x)$ (resp. $\int P_t\varphi(x)\mu(\dd x)=e^{\theta t} \int\varphi(x)\mu(\dd x)$). 
In a similar fashion, a $\theta$-excessive (resp. $\theta$-invariant) function  for $X$ is any measurable map $f:E\to [0,\infty]$ such that for all $t\in [0,\infty)$ we have $P_tf\leq e^{\theta t}f$ (resp. $P_tf=e^{\theta t}f$).  Excessive (resp. invariant) functions give rise to  supermartingales (resp. martingales), for  if $f$ is $\theta$-excessive (resp. $\theta$-invariant), then for $s\leq t$ from $[0,\infty)$, by the Markov property, $\EE_x[e^{-\theta t}f(X_t)\vert \sigma(X_v:v\in [0,s])]=e^{-\theta t}(P_{t-s}f)(X_s)\leq e^{-\theta s}f(X_s)$ (resp. with equality) a.s.-$\PP_x$ for all $x\in E$, entailing that $(e^{-\theta t}f(X_t))_{t\in [0,\infty)}$ is a supermartingale (resp. martingale)\footnote{We allow it to take on the value $\infty$ and/or to  be non-integrable.\label{foonote:allow}}. They therefore lend themselves well in applications, for instance in solving first-passage problems via optional sampling (see e.g. \cite[Theorem~2.8]{foucartvidmar}) or in studying  asymptotic behaviour through convergence theorems (see e.g. \cite{zbMATH07544388}).
\begin{proposition}\label{thm:excessivecm} Let $X$ be in $(\star_1,\star_2)$-Laplace duality  with $Y$ and $\theta\in [0,\infty)$. If $\mu$ is a  $\theta$-excessive (resp. $\theta$-invariant) measure for $Y$, then the functions $f^\downarrow,f^\uparrow$ defined on $[0,\infty]$ and given by (with $\star_1$, $\star_2$ in force)
\[ f^\downarrow(x):=\int e^{-xy} \mu(\dd y) \text{ and } f^\uparrow(x):=\int(1-e^{-x y}) \mu(\dd y),\quad x\in [0,\infty],\]
are respectively $\downarrow$ and $\uparrow$, $\theta$-excessive (resp. $\theta$-invariant) functions for $X$.
\end{proposition}
For $(0^+\cdot \infty, \infty^{-}\cdot 0)$-Laplace duals,  notice that $f^\uparrow$ does not depend on the value of $\mu(\{0\})$, while $f^\downarrow$ does not depend on $\mu(\{\infty\})$. This dovetails with the fact that in this case the boundaries  $0$ and $\infty$ are absorbing for $Y$  so that the atoms of the measure $\mu$ at $0$ and $\infty$ have no consequence for whether or not $\mu$ is $\theta$-excessive [resp. $\theta$-invariant when $\theta=0$] for $Y$ (automatically $(\mu Q_t)(\{0\})=\mu(\{0\})\leq e^{\theta t}\mu(\{0\})$ and $(\mu Q_t)(\{\infty\})=\mu(\{\infty\})\leq e^{\theta t}\mu(\{\infty\})$ [resp. with equalities when $\theta=0$], while $(\mu Q_t)\vert_{(0,\infty)}=(\mu(\cdot \cap (0,\infty))Q_t)\vert_{(0,\infty)}$ for all $t\in[0,\infty)$). We could make similar comments for the other types of dualities.
    
    A result in the same vein is known for processes in Siegmund duality, see Foucart and M\"ohle \cite[Theorem 4.1]{zbMATH07544388}. We refer also to Cox and R\"osler \cite[Theorem~1]{MR724061} for the link between exit and entrance laws of processes in duality. 
\begin{proof}[Proof of Proposition~\ref{thm:excessivecm}]
 Let $x\in [0,\infty]$ and $t\in [ 0,\infty)$. By Tonelli, the Laplace duality relationship and excessivity (resp. invariance) of  $\mu$ we have
\begin{align*}
    \mathbb{E}_x[f^\uparrow(X_t)]&=\int\mathbb{E}_x[1-e^{-X_t y}] \mu(\dd y)=\int\mathbb{E}^y[1-e^{-x Y_t}] \mu(\dd y)\\
&\leq e^{\theta t}\int (1-e^{-xy}) \mu(\dd y)=e^{\theta t}f^\uparrow(x)
    \end{align*}
  (resp. with equality in the invariant case).  The proof for $f^\downarrow$ is similar.
\end{proof}

We close this (sub)section by putting in evidence a correspondence between the limiting laws of two processes in Laplace duality. 
See e.g. \cite[Section 5]{zbMATH01396203} for a treatment in the context of population genetics, and  \cite{zbMATH03573037} for an early contribution in this direction.
\begin{proposition}\label{prop:LToflimitingdistribution}
    Let $X$ be in Laplace duality (of whichever kind) with $Y$. The following are equivalent: 
    \begin{enumerate}[(a)]
    \item\label{prop:LToflimitingdistribution.a} the weak limit $\nu_x:=\lim_{t\to \infty}{X_t}_\star\PP_x $ exists (as a probability on $[0,\infty]$) and does not depend on $x\in (0,\infty)$;
    \item\label{prop:LToflimitingdistribution.b} for all $y\in (0,\infty)$ the weak limit $\gamma^y:=\lim_{t\to\infty}{Y_t}_\star\PP^y$ exists (as a probability on $[0,\infty]$) and is  carried by $\{0,\infty\}$.
    \end{enumerate}
When \ref{prop:LToflimitingdistribution.a}-\ref{prop:LToflimitingdistribution.b} are met, then the Laplace transform of the common $\nu:=\nu_x$  (any $x\in (0,\infty)$) is given by
    \begin{equation*}\widehat{\nu}(y)=\gamma^y(\{0\}),\quad y\in (0,\infty).
    \end{equation*}
\end{proposition}
Thus, given \ref{prop:LToflimitingdistribution.a}-\ref{prop:LToflimitingdistribution.b}, the map  $((0,\infty)\ni y\mapsto \gamma^y(\{0\}))$ is automatically completely monotone and if further $\widehat{\nu}$ degenerates to the constant function $0$ or $1$ then the limiting law $\nu$ is equal to $\delta_\infty$ or $\delta_0$, respectively.
\begin{proof}[Proof of Proposition~\ref{prop:LToflimitingdistribution}]
Since $[0,\infty]$ is compact, Prokhorov's theorem and the fact that Laplace transforms determine probabilities on $\BB_{[0,\infty]}$ -- or, more directly, the continuity theorem for Laplace transforms on the halfline \cite[p.~431, Section~XIII.1, Theorem~2]{feller} -- coupled with  the duality relation \eqref{LD-main},  ensure that $\nu_x$ exists for all $x\in (0,\infty)$ if and only if $\gamma^y$ exists for all $y\in (0,\infty)$, in which case $$\widehat {\nu_x}(y)=\widehat{\gamma^y}(x),\quad \{x,y\}\subset (0,\infty).$$ This readily entails the desired conclusions.
\end{proof}

\section{Laplace duality: generators}\label{sec:LDgen}

\subsection{Courr\`ege form and Laplace symbol}\label{sec:courrage}
We describe here the form of the generator of a positive Markov process for which the only a priori knowledge is that its domain includes the exponential functions. 
\begin{definition}\label{definition:laplace-symbol}
\begin{enumerate}[(i)]
\item\label{definition:laplace-symbol:i} Let $x\in [0,\infty)$. Given an $x$-L\'evy quadruplet $l=(\nu,a,b,c)$, we introduce the Laplace symbol $\psi_l:(0,\infty)\to \mathbb{R}$ associated to  $l$ as $$\psi_l(y):=\int  \left(e^{-uy}-1+uy\mathbbm{1}_{[-1,1]}(u)\right)\nu(\dd u)+ay^2-by-c,\quad y\in (0,\infty).$$
\item\label{definition:laplace-symbol:ii} Given a L\'evy triplet $l=(\nu,d,c)$, we introduce the Laplace symbol $\psi_l:(0,\infty)\to \mathbb{R}$ associated to  $l$ as $$\psi_l(y):=\int  \left(e^{-uy}-1\right)\nu(\dd u)-dy-c,\quad y\in (0,\infty).$$
\item\label{definition:laplace-symbol:iii} Given a pair $l=(\nu,k)$ consisting of a measure $\nu$ on  $\BB_{(0,\infty)}$ with finite Laplace transform and $k\in [0,\infty)$ (henceforth, a L\'evy pair), we introduce the Laplace symbol $\psi_l:(0,\infty)\to \mathbb{R}$ associated to  $l$ as $$\psi_l(y):=\int  e^{-uy}\nu(\dd u)+k,\quad y\in (0,\infty).$$
\end{enumerate}
A family $l=(l_x)_{x\in [0,\infty]}$ is said to be \emph{L\'evy} if it consists of an $x$-L\'evy quadruplet  for each $x\in (0,\infty)$, a L\'evy triplet $l_0$, a L\'evy pair $l_\infty$, and has $(0,\infty)\ni x\mapsto l_x$  measurable. For such a family we introduce $$\psi_{l}(x,y):=
\psi_{l_x}(y),\quad (x,y)\in [0,\infty]\times (0,\infty).$$
A map of the form $\psi_{l}$ for some L\'evy family $l$ is called a Laplace symbol.
\end{definition}
The $\psi_l$ (resp. the negatives $-\psi_l$) of \ref{definition:laplace-symbol:i} (resp. \ref{definition:laplace-symbol:ii}) constitute precisely $\pkLp_x\vert_{(0,\infty)}$ (resp. $\pksubo\vert_{(0,\infty)}$), and we may recall that $\pkspLp=\pkLp_0$ (see Table~\ref{table.3}).  In \ref{definition:laplace-symbol:i}-\ref{definition:laplace-symbol:ii} (resp. \ref{definition:laplace-symbol:iii}) we could have 
incorporated $c$ (resp. $k$) into $\nu$ as a mass at infinity (resp. zero), but we prefer to keep it separate. The class of Laplace symbols of \ref{definition:laplace-symbol:i}, \ref{definition:laplace-symbol:ii}  and \ref{definition:laplace-symbol:iii} is each (on its own) a convex cone closed for pointwise convergence, which follows respectively from Proposition~\ref{proposition:closure-under-limits}, from the remark concerning $\pksubo\vert_{(0,\infty)}$ just after said proposition and from the fact that the family of completely monotone functions is also closed for pointwise limits \cite[Corollary~1.6]{bernstein}.
\begin{theorem}[Courrège form]\label{thm:courrege}
Assume $X$ is a positive Markov process  and that $\{\ee^y:y\in (0,\infty)\}\subset D_\XX$. Then there exists  a unique L\'evy family $l$
such that 
\begin{equation}\label{eq:courrege}
(\XX\ee^y)(x)=
\begin{cases}
\psi_{l}(x,y)e^{-xy},&x\in [0,\infty),\\
\psi_l(x,y),& x=\infty,
\end{cases}\quad y\in (0,\infty).
\end{equation}
\end{theorem}
It is remarkable that  \eqref{eq:courrege} is automatically continuous in $y$.  
\begin{definition}
Under the assumptions, and in the notation of Theorem~\ref{thm:courrege}, we introduce $\psi_X:=\psi_{l}$, the Laplace symbol of $X$.
\end{definition}
The proof of Theorem~\ref{thm:courrege} is relegated  to Appendix~\ref{appendix}, it stands on its own, but several comments are in order.  Setting  
\begin{equation}\label{the-exponential-domain}
\mathcal{D}:=\mathrm{lin}(\{\ee^y:y\in (0,\infty)\})\subset \mathsf{C}([0,\infty])
 \end{equation} (a piece of notation to remain in force throughout the remainder of this paper), then, writing:
\begin{align*}
l_0&=:(\nu_0,d_0,c_0);\,\,
l_x=:(\nu_x,a_x,b_x,c_x),\ x\in (0,\infty);\,\,
l_\infty=:(\nu_\infty,k_\infty);
\end{align*}
 we may expand \eqref{eq:courrege} as follows. For $f\in \mathcal{D}$:
\begin{equation}
(\XX f)(0)=\int\left(f(u)-f(0)\right)\nu_0(\dd u)+d_0f'(0)-c_0f(0);\label{laplace-representation-at-zero}\end{equation}
\begin{equation} (\XX f)(x)=\int\left(f(x+u)-f(x)-uf'(x)\mathbbm{1}_{[-1,1]}(u)\right)\nu_x(\dd u)+a_xf''(x)+b_xf'(x)-c_xf(x)\label{compare-courrage}\end{equation} for $x\in (0,\infty)$; and 
\begin{equation}(\XX f)(\infty)=\int (f(u)-f(\infty))\nu_\infty(\dd u)+k_\infty \big(f(0)-f(\infty)\big).\label{courrage-inf}\end{equation}

Part \eqref{compare-courrage} corresponds most directly to the so-called Courr\`ege form that is usually established for a real-valued (or Euclidean space-valued) Feller process under the assumption that the set of infinitely differentiable functions with compact support, $C^\infty_c(\mathbb{R})$, is included in the domain of its strong generator, see e.g. B\"ottcher et al. \cite[Theorem~2.21]{zbMATH06256582} and the references therein. The usual statement gives the action of the generator on the set of functions $C^{\infty}_c(\mathbb{R})$ or on the Schwartz space $\mathcal{S}(\mathbb{R})$.  Our approach to the proof of Theorem~\ref{thm:courrege} will use up only elementary arguments.

    An example of processes with a non-degenerate return law from $\infty$, i.e. $l_\infty\neq 0$, can be found in Pakes~\cite{zbMATH00496250} who studied  CB processes jumping back into $(0,\infty)$ after explosion. In our subsequent study of Laplace duality we shall however only encounter  cases for which  $\nu_\infty=0=k_\infty$.
    
    Starting from a sufficiently well-behaved $[0,\infty]$-valued c\`adl\`ag Markov process $X$ with Laplace symbol $\psi_X$, the quantity $\nu_x(A)$, for a given $x \in [0,\infty)$ and a Borel set $A\in \BB_{[-x,\infty)\backslash\{0\}}$, encodes the jump intensity of $X$ from the state $x$ to the set $x + A$. When $x \in (0,\infty)$, an $L^2$-compensation of small jumps is possible, and this is analytically represented by the first-order correction term $-u f'(x)$ in \eqref{compare-courrage}. We refer the reader, for instance, to Benveniste and Jacod~\cite{zbMATH03418463} for the theory of L\'evy systems, which elucidates the relationship between the measures $\nu_x$ and the jump behavior of the process. However, we shall not rely on this theory in the sequel. The structure of jumps will anyway be transparent when working with examples via the help of stochastic equations in Section~\ref{sec:examples}.

 The next definition is modeled directly on Theorem~\ref{thm:courrege} (cf. also \eqref{laplace-representation-at-zero}-\eqref{courrage-inf}): therein, in the terminology to be introduced presently, $\XX\vert_\mathcal{D}$ is the pregenerator associated to $\psi_X$.

\begin{definition}\label{def:pregen-1}
Let $\psi$ be a Laplace symbol. Denote by $l$ the unique L\'evy family for which $\psi=\psi_l$, writing $l_0=:(\nu_0,d_0,x_0)$, $l_\infty=:(\nu_\infty,k_\infty)$ and $l_x=:(\nu_x,a_x,b_x,c_x)$ for $x\in (0,\infty)$.  We introduce the linear operator $\XX_\psi:=(\mathcal{D}\ni f\mapsto \XX_\psi f)$ by insisting that for $f\in \DD$ the value $(\XX_\psi f)(x)$ is given at $x=0$, $x\in (0,\infty)$ and $x=\infty$ respectively by the r.h.s. of \eqref{laplace-representation-at-zero}, \eqref{compare-courrage} and \eqref{courrage-inf}. We call $\XX_\psi$ the pregenerator associated to $\psi$ and we say that:
\begin{itemize}
\item  it is vanishing at infinity when $\nu_\infty=0=k_\infty$  (i.e. $\psi(\infty,\cdot)\equiv 0$); 
\item it is spectrally positive if it is vanishing at infinity and  $\nu_x$ is carried by $(0,\infty)$ for all $x\in (0,\infty)$; 
\item it is continuous at zero if  $\XX_\psi f$ is continuous at zero for all $f\in \DD$ (i.e. $\psi$ is continuous at zero in its first coordinate);
\item it has no killing when $c_x=0$ for all $x\in [0,\infty)$; 
\item it is pure killing when $\nu_x=0$ for all $x\in [0,\infty]$, $k_\infty=0$, $d_0=0$ and $a_x=0=b_x$ for all $x\in [0,\infty)$;
\item it is local if it has $\nu_x=0$ for all $x\in [0,\infty]$ and $k_\infty=0$. 
\end{itemize}
Similarly, we define the extended pregenerator $\XX_{e,\psi}$ associated to $\psi$ by specifying its domain $D_{\XX_{e,\psi}}$ to consist of those $f: [0,\infty]\to \mathbb{R}$ for which the r.h.s. of \eqref{laplace-representation-at-zero}-\eqref{courrage-inf}  are well-defined and finite and asking that for such $f$, $(\XX_{e,\psi} f)(x)$  is given at $x=0$, $x\in (0,\infty)$ and $x=\infty$ respectively by the r.h.s. of \eqref{laplace-representation-at-zero}, \eqref{compare-courrage} and \eqref{courrage-inf}.   A (resp. vanishing at infinity, spectrally positive, etc.) pregenerator means simply a (resp. vanishing at infinity, spectrally positive, etc.) pregenerator associated to some Laplace symbol.
\end{definition}
So, in the context of Definition~\ref{def:pregen-1}, cf. \eqref{eq:courrege}, 
\begin{equation}\label{eq:gen-and-symbol}
(\XX_{e,\psi}\ee^y)(x)=(\XX_\psi\ee^y)(x)=
\begin{cases}
\psi(x,y)e^{-xy},&x\in [0,\infty),\\
\psi(x,y),& x=\infty
\end{cases},\quad y\in (0,\infty).
\end{equation}
In particular, a given pregenerator is associated to a unique Laplace symbol. Hence, thanks to Proposition~\ref{lemma:laplace-symbol-unique}, one may shift freely between L\'evy families, L\'evy symbols and pregenerators (between $l$, $\psi_l$ and $\XX_{\psi_l}$, respectively), whichever is the more convenient. 

\begin{remark}\label{rmk:spec-pos-pregen}
For a spectrally positive pregenerator  $\XX$ we have that  $((0,\infty)\ni y\mapsto \frac{\XX_{\psi}\ee^y(x)}{\ee^y(x)})\in \pkspLp\vert_{(0,\infty)}$ for all $x\in (0,\infty)$.
\end{remark}

The significance of Theorem~\ref{thm:courrege} is in making sure that the action of every generator $\XX$ of a positive Markov process $X$, whose domain $D_\XX$ includes the exponentials $\DD$, is that of a pregenerator as just delineated in Definition~\ref{def:pregen-1}. It does not mean, however, at least not a priori, that any pregenerator $\XX^0$  is associated to a (unique) positive Markov process $X$ whose generator $\XX$ satisfies $\DD\subset D_\XX$ and $\XX\vert_\DD=\XX^0$. Hence the reason for calling them \emph{pre}generators. 

\subsection{Generators and Laplace symbols in duality}\label{subsec:geninLaplaceduality}
We turn to the study of generators in Laplace duality.
\begin{definition}\label{definition:gen-in-duality}
 Maps $\XX$ and $\YY$  that include in their domain  $\{\ee^y:y\in (0,\infty)\}=\{\ee_x:x\in (0,\infty)\}$ and whose outputs are real maps on $[0,\infty]$ are said to be in Laplace duality when 
 \begin{equation}\label{infini-ld}
 (\XX \ee^y)(x)=(\YY\ee_x)(y),\quad \{x,y\}\subset (0,\infty).
 \end{equation} Similarly, if  \begin{equation}\label{def:dualsymbol}
 \psi(x,y)=\hat\psi(y,x),\quad \{x,y\}\subset (0,\infty),\end{equation} then the Laplace symbols $\psi$ and $\hat\psi$ are said to be in duality.
\end{definition}
\emph{Nota bene}: unlike at the level of semigroups \eqref{LD-main}, in   \eqref{infini-ld}-\eqref{def:dualsymbol} we \emph{exclude} from consideration the boundaries $0$ and $\infty$. We take the stance that the behaviour on the full class of exponentials $\{\ee^y:y\in [0,\infty]\}$, respectively $\{\ee_x:x\in [0,\infty]\}$, of the generators $\XX$ and $\YY$ of positive Markov processes $X$ and $Y$  at these two points, which is contingent on the choice of the conventions for $0\cdot \infty$ and $\infty\cdot 0$, is best left outside of duality relations at the generator level, an early indication of which is the fact that in the Courr\`ege form of Theorem~\ref{thm:courrege} we are only able to handle the maps $\ee^y$, $y\in (0,\infty)$. In an informal sense the boundary behaviour of the processes is related to the properties  of the generators at $0$ and $\infty$, but this relationship is in general highly non-trivial and lies largely beyond the intended scope of this investigation.

Rather trivially, but significantly, generators of processes in Laplace duality are themselves in Laplace duality when their domains include the exponential maps:
\begin{theorem}\label{thm:generator-duality-semigroup}
If $X$ is in $(\star_1,\star_2)$-Laplace duality with $Y$ and if (under the conventions $ \star_1$, $\star_2$) for given $(x,y)\in [0,\infty]^2$ we have  $\ee^y\in D_\XX$ and $\ee_x\in D_\YY$, then $\XX\ee^y(x)=\YY\ee_x(y)$. Under the assumption $\{\ee_x:x\in (0,\infty)\}=\{\ee^y:y\in (0,\infty)\}\subset D_\XX\cap D_\YY$:  $\XX$ and $\YY$ are in Laplace duality if and only if  $\psi_X$ and $\psi_Y$  are; if  $X$ is in Laplace duality (of whichever kind) with $Y$, then $\XX$ and $\YY$ are in Laplace duality.
\end{theorem}
\begin{proof}
For the first claim, one has only to note that ($\because$ $\ee^y(x)=\ee_x(y)$, irrespective of which conventions are in force, and by the duality relation itself) $$\frac{(P_t\ee^y)(x)-\ee^y(x)}{t}=\frac{(Q_t\ee_x)(y)-\ee_x(y)}{t},\quad t\in (0,\infty),$$ and pass to the limit $t\downarrow 0$.  As for the second claim,  from Theorem~\ref{thm:courrege} we have that $(\XX \ee^y)(x)=\psi_X(x,y)e^{-xy}$ and $(\YY \ee_x)(y)=\psi_Y(y,x)e^{-xy}$ for $\{x,y\}\subset (0,\infty)$. The asserted equivalence follows now directly from the definitions \eqref{infini-ld}-\eqref{def:dualsymbol}. The final implication is just a specialization of the first claim.
\end{proof}
Conversely, we may start with a pair of pregenerators and attempt to establish the Laplace duality of the associated processes. This is a two step process. First, one wants to establish the Laplace duality of the pregenerators themselves. In the second step, which will be the object of pursuit of Section \ref{sec:fromdualsymboltosemigroups}, one hopes to be able to ``integrate out'' the infinitesimal relation \eqref{infini-ld} to the semigroup level \eqref{LD-main}. We turn our attention now to the first part of this programme: pregenerators in duality.

To start us off, let us endeavour to explore  the landscape of spectrally positive pregenerators in Laplace duality that are continuous at zero. Those that correspond to  time-changed spectrally positive L\'evy processes -- possibly killed -- and that are in Laplace duality with another spectrally positive pregenerator, continuous at zero, can be pinned down precisely. 
\begin{definition}\label{defLPsi}
    Let $\Psi\in \pkspLp$ with the representation \eqref{eq:LKpsi} in terms of the L\'evy quadruplet $(\nu,a,b,c)$. We write $\mathrm{L}^\Psi$ for the linear operator on $\mathcal{D}$ whose values are specified by $\mathrm{L}^\Psi \ee^y(x):=\Psi(y)e^{-xy}$ for $(x,y)\in [0,\infty]\times (0,\infty)$.  We also denote by the same symbol $\mathrm{L}^\Psi$ the operator acting on a real-valued function $f$ with domain $D_f$ a subinterval of $[0,\infty)$ unbounded above, according to \small
$$(\mathrm{L}^\Psi  f)(x):=\int\left(f(x+u)-f(x)-uf'(x)\mathbbm{1}_{(0,1]}(u)\right)\nu(\dd u)+af''(x)+bf'(x)-cf(x),\quad x\in D_f,$$ \normalsize provided the r.h.s. is well-defined and finite (such $f$ we take as constituting the domain $D_{\mathrm{L}^\Psi}$ of this $\mathrm{L}^\Psi$). 
\end{definition}
Context will determine unambiguously which of the two meanings attached to $\mathrm{L}^\Psi$ we intend, but anyway we have the identity $(\mathrm{L}^\Psi f)\vert_{[0,\infty)}=\mathrm{L}^\Psi (f\vert_{[0,\infty)})$ for all $f\in\DD$. Acting on $\DD$ the operator $\mathrm{L}^\Psi$ is (essentially) the strong  generator of a possibly killed spectrally positive L\'evy process with Laplace exponent $\Psi$, see e.g. Sato's book \cite[Theorem~31.5]{MR3185174}. 

For the next theorem, recall the various classes of L\'evy-Khintchine forms, e.g. by consulting Table~\ref{table.3}.
\begin{theorem}\label{theorem:tensor-form}
 Let $\mathcal{X}$ be a spectrally positive pregenerator, continuous at zero, satisfying
 \begin{equation}\label{volkonskii-form}
 \mathcal{X}f(x)=R(x)\mathrm{L}^{\Psi}f(x),\quad x\in [0,\infty),\,  f\in \DD,
 \end{equation}
  for some $\Psi\in \pkspLp$ and some measurable function $R:[0,\infty)\to [0,\infty)$, neither  identically zero. Then, for $\XX$ to be in Laplace duality with some spectrally positive pregenerator $\mathcal{Y}$ that is continuous at zero, it is necessary and sufficient that:
\begin{enumerate}[(i)]
\item\label{clement:i} in case $\Psi\geq 0$: $\Psi=\Sigma$ and $R=\hat{\Sigma}$ for some $\{\Sigma,\hat{\Sigma}\}\subset\spLpnot$ and 
$$\mathcal{Y}\ee_x=\Sigma\mathrm{L}^{\hat{\Sigma}}\ee_x,\quad x\in (0,\infty);$$
\item\label{clement:ii} in case $\Psi\leq 0$: $\Psi=-\Phi$ and
$R=\hat{\Phi}$ for some $\{\Phi,\hat{\Phi}\}\subset \pksubo$  and
$$\mathcal{Y}\ee_x=\Phi\mathrm{L}^{-\hat{\Phi}}\ee_x, \quad   x\in (0,\infty);$$
\item\label{clement:iii} otherwise: $R=\gamma\mathrm{id}_{[0,\infty)}$ for some $\gamma\in (0,\infty)$ and  $$\mathcal{Y}\ee_x=-\Psi\mathrm{L}^{-\gamma\mathrm{id}_{[•0,\infty)}}\ee_x,\quad  x\in (0,\infty).$$
\end{enumerate}
At the level of the Laplace symbols $\psi$ and $\hat\psi$ to which are associated $\XX$ and $\YY$ respectively, for $\{x,y\}\subset (0,\infty)$: in case \ref{clement:i}, $\psi(x,y)=\hat{\Sigma}(x)\Sigma(y)$; in case \ref{clement:ii}, $\psi(x,y)=-\hat{\Phi}(x)\Phi(y)$; and in the last case \ref{clement:iii}, $\psi(x,y)=\gamma x\Psi(y)$.
\end{theorem}
By Volkonskii's theorem \cite[Chapter 6]{EthierKurtz} the generators of time-changed possibly killed spectrally positive L\'evy process  take the  form \eqref{volkonskii-form}.

The pregenerator $\XX$ of Theorem~\ref{theorem:tensor-form} with $R=\gamma\mathrm{id}_{[0,\infty)}$ -- from \ref{clement:iii} -- is the sum of  one from \ref{clement:i} and another one from \ref{clement:ii}  on taking $\hat\Sigma=\gamma\mathrm{id}_{[0,\infty)}=\hat\Phi$ and $\Psi=\Sigma-\Phi$ therein: $$\gamma x\mathrm{L}^\Psi f(x)=\hat\Sigma(x)\mathrm{L}^\Sigma f(x)+\hat\Phi(x)\mathrm{L}^{-\Phi} f(x),\quad x\in [0,\infty),\, f\in \DD$$ (such a decomposition of $\Psi$ into the difference $\Sigma-\Phi$ can always be effected, see Remark~\ref{rem:decompositionjumpmeasure}). These are nothing but the generators of CB processes (acting on $\DD$).
\begin{proof}[Proof of Theorem~\ref{theorem:tensor-form}]
The continuity at zero of $\XX$ and $\Psi\ne 0$ entail that $R$ is continuous at zero. Due to \eqref{volkonskii-form} and the definitions (namely,~\ref{definition:gen-in-duality} and~\ref{defLPsi}) the Laplace duality of $\XX$ and $\YY$ is equivalent to  $$R(x)\Psi(y)\ee_x(y)=\YY \ee_x(y),\quad (x,y)\in (0,\infty)^2.$$ The conditions \ref{clement:i}-\ref{clement:iii} are clearly sufficient. Let us prove necessity. By continuity at zero it suffices to determine the forms $R$ and $\Psi$ on $(0,\infty)$.

\ref{clement:i}. Assume $\Psi\geq 0$. Let $y_0\in (0,\infty)$ be arbitrary such that $\Psi(y_0)>0$. As the pregenerator $\YY$ is spectrally positive, by Remark~\ref{rmk:spec-pos-pregen} the map $$ (0,\infty)\ni x\mapsto \frac{\YY \ee_{x}(y_0)}{\ee_{x}(y_0)}=\Psi(y_0)R(x)$$ is the restriction to $(0,\infty)$  of an element of $\pkspLp$. But $\pkspLp$ is closed for multiplication with $\Psi(y_0)^{-1}$ (indeed $\pkspLp$ is a convex cone), hence $R\vert_{(0,\infty)}\in \pkspLp\vert_{(0,\infty)}$. By continuity at zero  of $R$ and of the members of $\pkspLp$ we deduce that $R\in \pkspLp$. The maps $R$ and (by assumption) $\Psi$ thus both belong to $\pkspLp$;  being nonnegative they must in fact be members of $\spLpnot$.

\ref{clement:ii}. Assume $\Psi\leq 0$. Picking $y_0\in (0,\infty)$ arbitrary such that $\Psi(y_0)<0$ we make an analogous argument, but use the fact that $\pkspLp\cap (-\infty,0]^{[0,\infty)}=-\pksubo$.

\ref{clement:iii}. As in the previous two items we find that $R\in\spLpnot\cap \pksubo=\{\gamma\mathrm{id}_{[0,\infty)}:\gamma\in [0,\infty)\}$.
\end{proof}
We state next a theorem characterizing the spectrally positive pregenerators that are in Laplace duality with local pregenerators, both continuous at zero. 

\begin{theorem}\label{thm:localgeninduality} Let $\mathcal{X}$ be a spectrally positive pregenerator, continuous at zero. It admits a spectrally positive Laplace dual pregenerator $\mathcal{Y}$ that is local, has no killing and is continuous at zero, if and only if 
\begin{equation}\label{condition-for-diffusion}
\mathcal{X}f(x)=x^2\mathrm{L}^{\Sigma}f(x)+x\mathrm{L}^{\Psi}f(x), \quad \ x\in [0,\infty),\, f\in \DD,
\end{equation}
for some $\Psi\in \pkspLp$ and some $\Sigma\in \spLpnot$. When so, then
\begin{equation}\label{condition-for-diffusion-candidate}
\mathcal{Y}f(y)=\Sigma(y)f''(y)-\Psi(y)f'(y), \quad y\in [0,\infty),\, f\in \DD,
\end{equation}
and at the level of the Laplace symbols $\psi$ and $\hat\psi$, to which are associated $\XX$ and $\YY$ respectively, one has
\begin{equation}\label{condition-for-diffusion-symbol}
\psi(x,y)=x^2\Sigma(y)+x\Psi(y)=\hat \psi (y,x), \quad \{x,y\}\subset (0,\infty).
\end{equation}
\end{theorem}
    A version of this result was first observed in \cite[Theorem 2.1]{foucartvidmar}. The associated Markov processes were were baptized in \cite{foucartvidmar} as continuous-state branching processes with \textit{collisions}.  We return to them in Paragraph~\ref{subsec:CBC-CBCI}. 
\begin{proof}[Proof of Theorem~\ref{thm:localgeninduality}]
If \eqref{condition-for-diffusion} is verified, then taking for $\YY$ the spectrally positive pregenerator specified by \eqref{condition-for-diffusion-candidate} -- which is indeed local, has no killing and is continuuos at zero -- we may compute
$$\XX \ee_x(y)=(x^2\Sigma(y)+x\Psi(y))e^{-xy}=\YY \ee^y(x),\quad \{x,y\}\subset (0,\infty),$$
 so that $\XX$ and $\YY$ are in Laplace duality  and we have also checked  \eqref{condition-for-diffusion-symbol}.

Let us now argue that \eqref{condition-for-diffusion} is also necessary. By the duality assumption,
\[\frac{\mathcal{Y}\ee_x(y)}{\ee_x(y)}=\frac{\mathcal{X}\ee^y(x)}{\ee^y(x)},\quad \{x,y\}\subset (0,\infty).\]
By Definitions~\ref{definition:laplace-symbol} and~\ref{def:pregen-1}, for each $x\in (0,\infty)$ there is a $\Psi_x\in \pkspLp$ such that  $$\mathcal{X}\ee^y(x)/\ee^y(x)=\Psi_x(y),\quad y\in (0,\infty).$$
Denote by $\hat{a}:[0,\infty)\to [0,\infty)$ and $\hat{b}:[0,\infty)\to \mathbb{R}$ respectively the diffusion and drift parameters of $\mathcal{Y}$ (for points of issue in $[0,\infty)$) so that \begin{equation*}
\frac{\mathcal{Y}\ee_x(y)}{\ee_x(y)}=\hat{a}_yx^2-\hat{b}_yx,\quad x\in (0,\infty),\, y\in [0,\infty).
\end{equation*}
Combining the three preceding displays, for each $y\in (0,\infty)$ and then also for $y=0$ by the continuity of $\YY$ ($\therefore$ $\hat a$, $\hat b$) and of $\Psi_x$ at zero, 
\begin{equation*}
\hat{a}_yx-\hat{b}_y=\frac{\Psi_x(y)}{x},\quad x\in (0,\infty);
\end{equation*}
the limit of this expression as $x\downarrow 0$ exists and equals $-\hat{b}_y$. Proposition~\ref{proposition:closure-under-limits} now entails that  $-\hat{b}\in \pkspLp$. Dividing the preceding display (again) by $x$, and letting $x\to\infty$ we deduce by the same token that (the nonnegative) $ \hat{a}$ belongs to $\spLpnot$.
\end{proof}
The next proposition sheds light on the Laplace duality at the generators' level between a subordinator and a pure killing process, cf. Example \ref{example:subo}.   
\begin{proposition} \label{proposition:killing-subo}Let $\mathcal{X}$ be a spectrally positive pregenerator, continuous at zero. It admits a Laplace dual spectrally positive pregenerator $\mathcal{Y}$ that is a pure killing operator and is continuous at zero, if and only if 
\begin{equation*}
    \mathcal{X}f(x)=\mathrm{L}^{-\Phi}f(x), \quad x\in [0,\infty),\, f\in \mathcal{D},
\end{equation*}
for some $\Phi\in \pksubo$. When so, then
    \[\mathcal{Y}f(y)=-\Phi(y)f(y), \quad y\in [0,\infty),\, f\in \DD,\]
    and at the level of the Laplace symbols $\psi$ and $\hat\psi$, to which are associated $\XX$ and $\YY$ respectively, one has
    \[\psi(x,y)=-\Phi(y)=\hat\psi(y,x), \quad \{x,y\}\subset (0,\infty).\]
\end{proposition}
\begin{proof}
As in Theorems~\ref{theorem:tensor-form}-\ref{thm:localgeninduality} the condition is seen to be sufficient by direct inspection. We turn to necessity. Verbatim as in the proof of Theorem~\ref{thm:localgeninduality} we find that 
\[\frac{\mathcal{Y}\ee_x(y)}{\ee_x(y)}=\frac{\mathcal{X}\ee^y(x)}{\ee^y(x)},\quad \{x,y\}\subset (0,\infty);\]
also that for each $x\in (0,\infty)$ there is a $\Psi_x\in \pkspLp$ such that  $$\mathcal{X}\ee^y(x)/\ee^y(x)=\Psi_x(y),\quad y\in (0,\infty),$$
while $$\mathcal{X}\ee^y(0)=-\Phi(y),\quad y\in (0,\infty),$$ for some $\Phi\in\pksubo$.
Denoting by $\hat{c}:[0,\infty)\to [0,\infty)$ the killing parameters of $\mathcal{Y}$ for points of issue in $[0,\infty)$ we identify \begin{equation*}
\frac{\mathcal{Y}\ee_x(y)}{\ee_x(y)}=-\hat{c}_y,\quad x\in (0,\infty),\, y\in[0,\infty).
\end{equation*}
Thus, for each $y\in (0,\infty)$, 
\begin{equation*}
-\hat{c}_y=\Psi_x(y),\quad x\in (0,\infty).
\end{equation*}
The l.h.s. of the preceding relation does not depend on $x$, therefore neither does the r.h.s. and we deduce that there is a single $\Psi\in \pkspLp$ for which $-\hat{c}=\Psi$ on $(0,\infty)$. Continuity at zero of $\XX$ requires $\Psi=-\Phi$, so that $\hat c=\Phi$ on $(0,\infty)$. Continuity of $\YY$ at zero extends the latter equality to the point zero, completing the proof.
\end{proof}
So far we only have encountered spectrally positive pregenerators. In order to widen our body of available examples we provide now a class pregenerators incorporating negative jumps and satisfying Laplace duality. 
\begin{definition}\label{eq:multiplicativelevy}
For a $\kappa\in \pkLp_1$  with the representation \eqref{eq:exponents:pkLp} in terms of a $1$-L\'evy quadruplet  $(\nu, a, b, c)$, define $\mathrm{L}_\kappa f:[0,\infty)\to \mathbb{R}$ as \small
\begin{equation*}
\mathrm{L}_\kappa f(x):=\int\big(f(x(1+m))-f(x)-xmf'(x)\mathbbm{1}_{[-1,1]}(m)\big)\nu(\dd m)+ax^2f''(x)+{b}xf'(x)- c f(x)
\end{equation*}\normalsize
for those  $f:[0,\infty)\to \mathbb{R}$ for which the right-hand side  is well-defined and finite for all $x\in [0,\infty)$ (such $f$ we take as constituting the domain $D_{\mathrm{L}_\kappa}$ of this $\mathrm{L}_\kappa$); also, by an abuse of notation, $\mathrm{L}_\kappa f:=\mathbbm{1}_{[0,\infty)}\mathrm{L}_\kappa (f\vert_{[0,\infty)})$ for $f\in \DD$, giving us the map $\mathrm{L}_\kappa$ defined on $\DD$, returning functions mapping $[0,\infty]\to\mathbb{R}$. We allow context to determine the distinction between the two meanings that we attach to $\mathrm{L}_\kappa$ (cf. Definition~\ref{defLPsi}).
\end{definition}
The next statement follows by direct inspection.
\begin{proposition}\label{proposition:mult-duality}
Let $\kappa\in \pkLp_1$. The operator $\XX:=(\mathrm{L}_\kappa\text{ on }\DD)$ is a continuous-at-zero vanishing-at-infinity pregenerator in Laplace duality with itself and associated to the  Laplace symbol $\psi$ that has
\begin{equation*}
\psi(x,y)=\kappa(xy),\quad (x,y)\subset (0,\infty). \tag*{\qed}
\end{equation*}
 \end{proposition}
We shall see in Paragraph~\ref{paragraph:CBRE} that the pregenerators of Proposition~\ref{proposition:mult-duality} are related to multiplicative L\'evy processes. 

The  concept to be introduced presently should appear quite natural at this point.
\begin{definition}\label{def:laplacesymboldualfamily}
Let  $\psi:[0,\infty]\times [0,\infty]\to \mathbb{R}$. We define  $$\hat\psi(x,y):=\psi(y,x),\quad (x,y)\in [0,\infty]^2,$$ and refer to $\hat\psi:[0,\infty]\times [0,\infty]\to \mathbb{R}$ as  the dual of $\psi$. When  $\psi\vert_{[0,\infty]\times (0,\infty)}$ and $\hat\psi\vert_{[0,\infty]\times (0,\infty)}$ are both  Laplace symbols (whence in Laplace duality) $\psi $ is called a Laplace dual symbol. We say $\psi$: vanishes at infinity if $\psi=0$ on $(\{\infty\}\times [0,\infty])\cup ([0,\infty]\times \{\infty\})$;  is continuous at zero if it is separately continuous at zero in each of the two variables.
\end{definition}
In principle the notion of a Laplace dual symbol involves some redundancy -- in general, in keeping with what we have said following Definition~\ref{definition:gen-in-duality}, it would be most naturally defined on $([0,\infty]\times (0,\infty))\cup ((0,\infty)\times [0,\infty])$ -- however we yield to simplicity. As we will see, for the theory that we have in mind, as well as for the examples that we will have occasion to encounter, it is actually quite useful, and certainly no impediment, to have $\psi$ defined on the whole of $[0,\infty]^2$: if a map $\psi$ is given initially on $([0,\infty]\times (0,\infty))\cup ((0,\infty)\times [0,\infty])$, has a limit at $(0,0)$, is such that both $\psi$ and $\hat \psi$ are continuous at zero and vanish at infinity,  and $\psi=\hat\psi$ on $(0,\infty)^2$ (in the clear meaning of these qualifications and pieces of notation), then it is natural to extend $\psi$ to $[0,\infty]^2$ by insisting on the vanishing at infinity and continuity at zero properties, such extension being unique and yielding a Laplace dual symbol. In this way, the basic instances of pregenerators in Laplace duality witnessed above correspond indeed all to Laplace dual symbols, continuous at zero and vanishing at infinity, as summarized in Table~\ref{table:fund-gen-dualities}. It appears this is the most parsimonious way of encoding the structure of pregenerators in Laplace duality. 
\begin{table}[!hbt]
\begin{center}
\begin{tabular}{cc}
$\psi(x,y)$ for $(x,y)\in [0,\infty)^2$ & References\\\hline
$\hat \Sigma(x)\Sigma(y)$ & Theorem~\ref{theorem:tensor-form}\ref{clement:i} and Example~\ref{exampleSigmaSigma}\\
$-\hat\Phi(x)\Phi(y)$ & Theorem~\ref{theorem:tensor-form}\ref{clement:ii} and Example~\ref{example:dual-subo}\\
$ x\Psi(y)$ & Theorems~\ref{theorem:tensor-form}\ref{clement:iii} and~\ref{thm.cbs-classical}\\
$\kappa(xy)$ & Proposition~\ref{proposition:mult-duality} and Paragraph~\ref{paragraph:CBRE}
\end{tabular}
\end{center}
     \caption{Fundamental Laplace dual symbols $\psi$ as $\Psi\in\pkspLp$, $\{\Sigma, \hat\Sigma\}\subset \spLpnot$, $\{\Phi, \hat \Phi\}\subset\pksubo$, $\kappa\in \pkLp_1$ range in their respective families. We insist that these $\psi$ vanish at infinity. The connection to the notation $\XX$ and $\YY$ of the references is that these are the pregenerators associated to $\psi\vert_{[0,\infty]\times (0,\infty)}$ and $\hat\psi\vert_{[0,\infty]\times (0,\infty)}$, respectively. Except for the last line, all the classes correspond to spectrally positive pregenerators. The first, second and last family are  in  ``self-duality'' (passing from $\psi$ to $\hat\psi$ we remain in the class).}\label{table:fund-gen-dualities}
 \end{table}

The class of Laplace dual symbols is a convex cone closed under pointwise limits. If $\psi$ is a Laplace dual symbol, then the (extended) pregenerators $\XX$ and $\YY$ associated to $\psi\vert_{[0,\infty]\times (0,\infty)}$ and $\hat\psi\vert_{[0,\infty]\times (0,\infty)}$ respectively are in Laplace duality. 

\subsection{A wide class of   Laplace dual symbols}\label{sec:wideclass}
As a last order of business in this section we provide a large family of Laplace dual symbols, vanishing at infinity and continuous at zero.

A subclass of the one that we will ultimately want to consider is given by 

\begin{proposition}\label{prop:nondecomposable-SpCMP} 
Let  $\bm{\nu}$ and $\bm{\eta}$ be two measures defined respectively on $\BB_{(0,\infty]^{2}}$ and $\BB_{(0,\infty)^{2}}$ such that 
\begin{equation}\label{eq:nondecomposable-SpCMP-}
\int(1\wedge v)(1\wedge u)\bm{\nu}(\dd v, \dd u)<\infty  \text{ and } \int(v\wedge v^2)(u\wedge u^2)\bm{\eta}(\dd v, \dd u)<\infty.
\end{equation}  For $(x,y)\in  [0,\infty)^2$ set
\begin{align}
\mathbf{\Sigma}(x,y)&:=\int (e^{-xv}-1+xv)(e^{-uy}-1+uy)\bm{\eta}(\dd v, \dd u) \text{ and }\label{symbolxySigma} \\ 
\mathbf{\Phi}(x,y)&:=\int(1-e^{-xv}\mathbbm{1}_{(0,\infty)}(v))(1-e^{-uy}\mathbbm{1}_{(0,\infty)}(u))\bm{\nu}(\dd v, \dd u). \label{symbolxyPhi}
\end{align}
Then  $\psi:=\mathbf{\Sigma}-\mathbf{\Phi}$, extended by $0$ at $\infty$, introduces a Laplace dual symbol that is continuous at zero, the corresponding action of the pregenerators $\XX$ and $\YY$ being given by
\begin{equation*}
\XX\ee^{y}(x)=\big(\mathbf{\Sigma}(x,y)-\mathbf{\Phi}(x,y)\big)e^{-xy}=\mathcal{Y}\ee_{x}(y), \quad \{x,y\}\subset (0,\infty). \tag*{\qed}
\end{equation*}
\end{proposition}
Note that the possible mass of $\bm\nu$ on $\{\infty\}\times (0,\infty]$ allows  for a dynamic of \emph{immigration} in $\mathcal{X}$, which, informally speaking, the Laplace duality puts  in correspondence with  \textit{killing}  in $\mathcal{Y}$, and vice-versa.  The proof of Proposition~\ref{prop:nondecomposable-SpCMP}  is  an elementary inspection. We explain instead how these pregenerators naturally emerge. 

Given two non-local spectrally positive pregenerators $\mathcal{X}$ and $\mathcal{Y}$, for simplicity without killing, and the decomposition of their L\'evy measures into two measures on $\BB_{(0,\infty)}$, $\nu_x=\nu^1_x+\nu^2_x$ and $\tilde{\nu}_y=\tilde{\nu}^1_y+\tilde{\nu}^2_y$ respectively, according to Remark~\ref{rem:decompositionjumpmeasure},
one notices that $\mathcal{X}$ and $\mathcal{Y}$ are in Laplace duality if and only if, for  $\{x,y\}\subset(0,\infty)$,\small
\begin{equation*}
\int \left(e^{-uy}-1+uy\right)\nu^2_x(\dd u)-\int \left(e^{-xv}-1+xv\right)\tilde{\nu}^2_y(\dd v)=\int (1-e^{-uy})\nu^1_x(\dd u)-\int(1-e^{-xv})\tilde{\nu}^1_y(\dd v).
\end{equation*}\normalsize
In the setting of Proposition \ref{prop:nondecomposable-SpCMP} it corresponds to:
\begin{align}\label{eq:nu1_xXXsigmaphi}
&\nu^1_x(\dd u):=\int(1-e^{-xv}\mathbbm{1}_{(0,\infty)}(v))\bm{\nu}(\dd v, \dd u) , \quad\tilde{\nu}^1_y(\dd v):=\int(1-e^{-uy}\mathbbm{1}_{(0,\infty)}(u))\bm{\nu}(\dd v, \dd u),\\\label{eq:nu2_xXXsigmaphi}
&\nu^2_x(\dd u):=\int(e^{-xv}-1+xv)\bm{\eta}(\dd v, \dd u) \text{ and }\tilde{\nu}^2_y(\dd v):=\int(e^{-uy}-1+uy)\bm{\eta}(\dd v, \dd u).
\end{align}

Another approach leading to candidates of such Laplace dual symbols is to take  mixtures:
\begin{equation}\label{eq:mixtures-for-laplace-dual}
\psi(x,y):=\int \big(\hat{\Sigma}_r(x)\Sigma_r(y)-\hat{\Phi}_r(x)\Phi_r(y)\big)\alpha(\dd r),\quad \{x,y\}\subset [0,\infty),
\end{equation}
with $\alpha$ some measure on an arbitrary parameter set $\mathcal{P}$ and $\{\Sigma_r,\hat\Sigma_r\}\subset \spLpnot$, $\{\Phi_r,\hat\Phi_r\}\subset\pksubo$ for $r\in \mathcal{P}$ (we omit technical reservations having to do with measurability, integrability etc.). Of course to stay within the confines of Proposition~\ref{prop:nondecomposable-SpCMP}  we are obliged to choose the $\Sigma_r,\hat\Sigma_r,\Phi_r,\hat\Phi_r$ all of ``pure jump'' type (we do not make it precise, but the reader can undoubtedly decipher what we mean). The value of the integrand in \eqref{eq:mixtures-for-laplace-dual} at each fixed $r$ then corresponds to $\bm\eta$ and $\bm\nu$ being product measures. 

The following are examples of (new) Laplace dual symbols obtained by ``mixing" stable or gamma cases according to \eqref{eq:mixtures-for-laplace-dual} with $\mathcal{P}=(0,1)$ for which the integrals can  be effected in closed form.

\begin{example}\label{ex:mixedsigmastable}
Take  $\alpha(\dd r):=\dd r$,  $\hat{\Sigma}_r(x):=x^{r+1}$ and $\Sigma_r(y):=y^{r+1}$  for $r\in (0,1)$, $\{x,y\}\subset [0,\infty)$. We have
\[\psi(x,y)=\mathbf{\Sigma}(x,y):=\int_{0}^{1} \hat{\Sigma}_r(x)\Sigma_r(y)\dd r=\int_{0}^{1}(xy)^{r+1}\dd r=\frac{xy(xy-1)}{\ln xy} ,
\quad \{x,y\}\subset [0,\infty), 
\] the expression being understood in the limiting sense for $xy=1$ and $xy=0$. 
\end{example} 

\begin{example}
 Take  now $\alpha(\dd r):=\dd r$, $\hat \Phi_r(x):=x^{r}$ and $
\Phi_r(y):=y^{1-r}$ for $r\in (0,1)$, $\{x,y\}\subset [0,\infty)$. We obtain
\[\psi(x,y)=\mathbf{\Phi}(x,y):=\int_{0}^{1} \Phi_r(x)\hat{\Phi}_r(y)\dd r=\int_{0}^{1}x^ry^{1-r}\dd r=y\frac{\frac{x}{y}-1}{\ln \frac{x}{y}},\quad \{x,y\}\subset [0,\infty), 
\]   the expression being understood in the limiting sense for $\frac{x}{y}=1$ and $\frac{x}{y}=0$ when $y\ne 0$. 
\end{example} 

\begin{example}\label{ex:mixedgamma}
Fix $\gamma\in (-1,\infty)$. Take  $\alpha(\dd r):=r^{\gamma}\dd r$, $\hat{\Phi}_r(x):=\log(1+xr)$ and $\Phi_r(y):=\log(1+y/r)$, for $r\in (0,1)$, $\{x,y\}\subset [0,\infty)$. It gives us  \[\psi(x,y)=\mathbf{\Phi}(x,y):=\int_0^1\log(1+xr)\log(1+y/r)r^{\gamma}\dd r, \quad \{x,y\}\subset [0,\infty).\] 
Notice that $$\frac{\partial^2\mathbf{\Phi}}{\partial x\partial y}(x,y)=\int_0^1\frac{r^{\gamma}}{(1+xr)(1+y/r)}\dd r, \quad \{x,y\}\subset [0,\infty).$$
In particular, \begin{equation}\label{eq:momentcondgamma} \frac{\partial^2\mathbf{\Phi}}{\partial x\partial y}(0,0)=\frac{1}{1+\gamma}<\infty,
\end{equation}
which will be needed later.
\end{example}
We proceed with delineating the announced large class of Laplace dual symbols. It subsumes those of Table~\ref{table:fund-gen-dualities} as well as Proposition~\ref{prop:nondecomposable-SpCMP}, indeed it is nothing but the smallest convex cone containing all of them.
\begin{definition}\label{def:LDS} 
We denote by $\mathsf{LDS}$  the set of Laplace dual symbols $\psi$ that are  vanishing at infinity, continuous at zero and such that
\begin{equation}\label{def:LDSpsi}\psi(x,y)=x\Psi(y)+x^2\Sigma(y)+\mathbf{\Sigma}(x,y)-\mathbf{\Phi}(x,y)+\hat{\Sigma}(x)y^2+\hat{\Psi}(x)y+\kappa(xy), \ \ \{x,y\}\subset [0,\infty),\end{equation}
for $\{\Psi,\hat\Psi\}\subset\pkspLp$, $\{\Sigma, \hat\Sigma\}\subset \spLpnot$, $\kappa\in \pkLp_1$ and $\mathbf{\Sigma}, \mathbf{\Phi}$ as in \eqref{symbolxySigma}-\eqref{symbolxyPhi}. 
We say that a $\psi\in \mathsf{LDS}$ is: \textit{simple} if it is of the form
\begin{equation*}
\psi(x,y)=\hat{\Sigma}(x)\Sigma(y)\text{ or }\psi(x,y)=-\hat{\Phi}(x)\Phi(y), \quad  \{x,y\}\subset [0,\infty),
\end{equation*}
for $\{\Sigma, \hat\Sigma\}\subset \spLpnot$, $\{\Phi, \hat \Phi\}\subset\pksubo$; \textit{decomposable} if it is a finite sum of simple symbols. 
\end{definition}
The terminology decomposable is borrowed from Kolokoltsov \cite[Section 6.5]{zbMATH05875699}; all $\psi$ of the form \eqref{def:LDSpsi} for which the $\bm{\eta}$ and $\bm\nu$ of \eqref{symbolxySigma}-\eqref{symbolxyPhi} (responsible for the terms $\mathbf{\Sigma}, \mathbf{\Phi}$)  are finite nonnegative sums of product measures, whilst $\kappa\equiv 0$, are decomposable. It is also plain that $\mathsf{LDS}$ is closed under the operation of taking the dual.  

In Subsection~\ref{subsection:LDS-via-MPs} we construct through martingale problems positive Markov processes whose Laplace dual symbols belong to $\mathsf{LDS}$, under suitable assumptions. Subsection~\ref{sec:decomposable} examines decomposable Laplace symbols via stochastic equations. 

\begin{question}\label{question:LDS}
Are the Laplace dual symbols, continuous at zero and vanishing at infinity, all in $\mathsf{LDS}$?
\end{question}

\section{From Laplace dual symbols to completely monotone Markov processes}\label{sec:fromdualsymboltosemigroups}
Below, when   an a.e. qualifier appears without an explicit reference to a measure, we intend it relative to the Lebesgue measure on the Borel sets of the real line.

\subsection{Martingale problems and duality}\label{sec:martingaleproblem}  Let $E$ be a topological space and recall the notation \eqref{various-E-spaces}. For  
a linear   $A\subset \mathsf{M}(E)\times \mathsf{M}(E)$ (usually, $A\subset \mathsf{B}(E)\times \mathsf{B}(E)$):
\begin{itemize}
\item we say $A$ satisfies the positive maximum principle (PMP), when
\begin{equation*}
\left(\text{$(f,g)\in A$, $x_0\in E$ and $\sup_{x\in E} f(x)=f(x_0)\geq 0$}\right) \Rightarrow g(x_0)\leq 0;
\end{equation*}
\item by a solution $(X,\PP)$ to the martingale problem (MP) for $(A,\mu)$ (resp. relative to a filtration $\GG$),  $\mu$ a given probability on $(E,\BB_E)$, we intend \cite[p.~173]{EthierKurtz} an $E$-valued (resp. $\GG$-) progressively measurable stochastic process $X=(X_t)_{t\in [0,\infty)}$ defined under a probability $\PP$,  such that ${X_0}_\star\PP=\mu$ and such that for each $(f,g)\in A$, \small $$\left(f(X_t)-\int_0^tg(X_s)\dd s\right)_{t\in [0,\infty)}\text{ is  a martingale  in the natural filtration of $X$ (resp. in $\GG$)},$$\normalsize  where it is implicit that we insist that the Lebesgue integrals are well-defined and finite at least $\PP$-a.s.; 
\item by a solution to the MP for $A$ we intend an $X$ together with a family of probabilities $(\PP_x)_ {x\in E}$ all defined on a common measurable space, such that for all $x\in E$, $(X,\PP_x)$ is a solution to the MP for $(A,\delta_x)$.
\end{itemize}
 When $E$ is locally compact but not compact, picking an arbitrary ``cemetery'' $\triangle\notin E$ we: 
\begin{itemize}
\item put $E^\triangle:=E\cup\{\triangle\}$ for the one-point compactification of $E$;
\item set, for $f\in \mathbb{D}_{E^\triangle}$, 
\begin{align*}
\zeta_\triangle(f)&:=\inf\{t\in [0,\infty):f_{t-}=\triangle\text{ or }f_t=\triangle\}\\
&\quad \text{ for the explosion time of $f$ at $\triangle$ } (f_{0-}:=f_0)\text{; and }\\
\mathbb{D}_{E^\triangle}^{\mathrm{m}\triangle}&:=\{f\in \mathbb{D}_{E^\triangle}:\text{if $\zeta_\triangle(f)<\infty$, then }f(t)=\triangle\text{ for all $t\in [\zeta_\triangle(f),\infty)$}\};
\end{align*}
\item for an $f\in \mathbb{D}_{E^\triangle}^{\mathrm{m}\triangle}$ say that it does not explode continuously at $\triangle$, when $$\zeta_\triangle(f)<\infty\Rightarrow f(\zeta_\triangle(f)-)\in E.$$  
\end{itemize}
The superscript $\mathrm{m}\triangle$ stands for ``minimal at $\triangle$'', which is a relatively standard usage in this context. We shall sometimes also say that paths   have $\triangle$ strongly absorbing when they belong to (the measurable subset) $\mathbb{D}_{E^\triangle}^{\mathrm{m}\triangle}$ (of $\mathbb{D}_{E^\triangle}$). For a compact $E$ we set $E^\triangle:=E$.

A sufficiently well-behaved linear operator satisfying the PMP always admits a solution to its associated MP (the not-compact case reduces to the compact one, but is spelled out for visibility) \cite[ p.~199, Theorem~5.4]{EthierKurtz}:
\begin{theostar}\label{thm:mp-exists}
Let $E$ be a separable locally compact metric space and  let $A$ be a linear operator defined on  a dense linear subset of $(\mathsf{C}_0(E),\Vert\cdot\Vert_\infty)$, mapping into $\mathsf{C}_0(E)$ and satisfying the PMP. If $E$ is compact assume $(1,0)\in A$ and set $A^\triangle:=A$; when $E$ is not compact, introduce the operator $A^\triangle$ on $D_{A^\triangle}:=\{f\in \mathsf{C}(E^\triangle):(f-f(\triangle))\vert_E\in D_A\}$ as follows: $$(A^\triangle f)(x):=\mathbbm{1}_E(x)\Big(A\big((f-f(\triangle))\vert_E\big)\Big)(x),\quad x\in E^\triangle,\, f\in D_{A^\triangle}.$$ Then for all probabilities $\mu$ on $(E^\triangle,\BB_{E^\triangle})$ there exists a $\mathbb{D}_{E^\triangle}$-valued solution  $(X,\PP)$ to the MP for $(A^\triangle,\mu)$. If $E$ is not compact but is $\sigma$-compact, then we may ask  that the solution is $\mathbb{D}_{E^\triangle}^{\mathrm{m}\triangle}$-valued.  \qed
\end{theostar}
The next proposition dealing with ``adding independent killing'' to a MP is certainly part of folklore, but we could not find in the literature a rendering of it that meets our puposes, so we spell it out, deferring its proof to  Appendix~\ref{appendix:reductions}.
\begin{propositionstar}\label{lemma:reductions}
Let $E$ be a locally compact but not compact topological space, let $\mu$ be  probability on $(E^\triangle,\BB_{E^\triangle})$, $A\subset \mathsf{B}(E^\triangle)\times \mathsf{B}(E^\triangle)$ linear, $(1,0)\in A$,  $g(\triangle)=0$ for all $(f,g)\in A$, and let  $(X,\PP)$ be any $\mathbb{D}_{E^\triangle}^{\mathrm{m}\triangle}$-valued solution to the MP for $(A,\mu)$. Let $k:E\to [0,\infty)$ be locally bounded and Borel. Possibly under an extension of $\PP$, let $\mathbbm{e}$ be a $(0,\infty)$-valued exponential random time of mean one independent of $X$, and set, with the interpretation $k(\triangle):=0$:
\begin{align}\label{reductions:one}
\zeta&:= \inf  \left\{t\in [0,\infty):\int_0^t k(X_s)\dd s>\mathbbm{e}\right\};\\
X_t'&:=
\begin{cases}
X_t,&t\in [0,\zeta),\\
\triangle,& t\in [\zeta,\infty).
\end{cases}\label{reductions:two}
\end{align}
Then $(X',\PP)$ is a $\mathbb{D}_{E^\triangle}^{\mathrm{m}\triangle}$-valued solution to the MP for $(A',\mu)$, where $A':=\{(f,g- (f-f(\triangle))k):(f,g)\in A\text{ for which }(f-f(\triangle))k\in \mathsf{B}(E^\triangle)\}$.
\end{propositionstar}

In many situations one cannot directly observe a semigroup  duality between two processes, the only data available being their infinitesimal dynamics. 
Famous results by Ethier and Kurtz \cite[Section 4.4]{EthierKurtz} provide a set of sufficient conditions to go from generator duality to semigroup duality via martingale problems; we refer also to \cite[Proposition 1.2]{duality} for a theoretical criterion involving the domains of the generators  ensuring semigroup duality. We shall have occasion to apply a refinement of Ethier-Kurtz's \cite[Theorem~4.4.11, Corollary~4.4.13]{EthierKurtz}, which is the content of

\begin{theostar}\label{theorem:duality-from-mtgs}
Let $X=(X_t)_{t\in [0,\infty)}$ together with a random time $\sigma$ and $Y=(Y_t)_{t\in [0,\infty)}$ together with a random time $\tau$ be measurable processes valued in $E$ and $F$ respectively, $(X,\sigma)$ being independent of $(Y,\tau)$.  Suppose further $H:E \times F\to \mathbb{R}$ is bounded measurable, and let $\phi:E\times F\to \mathbb{R}$ be a measurable map such that 
\begin{equation}\label{eq:weird-condition}
\int_0^T\int_0^{T}\mathbb{E}[\phi_\pm(X_{s\land \sigma},Y_{t\land\tau})]\mathbbm{1}_{[0,T]}(s+t)\dd s\dd t<\infty,\quad T\in [0,\infty)
\end{equation}
(we intend it to hold with a $+$ or with a $-$, not necessarily both!), such that
\begin{equation}\label{eq:weird-condition++}
\EE[\vert \phi(X_\sigma,Y_\tau) \vert;\sigma+\tau\leq t]<\infty\text{ for a.e. $t\in [0,\infty)$},
\end{equation}
and such that
\begin{equation}\label{eq:mtg-one}
\left(H(X_{t\land \sigma},y)-\int_0^{t\land \sigma}\phi(X_s,y)\dd s\right)_{t\in [0,\infty)}\text{ is a martingale for all $y\in F$},\end{equation}
while 
\begin{equation}\label{eq:mtg-two}
\left(H(x,Y_{t\land \tau})-\int_0^{t\land \tau}\phi(x,Y_s)\dd s\right)_{t\in [0,\infty)}\text{ is a martingale for all $x\in E$},
\end{equation} 
where the Lebesgue integrals in \eqref{eq:mtg-one}-\eqref{eq:mtg-two} are assumed to be well-defined and finite a.s.

Then, for a.e. $t\in [0,\infty)$,
\begin{align}\nonumber
\mathbb{E}[H(X_{t\land \sigma},Y_0)]-\mathbb{E}[H(X_0,Y_{t\land\tau})]&=\EE[H(X_{(t-\tau)\land \sigma},Y_\tau)-H(X_0,Y_\tau);\tau\leq t]\\\nonumber
&\quad -\EE[H(X_\sigma,Y_{(t-\sigma)\land \tau})-H(X_\sigma,Y_0);\sigma\leq t]\\
&\quad +\EE\Big[\phi\big(X_\sigma,Y_\tau\big)\big(((t-\tau)_+-\sigma)_+-((t-\sigma)_+-\tau)_+\big)\Big],
\label{thmD:general}
\end{align}
 all the expectations being absolutely convergent.  In particular, if  $\sigma\equiv\infty\equiv\tau$, then \eqref{thmD:general} becomes \begin{equation}\label{thmD:particular}
\mathbb{E}[H(X_t,Y_0)]=\mathbb{E}[H(X_0,Y_t)]\text{ for a.e. $t\in [0,\infty)$.}
\end{equation}
\end{theostar}
Conditions \eqref{eq:weird-condition}-\eqref{eq:weird-condition++} weaken \cite[p. 192, Eq.~(4.50)]{EthierKurtz}, which requires  uniform controls of the form
\[\sup_{s,t\leq T}|H(X_s,Y_t)|\leq \Gamma_T \text{ and } \sup_{s,t\leq T}|\phi(X_s,Y_t)|\leq \Gamma_T,\]
for some integrable random variable $\Gamma_T$. We may mention also \cite[p.~39]{greven2015multitypespatialbranchingmodels} where the authors observed that some simplification could be made in Ethier-Kurtz's condition. 

The martingale conditions \eqref{eq:mtg-one}-\eqref{eq:mtg-two} can be further relaxed to the the integrability and constancy of the expectations of the processes in question. In case $\sigma=\infty$ we may formulate \eqref{eq:mtg-one} as $X$  solving (under the ambient probability $\PP$) the martingale problem for $\{(H(\cdot,y),\phi(\cdot,y)):y\in F\}$ (with initial law ${X_0}_\star \PP$), and anyway the latter implies \eqref{eq:mtg-one} by optional stopping if $\sigma$ is a stopping time (assuming sufficient regularity so that optional stopping applies).

The proof of Theorem \ref{theorem:duality-from-mtgs} can be found in Appendix \ref{sec:proofoftheorem:duality-from-mtgs}.

\subsection{Application of Theorems~\ref{thm:mp-exists} and~\ref{theorem:duality-from-mtgs} to the case of Laplace duality}\label{subsection:applications-CE-to-laplace}
Appealing to Theorem~\ref{thm:mp-exists} we may put in evidence a result on the existence of a process associated to the MP of a given Laplace symbol. For simplicity we restrict attention to the case in which the Laplace symbol vanishes at the point $\infty$. Recall from \eqref{the-exponential-domain} that we denote by $\mathcal{D}$ the linear span of $\{\ee^y, \ y\in (0,\infty)\}$ as well as Definition~\ref{def:pregen-1}.
\begin{proposition}\label{prop:existencesolutionMP}
Let $\psi$ be a Laplace symbol with $\psi(\infty,\cdot)\equiv 0$ and set $\XX_{e,\psi}':=\XX_{e,\psi}\cap (\{f\in \mathsf{C}([0,\infty]):f(\infty)=0\}\times \mathsf{C}([0,\infty]))$.
 Assume $\mathcal{D}\subset D_{\XX_{e,\psi}'}$. Then, for every $x\in [0,\infty]$, there exists a $\mathbb{D}_{[0,\infty]}^{\mathrm{m}\infty}$-valued solution to the MP for $(\XX_{e,\psi}',\delta_x)$.
\end{proposition}
Since $\psi(\infty,\cdot)\equiv 0$, $D_{\XX_{e,\psi}'}$ actually consists of those continuous $f\in D_{\XX_{e,\psi}}$ that are equal to $0$ at $\infty$ for which the output $\XX_{e,\psi}f$ is also continuous \emph{and} equal to $0$ at $\infty$. Remark also that under the conditions of the proposition $\XX_\psi$ is continuous at zero.
\begin{proof}[Proof of Proposition~\ref{prop:existencesolutionMP}]
The case $x=\infty$ is elementary (just take the process identically equal to $\infty$). From now on $x<\infty$. Denote by $l$ the unique L\'evy family for which $\psi=\psi_l$, writing $l_0=:(\nu_0,d_0,x_0)$ and $l_x=:(\nu_x,a_x,b_x,c_x)$ for $x\in (0,\infty)$ (since $\psi(\infty,\cdot)\equiv 0$, $l_\infty=(0,0)$). Put $A:=\{(f,(\mathcal{X}_{e,\psi}(f\mathbbm{1}_{[0,\infty)}))\vert_{[0,\infty)}):f\in \mathsf{C}_0([0,\infty))\text{ such that } f\mathbbm{1}_{[0,\infty)}\in D_{\XX_{e,\psi}}\text{ and }(\mathcal{X}_{e,\psi}(f\mathbbm{1}_{[0,\infty)}))\vert_{[0,\infty)}\in  \mathsf{C}_0([0,\infty))\}$.
     We claim the operator $A$ satisfies the PMP: 
\begin{equation}
{\footnotesize \left(\text{$x_0\in [0,\infty)$ and $\sup_{x\in [0,\infty)} f(x)=f(x_0)\geq 0$}\right) \Rightarrow (\XX_{e,\psi} (f\mathbbm{1}_{[0,\infty)}))(x_0)\leq 0,\quad f\in D_A;}
\end{equation}
let us check it. On the one hand, for $x_0\in (0,\infty)$, the condition $\sup_{x\in [0,\infty)} f(x)=f(x_0)\geq 0$ requires $f'(x_0)=0$ if $b_{x_0}\ne 0$, $f''(x_0)\leq 0$ if $a_{x_0}\ne 0$, and $f(x_0+u)-f(x_0)\leq 0$ for all $u\in [-x_0,\infty]$, so that indeed \small$$(\XX_{e,\psi} (f\mathbbm{1}_{[0,\infty)}))(x_0)=\int\left(f(x_0+u)-f(x_0)\right)\nu_{x_0}(\dd u)+a_{x_0}f''(x_0)+b_{x_0}f'(x_0)-c_{x_0}f(x_0)\leq 0,$$ \normalsize on using also $\{a_{x_0},c_{x_0},f(x_0)\} \subset [0,\infty)$ and the fact that $\nu_{x_0}$ is a measure on $\BB_{[-x_0,\infty)}$. On the other hand, for $x_0=0$, $\sup_{x\in [0,\infty)} f(x)=f(0)\geq 0$ requires $f'(0)\leq 0$ if $d_{x_0}\ne 0$, and $f(u)-f(0)\leq 0$ for all $u\in (0,\infty)$, which in turn  renders
$$(\XX_{e,\psi} (f\mathbbm{1}_{[0,\infty)}))(0)=\int\left(f(u)-f(0)\right)\nu_{0}(\dd u)+d_0f'(0)-c_0f(0)\leq 0,$$
on using also $\{d_0,c_0,f(0)\}\subset [0,\infty)$ and the fact that $\nu_0$ is a measure on $\BB_{(0,\infty)}$. Thus the PMP for $A$ has been verified. Moreover, $D_A$ is dense in $(\mathsf{C}_0([0,\infty)),\Vert\cdot\Vert_\infty)$ by Stone-Weierstrass and since $\mathcal{D}\subset D_{\XX_{e,\psi}'}$.
Theorem~\ref{thm:mp-exists} with $E=[0,\infty)$, $\triangle=\infty$, $A$ as itself now assures  us that  there exists a $\mathbb{D}_{[0,\infty]}^{\mathrm{m}\infty}$-valued solution to the MP for  $(\XX_{e,\psi}',\delta_x)$, indeed, in the notation of said theorem, $A^\triangle\supset \XX_{e,\psi}'$.
\end{proof}    

We now specialize Theorem~\ref{theorem:duality-from-mtgs}  to the setting of Laplace duality. Recall Definition~\ref{def:laplacesymboldualfamily}.
\begin{theorem}\label{thm:generators}-
Work under $0^+\cdot\infty$, $\infty\cdot 0^+$. Let $\psi$ be a Laplace dual symbol,  vanishing at infinity. Put 
\begin{equation}\label{thm:generators-phi}
\phi(x,y):=\psi(x,y)e^{-xy},\quad (x,y)\in [0,\infty]^2,
\end{equation} and introduce the conditions
 \begin{align}
 &\phi(x,\cdot)\text{ is continuous at zero for all $x\in [0,\infty)$};\label{phi:0}\\
&\lim_{y\to \infty}\phi(x,y)=0\text{ for }x\in  (0,\infty);\label{phi:2}\\
&\sup_{(x,y)\in [0,m]^2}\vert\phi(x,y)\vert<\infty\text{ for all $m\in [0,\infty)$};\label{phi:4}\\
&\sup_{(x,y)\in [0,m]\times(\epsilon,n)}\vert\phi(x,y)\vert<\infty\text{ for all }(m,n,\epsilon)\in (0,\infty)^3;\label{phi:variation}\\
&\phi(0,\cdot)\text{ vanishes on $[0,\infty]$}\label{phi:variation'}.
\end{align}
If we intend, say, \eqref{phi:0} to hold with $\hat\phi$ in lieu of $\phi$, we will write $\widehat{\eqref{phi:0}}$.  

 Fix next $\{x_0,y_0\}\subset  [0,\infty]$. Denoting by $\XX$ and $\YY$ the  pregenerators associated to $\psi\vert_{[0,\infty]\times (0,\infty)}$ and $\hat\psi\vert_{[0,\infty]\times (0,\infty)}$, respectively, let $(X,\PP_{x_0})$ be a $\mathbb{D}_{[0,\infty]}^{\mathrm{m}\infty}$-valued solution to the MP for $(\XX,\delta_{x_0})$ and let $(Y,\PP^{y_0})$  be a $\mathbb{D}_{[0,\infty]}^{\mathrm{m}\infty}$-valued solution to the MP for $(\YY,\delta_{y_0})$.   Set 
 \begin{align*}
 \sigma_m^{+}&:= \inf\{t\in [0,\infty):X_t> m\}\text{ for }m\in \mathbb{N},\, \sigma_{\delta}^{-}:= \inf\{t\in [0,\infty):X_t<\delta\}\text{ for }\delta\in (0,\infty),\\
 \tau_n^+:&= \inf\{t\in [0,\infty):Y_t> n\}\text{ for }n\in \mathbb{N}\text{ and } \tau^-_{\epsilon}:=\inf\{t\in [0,\infty):Y_t<\epsilon\}\text{ for }\epsilon\in (0,\infty).
 \end{align*}
 
 Write $\PP:=\PP_{x_0}\times \PP^{y_0}$ and $\EE$ for the  associated expectation. For random times $\sigma$ and $\tau$ under the probabilities $\PP_{x_0}$ and $\PP^{y_0}$, respectively, we refer below to the conditions 
\begin{equation}\label{eq:weird-conditionLD}
\int_0^T\int_0^{T}\mathbb{E}[\phi_\pm(X_{s\land \sigma},Y_{t\land\tau})]\mathbbm{1}_{[0,T]}(s+t)\dd s\dd t<\infty,\quad T\in [0,\infty),
\end{equation}
(we mean it to hold with a ``$+$'' for all $T\in [0,\infty)$ \emph{or}  with a ``$-$'' for all $T\in [0,\infty)$) and
\begin{equation}\label{eq:weird-condition++LD}
\EE[\vert \phi(X_\sigma,Y_\tau) \vert;\sigma+\tau\leq t]<\infty\text{ for a.e. $t\in [0,\infty)$}.
\end{equation}

Suppose  one of \ref{thm:generators:0}-\ref{thm:generators:III} below is met.

\begin{enumerate}[(I)]
\item\label{thm:generators:0} \eqref{eq:weird-conditionLD} is verified with $\sigma\equiv \infty\equiv \tau$. The process $(\mathbbm{1}_{\{Y_t<\infty\}}-\int_0^t\phi(0,Y_s)\dd s)_{t\in [0,\infty)}$ is a martingale, $Y$ has $0$ non-sticky and $y_0>0$. 
\item\label{thm:generators:I} For all $(m,n)\in \mathbb{N}^2$, \eqref{eq:weird-conditionLD}-\eqref{eq:weird-condition++LD} are met with $\sigma=\sigma_m^{+}$, $\tau=\tau_n^+$. The function $\phi$ satisfies \eqref{phi:2}, $\widehat{\eqref{phi:2}}$. Also, $X$ and $Y$ have $0$ non-sticky and $x_0>0$, $y_0>0$.
\item\label{thm:generators:II}  For all $(n,\epsilon)\in \mathbb{N}\times (0,\infty)$, \eqref{eq:weird-conditionLD} is met with $\sigma=\infty$, $\tau=\tau_n^+\land \tau^-_{\epsilon}$. The function $\phi$ satisfies  $\widehat{\eqref{phi:0}}$, \eqref{phi:variation} and  $\widehat{\eqref{phi:variation'}}$. Also,  $X$ has $\infty$ non-sticky, there is a.s. no continuous explosion for $Y$ at $\infty$ and $Y$ is $\mathbb{D}_{[0,\infty]}^{\mathrm{m}0}$-valued. 
\item\label{thm:generators:III} For all $(\delta,\epsilon)\in (0,\infty)^2$, \eqref{eq:weird-conditionLD}-\eqref{eq:weird-condition++LD} are met with $\sigma=\sigma_{\delta}^{-}$, $\tau= \tau^-_{\epsilon}$. The function $\phi$ satisfies  \eqref{phi:0}, $\widehat{\eqref{phi:0}}$,  \eqref{phi:4} \eqref{phi:variation'} and $\widehat{\eqref{phi:variation'}}$. Also,  $X$ and $Y$ are both $\mathbb{D}_{[0,\infty]}^{\mathrm{m}0}$-valued and have $\infty$ non-sticky.
\end{enumerate}
Then
\begin{equation}\label{eq:ld-apriori}
\EE_{x_0}[e^{-X_ty_0}]=\EE^{y_0}[e^{-x_0Y_t}],\quad  t\in [0,\infty).
\end{equation}
\end{theorem}
In the proof, \eqref{eq:weird-conditionLD}-\eqref{eq:weird-condition++LD} will correspond to \eqref{eq:weird-condition}-\eqref{eq:weird-condition++} of Theorem~\ref{theorem:duality-from-mtgs} applied with $E=[0,\infty]=F$ (unless otherwise indicated) and $H(x,y)=e^{-xy}$ for $(x,y)\in E\times F$ (under the conventions $0^+\cdot\infty$, $\infty\cdot 0^+$). The conditions \ref{thm:generators:0}-\ref{thm:generators:III}  are set up so that, possibly after stopping the processes $X$ and $Y$ on passing below and/or above a given level, we are in a position to apply Theorem~\ref{theorem:duality-from-mtgs}, and then to remove the stopping by a passage to the limit.

\begin{remark}\label{rmk:gen-to-duality}
With reference to \ref{thm:generators:0}-\ref{thm:generators:I}, if, holding $x_0$ fixed, \eqref{eq:ld-apriori} has been established for all $y_0\in (0,\infty]$, say, and if $Y$ is weakly-$[0,\infty)$ continuous at $0$  under the family of probabilities $(\PP^{y_0})_{y_0\in [0,\infty]}$, then it holds also for $y_0=0$ by passing to the limit and taking note of \eqref{continuityat0LT}. Similarly with the roles of $X$ and $Y$ transposed. 
\end{remark}

Concerning \ref{thm:generators:II}, the proof will show that in lieu of $\widehat{\eqref{phi:variation'}}$ one can assume $Y$ a.s. does not jump to zero from positive levels and the same conclusion prevails provided $y_0>0$. Similarly, in \ref{thm:generators:0}, we can drop the martingale condition and assume instead that $X$ and $Y$ have both $0$ non-sticky and $x_0>0$, $y_0>0$.

\begin{proof}[Proof of Theorem~\ref{thm:generators}]
\eqref{eq:ld-apriori} is trivial for $x_0=\infty$ or $y_0=\infty$ ($\because$ $0^+\cdot\infty$, $\infty\cdot 0^+$ and  $\because$ $\infty$ is absorbing for $X$ and $Y$); restrict to  $x_0\in [0,\infty)$  and $y_0\in [0,\infty)$. 

\ref{thm:generators:0}. Here we are in a position to apply directly Theorem~\ref{theorem:duality-from-mtgs}\eqref{thmD:particular}. Notice that we can take $F=(0,\infty]$ therein due to the non-stickiness of $0$ for $Y$ and $y_0>0$. This is why we do not need the corresponding martingale $(\mathbbm{1}_{\{X_t<\infty\}}-\int_0^t\phi(X_s,0)\dd s)_{t\in [0,\infty)}$. True, we have \eqref{eq:ld-apriori} a priori only for a.e. $t\in [0,\infty)$, but for general $t$ we may pass to the limit via bounded convergence, exploiting right-continuity of $X$ and $Y$ together with $X$, $Y$ having $\infty$ strongly absorbing.

 \ref{thm:generators:I}. Remark that $\sigma_n^{+}$ (resp. $\tau_n^+$)  is a stopping time in the right-continuous modification of the natural filtration of $X$ (resp. $Y$) for all $n\in \mathbb{N}$. Note also that if there is a.s. no continuous explosion for $Y$ at $\infty$, then  $Y_{\tau_n^+}=\infty$ for all large enough $n\in \mathbb{N}$ a.s.-$\PP^{y_0}$ on $\{\lim_{n\to\infty}\tau_n^+<\infty\}$. 
 
For $m\in \mathbb{N}$, by optional stopping, the process
$\left(e^{-X_{t\land \sigma_m^{+}}y}-\int_0^{t\land \sigma_m^{+}}\phi(X_s,y)\dd s\right)_{t\in [0,\infty)}$
is a $\PP_{x_0}$-martingale for $y\in (0,\infty)$; but it is so also for  $y=\infty$ (trivially, due to $\phi(\cdot,\infty)=0$, $0^+\cdot\infty$). Analogously we avail ourselves of the observation that so too  the processes
$\left(e^{-xY_{t\land \tau_n^+}}-\int_0^{t\land \tau_n^+}\phi(x,Y_s)\dd s\right)_{t\in [0,\infty)}$ for $ x\in (0,\infty]$ are $\PP^{y_0}$-martingales for all $n\in \mathbb{N}$.  As $x_0>0$ and $y_0>0$ we are now in a position to apply Theorem~\ref{theorem:duality-from-mtgs}\eqref{thmD:general} with $E=(0,\infty]=F$, obtaining, for a.e. $t\in [0,\infty)$, for all $(m,n)\in \mathbb{N}^2$, 
\begin{align*}
&\EE_{x_0}[e^{-X_{t\land \sigma_m^{+}}y_0}]-\EE^{y_0}[e^{-x_0 Y_{t\land \tau_n^+}}]\\
&\quad =\EE[e^{-X_{(t-\tau_n^+)_+\land \sigma_m^{+}}Y_{\tau_n^+}}-e^{-x_0Y_{\tau_n^+}};\tau_n^+\leq t]-\EE[e^{-X_{\sigma_m^{+}}Y_{(t-\sigma_m^{+})_+\land \tau_n^+}}-e^{-X_{\sigma_m^{+}}y_0};\sigma_m^{+}\leq t]\\
&\qquad +\EE\Big[\phi\Big(X_{\sigma_m^{+}},Y_{\tau_n^+}\Big)\Big(\big((t-\tau_n^+)_+-\sigma_m^{+}\big)_+-((t-\sigma_m^{+})_+-\tau_n^+)_+\Big)\Big].
\end{align*}
 We now pass to the limit $m\to\infty$ and then $n\to\infty$ (or vice versa) in the preceding and deduce, via bounded convergence,   that  \eqref{eq:ld-apriori} holds  for a.e. $t\in [0,\infty)$. Here, on the r.h.s. we apply \eqref{phi:2}, $\widehat{\eqref{phi:2}}$ for the last term, whereas for the first two terms we note crucially the non-stickiness of $0$ for $X$ and $Y$ and $x_0$, $y_0>0$. On the l.h.s. the convergence is plain due to $x_0,y_0>0$.

 We have established that \eqref{eq:ld-apriori} holds  for a.e. $t\in [0,\infty)$. But then it holds for all $t\in [0,\infty)$ by right-continuity and bounded convergence. 
 
 \ref{thm:generators:II}. Still $\tau_n^+$  is a stopping time in the right-continuous modification of the natural filtration of $Y$ for all $n\in \mathbb{N}$, the same is true of $\tau^-_{\epsilon}$ for all $\epsilon\in (0,\infty)$. 
 By the non-stickiness of $\infty$ for $X$ and since $0$ is absorbing for $Y$ the claim is trivial for $y_0=0$ so assume $y_0>0$.   
 
 The process $\left(e^{-X_{t}y}-\int_0^{t}\phi(X_s,y)\dd s\right)_{t\in [0,\infty)}$ 
is a $\PP_{x_0}$-martingale for $y\in (0,\infty)$, but also for $y=0$ (due to $\widehat{\eqref{phi:variation'}}$, $X$ having $\infty$ non-sticky) and for $y=\infty$ ($\because$ $\phi(\cdot,\infty)=0$, $0^+\cdot\infty$). For all $(n,\epsilon)\in \mathbb{N}\times (0,\infty)$, we infer by optional stopping, that the process
$\left(e^{-xY_{t\land \tau_n^+\land \tau^-_{\epsilon}}}-\int_0^{t\land \tau_n^+\land \tau^-_{\epsilon}}\phi(x,Y_s)\dd s\right)_{t\in [0,\infty)}$
is a $\PP^{y_0}$-martingale for $x\in (0,\infty)$, but also for $x=\infty$ ($\because$ $\phi(\infty,0)=0$, $\infty\cdot 0^+$) and $x=0$ on passing to the limit $x\downarrow 0$ by bounded convergence, $0^+\cdot\infty$, $\widehat{\eqref{phi:0}}$ and \eqref{phi:variation}.

Applying Theorem~\ref{theorem:duality-from-mtgs}\eqref{thmD:general} we deduce that for all $(n,\epsilon)\in \mathbb{N}\times (0,\infty)$, for a.e. $t\in [0,\infty)$, 
$$\EE_{x_0}[e^{-X_{t}y_0}]-\EE^{y_0}[e^{-x_0 Y_{t\land \tau_n^+\land \tau^-_{\epsilon}}}]=\EE[e^{-X_{(t-\tau_n^+\land \tau^-_{\epsilon})_+}Y_{\tau_n^+\land \tau^-_{\epsilon}}}-e^{-x_0Y_{\tau_n^+\land \tau^-_{\epsilon}}};\tau_n^+\land \tau^-_{\epsilon}\leq t].$$
  Due to the there being  a.s. no continuous explosion for $Y$ at $\infty$ and $0^+\cdot\infty$ we may now pass to the limit $n\to\infty$ to obtain
$$\EE_{x_0}\left[e^{-X_{t}y_0}\right]-\EE^{y_0}\left[e^{-x_0 Y_{t\land \tau^+_\infty\land \tau^-_{\epsilon}}}\right]=\EE\left[e^{-X_{(t-\tau^+_\infty\land \tau^-_{\epsilon})_+}Y_{\tau^+_\infty\land \tau^-_{\epsilon}}}-e^{-x_0Y_{\tau^+_\infty\land \tau^-_{\epsilon}}};\tau^+_\infty\land \tau^-_{\epsilon}\leq t\right]$$ (with $\tau^+_\infty:=\lim_{n\to\infty}\tau_n^+$),  and then ($\because$ $Y$ is $\mathbb{D}_{[0,\infty]}^{\mathrm{m}0}\cap \mathbb{D}_{[0,\infty]}^{\mathrm{m}\infty}$-valued, $X$ has $\infty$ non-sticky, $0^+\cdot\infty$) $\epsilon\downarrow 0$ in the preceding (in this order) and deduce, via bounded convergence,   that  \eqref{eq:ld-apriori} holds  for a.e. $t\in [0,\infty)$. Therefore it holds in fact for all $t\in [0,\infty)$ by right-continuity, bounded convergence and $X$ (resp. $Y$) having $\infty$ non-sticky (resp. strongly absorbing).

 \ref{thm:generators:III}.  By the non-stickiness of $\infty$ for $X$ and $Y$ the claim is trivial for $y_0=0$ or $x_0=0$ so assume $y_0>0$ and $x_0>0$.

 For $\delta\in (0,\infty)$, by optional stopping, the process $\left(e^{-X_{t\land \sigma_{\delta}^{-}}y}-\int_0^{t\land \sigma_{\delta}^{-}}\phi(X_s,y)\dd s\right)_{t\in [0,\infty)}$
is a $\PP_{x_0}$-martingale for $y\in (0,\infty)$; but it is so also for $y=0$ ($\because$ $\widehat{\eqref{phi:variation'}}$, $X$ has $\infty$ non-sticky) and for $y=\infty$ (trivially, due to $\phi(\cdot,\infty)=0$, $0^+\cdot\infty$). Analogously we avail ourselves of the observation that the processes
$\left(e^{-xY_{t\land \tau^-_{\epsilon}}}-\int_0^{t\land \tau^-_{\epsilon}}\phi(x,Y_s)\dd s\right)_{t\in [0,\infty)}$ for $ x\in [0,\infty]$
are $\PP^{y_0}$-martingales for all $\epsilon\in (0,\infty)$. From Theorem~\ref{theorem:duality-from-mtgs}\eqref{thmD:general} we deduce that,  for $(\delta,\epsilon)\in (0,\infty)^2$, for a.e. $t\in [0,\infty)$, 
\begin{align*}
&\EE_{x_0}[e^{-X_{t\land \sigma_{\delta}^{-}}y_0}]-\EE^{y_0}[e^{-x_0 Y_{t\land \tau^-_{\epsilon}}}]\\
&\quad =\EE[e^{-X_{(t-\tau^-_{\epsilon})_+\land \sigma_{\delta}^{-}}Y_{\tau^-_{\epsilon}}}-e^{-x_0Y_{\tau^-_{\epsilon}}};\tau^-_{\epsilon}\leq t]-\EE[e^{-X_{\sigma_{\delta}^{-}}Y_{(t-\sigma_{\delta}^{-})_+\land \tau^-_{\epsilon}}}-e^{-X_{\sigma_{\delta}^{-}}y_0};\sigma_{\delta}^{-}\leq t]\\
&\qquad +\EE\Big[\phi\big(X_{\sigma_{\delta}^{-}},Y_{\tau^-_{\epsilon}}\big)\big(((t-\tau^-_{\epsilon})_+-\sigma_{\delta}^{-})_+-((t-\sigma_{\delta}^{-})_+-\tau^-_{\epsilon})_+\big)\Big].
\end{align*}
We now pass to the limit $\epsilon\downarrow 0$ and then $\delta\downarrow 0$ (or vice versa) in the preceding and deduce, via bounded convergence,   that  \eqref{eq:ld-apriori} holds  for a.e. $t\in [0,\infty)$, as follows. On the l.h.s. there is convergence due to $X$ and $Y$ being $\mathbb{D}_{[0,\infty]}^{\mathrm{m}0}$-valued. On the r.h.s. we apply \eqref{phi:0}, $\widehat{\eqref{phi:0}}$, \eqref{phi:4} for the last term, whereas for the first two terms we note crucially the non-stickiness of $\infty$ for $X$ and $Y$.

 We have established that \eqref{eq:ld-apriori} holds  for a.e. $t\in [0,\infty)$. But then it holds for all $t\in [0,\infty)$ by right-continuity, bounded convergence and the non-stickiness of $\infty$ for $X$ and $Y$. 
\end{proof}
As an immediate illustration of the power of Theorem~\ref{thm:generators} we present the class of Laplace dualities corresponding to Theorem~\ref{theorem:tensor-form}\ref{clement:ii}. 
\begin{example}\label{example:dual-subo}
Take any $\{\Phi,\hat\Phi\}\subset\pksubo$ and put $$\psi(x,y):=-\mathbbm{1}_{[0,\infty)^2}(x,y)\Phi(y)\hat\Phi(x),\quad (x,y)\in [0,\infty]^2.$$ 
Denote by $\XX$ and $\YY$ the  pregenerators associated to $\psi\vert_{[0,\infty]\times (0,\infty)}$ and $\hat\psi\vert_{[0,\infty]\times (0,\infty)}$ respectively. We agree to interpret $\frac{1}{0}:=\infty$, $\xi_\infty:=\xi_{\infty-}$.

Let $\xi$ be a possibly killed subordinator with Laplace exponent $\Phi$ under the probability $\PP$, $\uparrow$ and vanishing at zero with certainty. For $x\in (0,\infty)$ define the continuous time-change $$\rho_t^x:=\inf \left\{s\in [0,\infty):\int_0^s \frac{\dd u}{\hat \Phi(x+\xi_u)\mathbbm{1}_{[0,\infty)}(\xi_u)}>t\right\},\quad t\in [0,\infty),$$ and set $\PP_x:={X^x}_\star \PP$ for the  process $X^x:=(x+\xi_{\rho_t^x})_{t\in [0,\infty)}$, a law on $\mathbb{D}^{\mathrm{m}\infty,\uparrow}_{[0,\infty]}:=\mathbb{D}^{\mathrm{m}\infty}_{[0,\infty]}\cap \{\uparrow \,\text{paths}\}$. Additionally we introduce $\PP_\infty$ as the only probability on the space $\mathbb{D}^{\mathrm{m}\infty}_{[0,\infty]}$ under which the canonical process $X$ stays at $\infty$ a.s., i.e. $\PP_\infty=\delta_{\underline{\infty}}$, where $\underline{\infty}$ is the constant path identically equal to $\infty$. As for the point of issue $0$ we notice that $(x+\xi_{\rho_t^x})_{t\in[0,\infty)}$ is $\uparrow$ in $x\in (0,\infty)$, therefore we may unambiguously ask that $\PP_0:={X^0}_\star\PP$ for $X^0:=(\text{$\downarrow$-$\lim$}_{x\downarrow 0}(x+\xi_{\rho_t^x}))_{t\in [0,\infty)}$, which  still takes its values in $\mathbb{D}^{\mathrm{m}\infty,\uparrow}_{[0,\infty]}$ (right-continuity follows because we can interchange monotone double limits). In this manner we construct the positive Markov process $X$ under the probabilities $(\PP_x)_{x\in [0,\infty]}$ (the Markov property of $X$ follows from the strong Markov property of $\xi$). 

From the martingale $\left(e^{-(x+\xi_s) y}+\Phi(y)\int_0^se^{-(x+\xi_v) y}\dd v\right)_{s\in [0,\infty)}$ for $\xi$ under $\PP$ we get by time-change (optional sampling and a change of variables in the Lebesgue integral) the deterministically-locally-bounded martingale $\left(e^{-X^x_t y}+\int_0^t\Phi(y)\hat\Phi(X^x_u)e^{-X^x_u y}\dd u\right)_{t\in [0,\infty)}$ under $\PP$, this for all $y\in (0,\infty)$ and all $x\in (0,\infty)$, but also for $x=0$ in the limit (by bounded convergence). Put another way, for all $x\in [0,\infty]$ (the case $x=\infty$ is trivial) $(X,\PP_x)$ is a solution to the martingale problem for $(\XX,\delta_x)$. It also follows easily that $\XX$ is the restriction of the generator of $X$ to $\DD$, which is to say that 
\begin{equation*}
\text{the Laplace symbol of $X$ is $\psi\vert_{[0,\infty]\times (0,\infty)}$.}
\end{equation*}

Interchanging the roles of $\Phi$ and $\hat\Phi$ we avail ourselves of the $\mathbb{D}^{\mathrm{m}\infty,\uparrow}_{[0,\infty]}$-valued positive Markov process $Y$ under the probabilities $(\QQ_y)_{y\in [0,\infty]}$, 
\begin{equation*}
\text{the Laplace symbol of $Y$ being $\hat\psi\vert_{[0,\infty]\times (0,\infty)}$,}
\end{equation*}
and such that for all $y\in [0,\infty]$, $(Y,\QQ_y)$ is a solution to the martingale problem for $(\YY,\delta_y)$. 

Applying Theorem~\ref{thm:generators}\ref{thm:generators:I} (it is trivial to check the conditions because the $\Phi$, $\hat \Phi$ are bounded by linear growth, because the $\phi$ of \eqref{thm:generators-phi} does not change sign, and since $X$ and $Y$ are $\uparrow$) together with Remark~\ref{rmk:gen-to-duality} we deduce that 
\begin{equation*}
\text{$X$ is in $(0^+\cdot \infty, \infty \cdot 0^+)$-Laplace duality with $Y$.}
\end{equation*}
We emphasize the great simplification in the argument afforded by the fact that $\phi\leq 0$ everywhere, indeed \eqref{eq:weird-conditionLD} holds just because $\phi_+\equiv 0$.
\end{example}
In a similar fashion, for  $\{\Sigma, \hat\Sigma\} \subset \spLpnot$, we could define
\begin{equation*}
\psi(x,y) := \mathbbm{1}_{[0,\infty)^2}(x,y)\,\Sigma(y)\hat{\Sigma}(x), \quad (x,y) \in [0,\infty]^2,
\end{equation*}
coresponding to Theorem~\ref{theorem:tensor-form}\ref{clement:i};  construct, via time-change, two processes $X$ and $Y$ with symbols $\psi\vert_{[0,\infty]\times (0,\infty)}$ and $\hat\psi\vert_{[0,\infty]\times (0,\infty)}$, respectively; and establish that they lie in $(0^+\cdot \infty, \infty \cdot 0^+)$-Laplace duality by applying Theorem~\ref{thm:generators}\ref{thm:generators:III}. We omit the details here, as this case will be covered in the next subsection, cf. forthcoming Example~\ref{exampleSigmaSigma}.
 
Excepting the cases in which the Laplace dual symbol maintains a constant sign, the most challenging conditions to verify in Theorem \ref{thm:generators} are typically \eqref{eq:weird-conditionLD}-\eqref{eq:weird-condition++LD}. By contrast, conditions \eqref{phi:0}-\eqref{phi:variation} are always satisfied for any element of $ \mathsf{LDS}$ (and, naturally, also for its dual), except \eqref{phi:2} and even that fails only when in the representation \eqref{def:LDSpsi} the L\'evy measure of $\kappa$ has an atom at $-1$.

\subsection{Processes with Laplace symbols in $\mathsf{LDS}$}\label{subsection:LDS-via-MPs}
We seek here sufficient conditions on a function $\psi\in \mathsf{LDS}$ that ensure the existence and uniqueness in law of positive Markov processes $X$ and $Y$, whose  symbols are $\psi\vert_{[0,\infty]\times (0,\infty)}$ and $\hat{\psi}\vert_{[0,\infty]\times (0,\infty)}$, respectively, and which satisfy $(0^+\cdot \infty, \infty \cdot 0^+)$-Laplace duality  at the level of their semigroups. Recall that by Definition~\ref{def:LDS} of the class $\mathsf{LDS}$, any such $\psi$ is vanishing at infinity and continuous at zero.

Let us begin our analysis with a technical result having to do with the nature of the extended pregenerator of Definition~\ref{def:pregen-1} associated to a Laplace dual symbol from $\mathsf{LDS}$. Before that, some more notation is in order.
\begin{definition}\label{definition:more-generators}
 Introduce  
 \begin{align*}
 \mathcal{D}_{e}:=\{f\in \mathbb{R}^{[0,\infty]}&:f(\infty)=0,\, f\vert_{[0,\infty)}\in \mathsf{C}^2([0,\infty))\\*
 &\qquad\text{and } \lim_{x\to\infty} x^2(\vert f(x)\vert+\vert f'(x)\vert+\vert f''(x)\vert )=0\}.\end{align*} Given a Laplace dual symbol $\psi$, write $\XX_\psi:=\XX_{\psi\vert_{[0,\infty]\times (0,\infty)}}$ (resp. $\XX_{e,\psi}:=\XX_{e,\psi\vert_{[0,\infty]\times (0,\infty)}}$) for the  (resp. extended) pregenerator associated with $\psi$, also $\XX_{e,\psi}':=\XX_{e,\psi}\cap (\{f\in \mathsf{C}([0,\infty]):f(\infty)=0\}\times \mathsf{C}([0,\infty]))$ (cf. Proposition~\ref{prop:existencesolutionMP}).
  \end{definition}
  Notice that $\DD\subset \mathcal{D}_e$.

\begin{proposition}\label{lem:extendedpregen} Let $\psi\in \mathsf{LDS}$. We have $D_{\XX_{e,\psi}}\supset \DD_e$.  Moreover, if in the representation \eqref{def:LDSpsi} of $\psi$ the   L\'evy measure of $\kappa$ has no atom at $-1$, then $(\XX_{e,\psi}f)\vert_{[0,\infty)}\in \mathsf{C}_0([0,\infty))$ for all $f\in \mathcal{D}_{e}$ and therefore $\DD_e\subset D_{\XX_{e,\psi}'}$.
\end{proposition}
 Since $(\XX_{e,\psi}f)(\infty)=0$, the last assertion is equivalent to maintaining that $\XX_{e,\psi}f\in \mathsf{C}([0,\infty])$.
\begin{proof}[Proof of Proposition~\ref{lem:extendedpregen}]
Directly from the definitions it follows that $\mathcal{D}_{e}\subset D_{\XX_{e,\psi}}$. Let now $f\in \mathcal{D}_{e}$ and write out $\XX_{e,\psi}f$ in the representation \eqref{def:LDSpsi} of $\psi$: \small
$$\XX_{e,\psi}f(x)=\hat{\Sigma}(x)f''(x)-\hat{\Psi}(x)f'(x)+x\mathrm{L}^{\Psi}f(x)+x^2\mathrm{L}^{\Sigma}f(x)+\XX_{e,\mathbf{\Sigma}-\mathbf{\Phi}}f(x)+\mathrm{L}_{\kappa}f(x),\quad  x\in [0,\infty).$$ \normalsize
It is not difficult to see that $(\XX_{e,\psi}f)\vert_{[0,\infty)}$ is continuous (the continuity of the various integrals appearing in $\mathrm{L}^\Psi$, $\mathrm{L}^\kappa$ and $\XX_{e,\mathbf{\Sigma}-\mathbf{\Phi}}$ is handled e.g. via dominated convergence, directly by boundedness for integrals on $(1,\infty)$, and via a Taylor expansion on $[-1,1]$). It remains to verify that $(\XX_{e,\psi}f)\vert_{[0,\infty)}$ vanishes at infinity.   To this end we  treat separately the various terms in the display just above. 

\emph{Local parts}. The first two terms, driven by $\hat{\Sigma}$ and $\hat{\Psi}$, are easily studied. From Subsection~\ref{subsection:levy-khintchin} we know that $\limsup_{x\to\infty}\frac{\hat\Sigma(x)+\vert \hat\Psi(x)\vert}{x^2}<\infty$. On using the assumption  $f\in \mathcal{D}_e$ we see at once that in fact $\hat{\Sigma}(x)f''(x)-\hat{\Psi}(x)f'(x)\underset{x\rightarrow \infty}{\longrightarrow} 0$.

\emph{Continuous-state branching processes with collisions (CBCs)}. The second two terms correspond to the extended generators of CBCs. It has been established in \cite[Corollary 3.1]{foucartvidmar} that for $f\in \mathcal{D}_e$, $x^2\mathrm{L}^{\Sigma}f(x)$ and $x\mathrm{L}^{\Psi}f(x)$ both tend to $0$ as $x\to \infty$.

\emph{The term involving $\mathcal{X}_{e,\mathbf{\Sigma}-\mathbf{\Phi}}$}. The L\'evy measures $\nu_x$, $x\in [0,\infty)$, in the associated Courr\`ege form are given by \eqref{eq:nu1_xXXsigmaphi}-\eqref{eq:nu2_xXXsigmaphi} (for notational ease we incorporate killing terms as masses at infinity):\small
\begin{equation*}
\nu_x(\dd u)=\mathbbm{1}_{(0,\infty)}(u)\int(e^{-xv}-1+xv)\bm{\eta}(\dd v, \dd u)+\int(1-e^{-xv}\mathbbm{1}_{(0,\infty)}(v))\bm{\nu}(\dd v, \dd u),\quad u\in (0,\infty].
\end{equation*}\normalsize
The main idea is to estimate the integrands in the preceding integrals and to combine this with estimates,  in terms of the derivatives of $f$, for the expressions  $f(x+u)-f(x)$ and $f(x+ u)-uf'(x)-f(x)$ that appear in the action of $\XX_{e,\mathbf{\Sigma-\Phi}}$ on a map $f\in \DD_e$. Said estimates must be such that ultimately we can appeal to \eqref{eq:nondecomposable-SpCMP-}. In this direction, first, by Taylor expansion or directly we see that there exists a constant $C<\infty$ such that $$e^{-xv}-1+xv\leq Cx^2(v^2\wedge v) \text{ and } 1-e^{-xv}\leq Cx(v\wedge 1),\quad  v\in (0,\infty),\, x\in [1,\infty).$$ Set $$\underline{\bm{\eta}}(\dd u):=\int(v^2\wedge v)\bm{\eta}(\dd v, \dd u),\quad u\in (0,\infty),$$ which, cf. \eqref{eq:nondecomposable-SpCMP-},  satisfies $\int (u^2\land u) \underline{\bm{\eta}}(\dd u)<\infty$, also $$ \underline{\bm\nu}(\dd u):=\int \big(v\wedge 1\big)\bm{\nu}(\dd v, \dd u),\quad u\in (0,\infty],$$ for which, cf. again \eqref{eq:nondecomposable-SpCMP-}, $\int (u\land 1) \underline{\bm\nu}(\dd u)<\infty$.  Then, for $x\in [1,\infty)$,\small
$$ |\XX_{e,\mathbf{\Sigma-\Phi}}f(x)|\leq C\left(x^2 \int \vert f(x+ u)-uf'(x)-f(x)\vert \underline{\bm\eta}(\dd u)+x\int \vert f(x+ u)-f(x)\vert \underline{\bm\nu}(\dd u)\right).$$\normalsize
 In the preceding display the passage to the limit $x\to\infty$ is now evident for the integrals over $\underline{\bm\eta}(\dd u)$ in $u\in (1,\infty)$ and over $\underline{\bm\nu}(\dd u)$ in $u\in (1,\infty]$, whereas on $u\in (0,1]$ we write out  $$f(x+u)-f(x)=\int_x^{x+u}f'(r)\dd r\text{ and } f(x+ u)-uf'(x)-f(x)=\int_x^{x+u}\int_x^rf''(w)\dd w\dd r,$$ and effect some simple triangle inequality estimates first, together with bounds on the derivatives of $f$ coming out of $f\in \mathcal{D}_e$ (we will write out in detail the analogous string of inequalities for the final case to follow presently). In any event we find that $\lim_{x\to\infty}\XX_{e,\mathbf{\Sigma-\Phi}}f(x)=0$ as required.
    \item \emph{Study of $\mathrm{L}_{\kappa}$}. Recall the form of the operator in Definition~\ref{eq:multiplicativelevy}. In the notation thereof decompose $\mathrm{L}_\kappa f=\mathrm{L}_\kappa^{1}f+\mathrm{L}_\kappa^{2}f$ with, fixing $s\in (-1,0)$,
\begin{align*}
\mathrm{L}_\kappa^{1}f(x)&:=\int_{(-1,s)}\big(f(x(1+m))-f(x)-xmf'(x)\big)\nu(\dd m)\\
&\quad +\int_{(1,\infty)}\big(f(x(1+m))-f(x)\big)\nu(\dd m)+x^2f''(x)+{b}xf'(x)- c f(x)
\end{align*}    
and
\begin{equation*}
\mathrm{L}_\kappa^{2}f(x):=\int_{[s,1]}\big(f(x(1+m))-f(x)-xmf'(x)\big)\nu(\dd m)
\end{equation*}    
for $x\in [0,\infty)$ (we have involved here the supposition that $\nu$ has no atom at $-1$). 
By bounded convergence trivially $\mathrm{L}_\kappa^{1}f(x)\underset{x\rightarrow \infty}{\longrightarrow} 0$. On the other hand, 
$$\vert \mathrm{L}_\kappa^{2}f(x)\vert\leq \int_{[s,1]}\int_0^{xm}\int_0^r\vert f''(x+w)\vert\dd w\dd r\nu(\dd m)\leq \frac{x^2\sup_{[(1+s)x,\infty)}\vert f''\vert}{2}\int_{[-1,1]} m^2\nu(\dd m)$$ and evidently this is $\to 0$ as $x\to\infty$ since $f\in \DD_e$.
\end{proof}
We are ready to formulate and prove the main result of this subsection.
\begin{theorem}\label{thm:LaplacedualitysemigroupLDS} Let $\psi\in \mathsf{LDS}$ vanish at zero, i.e.  $\psi(0,\cdot)\equiv 0\equiv \psi(\cdot,0)$, and let it satisfy  \begin{equation}
\sup_{(x,y)\in [0,\infty)^2}\psi_-(x,y)e^{-xy}<\infty.\label{eq:controlmeanC0}
   \end{equation}
Suppose also that  in the representation \eqref{def:LDSpsi} of $\psi$ the  L\'evy measure of $\kappa$ has no atom at $-1$. Then there exist unique-in-law  $\mathbb{D}^{\mathrm{m}\infty}_{[0,\infty]}\cap  \mathbb{D}_{[0,\infty]}^{\mathrm{m}0}$-valued processes, $X$ under the probabilities $(\PP_{x_0})_{x_0\in [0,\infty]}$  and $Y$ under the probabilities  $(\PP^{y_0})_{y_0\in [0,\infty]}$, solving the MPs for $\XX:=\XX_\psi$ and $\YY:=\XX_{\hat\psi}$ respectively. The processes are in fact strong Markov, non-explosive (i.e. conservative) for finite starting values, and solve the MPs  for $\XX':=\XX_{e,\psi}'$ and $\YY':=\XX_{e,\hat\psi}'$ respectively. Lastly, the processes $X$ and $Y$ are in $(0^+\cdot \infty, \infty \cdot 0^+)$-Laplace duality. 
\end{theorem}
A more hands-on condition entailing the vanishing at zero of $\psi$ and \eqref{eq:controlmeanC0} is given in the forthcoming Proposition \ref{prop:momentcond}.

\begin{proof}[Proof of Theorem~\ref{thm:LaplacedualitysemigroupLDS}]
Notice that by Proposition~\ref{lem:extendedpregen} $\XX\subset \XX'$ and $\YY\subset \YY'$, in particular $\DD\subset D_{\XX'}\cap D_{\YY'}$, where we have taken into account the ``no atom at $-1$'' assumption. Let us establish existence of $\mathbb{D}^{\mathrm{m}\infty}_{[0,\infty]}\cap  \mathbb{D}_{[0,\infty]}^{\mathrm{m}0}$-valued solutions to the MP for  $\XX'$. The MP for $\YY'$ is  entirely analogous.

The cases of the starting values $x_0=0$ and $x_0=\infty$ are trivial. Let then $x_0\in (0,\infty)$. According to Proposition~\ref{prop:existencesolutionMP} there exists a $\mathbb{D}_{[0,\infty]}^{\mathrm{m}\infty}$-valued solution $(X',\PP_{x_0})$ to the MP for $(\XX',\delta_{x_0})$. Setting $\sigma_0:=\inf\{t\in (0,\infty):0\in \{X'_t,X_{t-}'\}\}$ it remains to pass from $X'$ to the  process $X$ that concides with $X'$ on $[0,\sigma_0)$ and is identically zero on $[\sigma_0,\infty)$ (the martingales are preserved due to $\phi(0,\cdot)\equiv 0$ and by applying optional stopping -- in the right-continuous modification of the natural filtration of $X$ -- at the times $\sigma_\epsilon^-:=\inf\{t\in [0,\infty):X'_t<\epsilon\}\uparrow \sigma_0$ as $\epsilon\downarrow 0$). 

For $x_0\in [0,\infty)$ let now $(X,\PP_{x_0})$ be any $\mathbb{D}^{\mathrm{m}\infty}_{[0,\infty]}\cap  \mathbb{D}_{[0,\infty]}^{\mathrm{m}0}$-valued solution to the MP for $(\XX,\delta_{x_0})$. Set  $$\zeta_\infty=\inf \{t\in [0,\infty):X_{t-}\lor X_t=\infty\}.$$  By the martingale problem (after optional stopping at $\zeta_\infty$), for all $t\in [0,\infty)$,
 \begin{align*}
 \EE_{x_0}\left[e^{-X_{t\land \zeta_\infty}y}\right]&=e^{-x_0y}+\EE_{x_0}\left[\int_0^{t\land \zeta_\infty}\psi(X_s,y)e^{-X_sy}\dd s\right]\\
 &\geq e^{-x_0y}+\EE_{x_0}\left[\int_0^{t\land \zeta_\infty}\psi_-(X_s,y)e^{-X_sy}\dd s\right],\quad y\in (0,\infty).
 \end{align*}
Passing to the limit $y\downarrow 0$, by monotone (on the l.h.s.) and bounded (on the r.h.s., noting \eqref{eq:controlmeanC0}) convergence we obtain, thanks to  $\psi(\cdot,0)\equiv 0$ and the continuity of $\psi$ at zero, that $\PP_{x_0}(X_{t\land \zeta_\infty}<\infty)\geq 1$ (and therefore $=1$). This being so for all $t\in [0,\infty)$ we find that $X$ is conservative. 

Let next $X$ under the probabilities $(\PP_{x_0})_{x_0\in [0,\infty]}$ be any $\mathbb{D}_{[0,\infty]}^{\mathrm{m}\infty}\cap  \mathbb{D}_{[0,\infty]}^{\mathrm{m}0}$-valued solution to the MP for $\XX$ and let $Y$  under the probabilities $(\PP^{y_0})_{y_0\in [0,\infty]}$ be any $\mathbb{D}_{[0,\infty]}^{\mathrm{m}\infty}\cap  \mathbb{D}_{[0,\infty]}^{\mathrm{m}0}$-valued solution to the MP for $\YY$ (they exist by the above; we want any). 

 We verify the conditions for Theorem~\ref{thm:generators}\ref{thm:generators:II}. Trivially, \eqref{eq:controlmeanC0} ensures that \eqref{eq:weird-conditionLD} holds with $\sigma=\infty$ and $\tau=\tau_n^+\wedge \tau_\epsilon^{-}$ (even with $\tau=\infty$).  For all $y\in [0,\infty)$, the function $[0,\infty)\ni x\mapsto \phi(x,y)=e^{-xy}\psi(x,y)$, cf. \eqref{thm:generators-phi}, is  continuous at $0$ (indeed, everywhere), i.e. $\widehat{\eqref{phi:0}}$ is satisfied. Since $\psi(\cdot,0)\equiv 0$ we get $\widehat{\eqref{phi:variation'}}$. \eqref{phi:variation} would be valid even if we had merely $\psi\in \mathsf{LDS}$. By the preceding both processes $X$ and $Y$ are conservative so certainly $X$ has $\infty$ non-sticky  and certainly there is a.s. no continuous explosion in $Y$. By choice the process $Y$ is strongly absorbed at $0$. Altogether, via Theorem~\ref{thm:generators}\ref{thm:generators:II}, this gives us the 
$(0^+\cdot \infty, \infty \cdot 0^+)$-Laplace duality and thereby uniqueness in law. 

In order to see that $X$ and $Y$ solve even the MPs for $\XX'$ and $\YY'$ respectively,  one has simply to note that any two solutions $X'$ and $Y'$ to $\XX'$ and $\YY'$, valued in  $\mathbb{D}_{[0,\infty]}^{\mathrm{m}\infty}\cap  \mathbb{D}_{[0,\infty]}^{\mathrm{m}0}$ -- which exist by the above -- solve also the MPs for $\XX$ and $\YY$, respectively. By what we have just shown their laws are uniquely determined and must be the same as the laws of the solutions to the MPs for respectively $\XX$ and $\YY$. But solving a martingale problem concerns only the laws of the processes so the matter is thereby settled.

For the strong Markov property we refer to \cite[p.~184, Theorem 4.2(c)]{EthierKurtz}. Its application requires checking the measurability of the map $([0,\infty]\ni x_0\mapsto \PP_{x_0}(X\in A))$ for measurable $A\subset \mathbb{D}_{[0,\infty]}$ (and analogously for the process $Y$). Since $\{\ee^y:y\in (0,\infty)\}$ is a multiplicative class of maps generating the Borel $\sigma$-field on $[0,\infty]$, for $t\in [0,\infty)$, the Laplace duality relation between $X$ and $Y$ and functional monotone class entail that $([0,\infty]\ni x_0\mapsto \EE_{x_0}[f(X_t)])$ is measurable for all Borel $f:[0,\infty]\to [0,\infty]$. By induction and the Markov property of $X$ -- which is ensured by \cite[p.~184, Theorem 4.2(a)]{EthierKurtz} -- we find that $([0,\infty]\ni x_0\mapsto \EE_{x_0}[f_1(X_{t_1})\cdots f_n(X_{t_n})])$ is measurable for all real $0\leq t_1<\cdots<t_n$, all Borel $f_i:[0,\infty]\to [0,\infty]$ for $i\in \{1,\ldots,n\}$, all $n\in \mathbb{N}$. A final application of Dynkin's lemma \cite[Theorem~2.38]{pollard} concludes the measurability for arbitrary measurable $A$ (since the $\sigma$-field on $\mathbb{D}_{[0,\infty]}$ is generated by the coordinate process).
\end{proof}
\begin{proposition}\label{prop:momentcond}
Let $\psi\in \mathsf{LDS}$ with the representation \eqref{def:LDSpsi}. Assume that 
$$\Psi(0)=0=\hat{\Psi}(0)\text{  and }\mathbf{\Phi}(0,\cdot)\equiv 0\equiv \mathbf{\Phi}(\cdot,0)$$ (which implies that $\psi$ vanishes at $0$) and that 
   \begin{equation*}\label{eq:boundforallbutkappa}|\Psi'(0)|+|\hat{\Psi}'(0)|+\frac{\partial^2\mathbf{\Phi} }{\partial x \partial y}(0,0)<\infty.\end{equation*}
Then \eqref{eq:controlmeanC0} is satisfied. 
\end{proposition}
Notice that we have had to impose no conditions on the  $\Sigma$, $\hat\Sigma$, $\mathbf{\Sigma}$ and $\kappa$ of \eqref{def:LDSpsi} in order for the conclusion to prevail. 
\begin{proof}[Proof of Proposition~\ref{prop:momentcond}]
Rewrite  \eqref{def:LDSpsi}:\small \begin{equation}\label{orderedLDS}\psi(x,y)=\kappa(xy)+x\Psi(y)+\hat{\Psi}(x)y-\mathbf{\Phi}(x,y)+x^2\Sigma(y)+\mathbf{\Sigma}(x,y)+\hat{\Sigma}(x)y^2, \ \{x,y\}\subset [0,\infty).
\end{equation}\normalsize
In verifying \eqref{eq:controlmeanC0} we may estimate the negative part of each term in the preceding display separately. As far as $\kappa$ is concerned it is almost trivial because even $([0,\infty)\ni z\mapsto \kappa(z)e^{-z})$ is bounded. For the second two terms: $\Psi_-$ and $\hat\Psi_-$ are bounded by a linear (not just affine!) function on $[0,\infty)$ ($\because$ $\Psi(0)=0=\hat{\Psi}(0)$, $|\Psi'(0)|+|\hat{\Psi}'(0)|<\infty$); it remains to note that $([0,\infty)\ni z\mapsto ze^{-z})$ is bounded. Recalling \eqref{symbolxyPhi}, since $\mathbf{\Phi}(0,\cdot)\equiv 0\equiv \mathbf{\Phi}(\cdot,0)$, the measure $\bm\nu$ charges only $(0,\infty)^2$;  using the elementary estimate $1-e^{-z}\leq z$ for $z\in [0,\infty)$ we then find that $$\mathbf{\Phi}(x,y)\leq xy \int vu \bm\nu(\dd v,\dd u),\quad (x,y)\subset [0,\infty);$$ together with $\int vu \bm\nu(\dd v,\dd u)<\infty$ ($\because$ $\frac{\partial^2\mathbf{\Phi} }{\partial x \partial y}(0,0)<\infty$) it is enough to handle the term involving $\mathbf{\Phi}$. For the remaining terms the negative parts vanish and we are done.
\end{proof}
A straightforward application of Theorem \ref{thm:LaplacedualitysemigroupLDS} and Proposition~\ref{prop:momentcond} yields the following family, which should be compared with Theorem~\ref{theorem:tensor-form}\ref{clement:i}. 
\begin{example}\label{exampleSigmaSigma} Take $\{\Sigma,\hat\Sigma\}\subset \spLpnot$ and set 
\begin{equation*}
\psi(x,y):=\mathbbm{1}_{[0,\infty)^2}(x,y)\hat{\Sigma}(x)\Sigma(y),\quad (x,y)\in [0,\infty]^2.
\end{equation*}
Then there exist unique-in-law $\mathbb{D}^{\mathrm{m}\infty}_{[0,\infty]}\cap  \mathbb{D}_{[0,\infty]}^{\mathrm{m}0}$-valued  Markov processes $X$ and $Y$ whose Laplace symbols are 
$\psi_X=\psi\vert_{[0,\infty]\times (0,\infty)}$ and $\psi_Y=\hat{\psi}\vert_{[0,\infty]\times (0,\infty)}$ respectively, and these processes are  in $(0^+\cdot \infty, \infty \cdot 0^+)$-Laplace duality. 
\end{example}
Examples~\ref{example:dual-subo} and~\ref{exampleSigmaSigma} establish that all processes with  symbols $\psi\vert_{[0,\infty] \times (0,\infty)}$ and $\hat{\psi}\vert_{[0,\infty] \times (0,\infty)}$, where $\psi\in \mathsf{LDS}$ is simple as per Definition~\ref{def:LDS},  exist and lie in Laplace duality. Combining them leads to decomposable Laplace dual symbols. By Theorem~\ref{thm:LaplacedualitysemigroupLDS}, provided the conditions set forth therein are met, there exist processes with such decomposable Laplace dual symbols and they too exhibit Laplace duality. A relatively wide subclass of these processes will be studied via stochastic equations in the forthcoming Subsection~\ref{sec:decomposable}, where we will also relax condition~\eqref{eq:controlmeanC0}. 
\begin{example} 
The Laplace dual symbol  obtained by a mixture of Gamma subordinators of Example~\ref{ex:mixedgamma}, $$\psi(x,y)=\mathbbm{1}_{[0,\infty)^2}(x,y)\int_0^1\log(1+xr)\log(1+y/r)r^{\gamma}\dd r, \quad \{x,y\}\subset [0,\infty]^2,$$ 
satisfies the conditions of   Proposition~\ref{prop:momentcond} due to \eqref{eq:momentcondgamma}. Theorem \ref{thm:LaplacedualitysemigroupLDS} in turn guarantees existence and uniqueness in law of $\mathbb{D}^{\mathrm{m}\infty}_{[0,\infty]}\cap  \mathbb{D}_{[0,\infty]}^{\mathrm{m}0}$-valued Markov processes $X$ and $Y$ whose Laplace symbols are 
$\psi_X=\psi\vert_{[0,\infty]\times (0,\infty)}$ and $\psi_Y=\hat{\psi}\vert_{[0,\infty]\times (0,\infty)}$ respectively, and which are in $(0^+\cdot \infty, \infty \cdot 0^+)$-Laplace duality. 
\end{example}

\section{Construction of examples via stochastic equations}\label{sec:examples}
In the present section the processes will be constructed as solutions to stochastic equations driven by specified Poisson random measure (PRM) and Brownian motions on a common filtered probability space.

 We shall restrict attention to  equations for which strong existence and pathwise uniqueness  follow from standard techniques, making automatic the Markov property of the solutions. Particularly well-suited to our context in this regard  is Proposition 1 of Palau and Pardo \cite{zbMATH06836271}, which, building on earlier work by Li and Pu \cite[Theorem 6.1]{zbMATH06098183}, establishes sufficient conditions -- allowing for jump measures with infinite mean -- for the unique existence of a solution that may exhibit explosion.  Except for the setting of non-local decomposable symbols discussed in Subsection \ref{sec:decomposable}, where we duly check for existence and uniqueness, we shall then essentially take for granted that there exist unique (possibly explosive) solutions to the forthcoming jump SDEs. Meticulous explicit checks of these unique existence results are not our focus here, but  references in which such equations are studied will be provided.  

The stochastic equations' approach is especially indicated because it provides us, via It\^o's lemma for discontinuous semimartingales for which see e.g. Jacod and Shyriaev's book \cite[p.~76, Theorem~4.57 and p.~555, Remark~4.5]{jacod1987limit},  with solutions to MPs, and this in turn allows us to apply Theorem~\ref{thm:generators}. 
Moreover, such approach facilitates the identification of processes previously examined in the literature as completely monotone Markov processes with explicit Laplace dual symbols. 

Below, when working with solutions $X$ with variable initial value ${x_0}\in [0,\infty]$, under a single probability $\PP$, then in considering the Laplace duality, strictly speaking, we should pass to the laws $\PP_{x_0}$ of $X$ under $\PP$ when  started from ${x_0}$; 
instead, by a relatively standard abuse of notation, we let the $x_0$ in the expectation operator $\EE_{x_0}$ (or in the  probability $\PP_{x_0}$) -- and sometimes, in case there is risk of ambiguity, in the superscript of the process $X^{x_0}$ -- indicate the initial value. Similarly for the $Y$ processes. In addition, implicitly we insist all the solutions to be $\mathbb{D}_{[0,\infty]}^{\mathrm{m}\infty}$-valued (with certainty, not just a.s.).

The comparison property between solutions with respect to their initial values is also going to be satisfied, in the sense that for all starting values $x_0\leq x_1$ we will have $X^{x_0}_t\leq X^{x_1}_t$ for all $t\in [0,\infty)$ a.s., see e.g. Fu and Li \cite[Theorem 5.5]{zbMATH05676165} for a condition ensuring  this, which will always be met in the equations to be studied. In particular, the processes we will work with  will all be stochastically monotone. 

Recall  from \eqref{the-exponential-domain} that we denote $\mathcal{D}=\mathrm{lin}(\{\ee^y, \ y\in (0,\infty)\})$, also the various classes of L\'evy-Khintchine forms of Table~\ref{table.3}.  For $s\leq t$ from $[0,\infty]$, an integral of the form $\int_s^t\ldots$ is interpreted to be so over the interval $(s,t]\cap (0,\infty)$.

\subsection{Generalisations of continuous-state branching processes}\label{sec:genCB}
We look first at Laplace dual symbols $\psi$ that satisfy
\begin{equation}\label{symbolgencb}
\psi(x,y)=\mathbbm{1}_{[0,\infty)^2}(x,y)\left(x\Psi(y)+x^2\Sigma(y)-\Phi(y)+\kappa(xy)\right),\quad (x,y)\in [0,\infty]^2.
\end{equation} 
Each term in the expression above corresponds to specific population dynamics that has been studied in the literature. A process with Laplace symbol $\psi$ would indeed have a branching part with mechanism $\Psi\in \pkspLp$, a collision part driven by $\Sigma\in \spLpnot$, an immigration encoded by $\Phi\in \pksubo$ and a random environment governed by $\kappa\in \pkLp_1$.

Aside from the CB and CBI processes, which have received considerable attention, many fundamental questions -- such as the classification of boundaries, the existence of Markovian extensions, and the characterization of stationary distributions -- remain largely unexplored for such ``enhanced'' CB processes. We do not undertake a detailed investigation of these issues here, as a case-by-case analysis may be more appropriate/feasible. However, the duality relationship and its implications -- for instance, Corollary~\ref{cor:absorbing-nonsticky} -- provide a natural starting point for addressing such questions, see for instance \cite{foucartvidmar,MR3940763} for works along such lines. 
 
Instead of constructing the entire class of processes with Laplace symbols corresponding to \eqref{symbolgencb} at once -- subject to appropriate boundary conditions -- we will describe several subcases, either previously defined in the literature or straightforward generalizations thereof.

\subsubsection{CBs}\label{para:cb-processes}
We start with basic continuous-state branching processes. As already pointed out in the Introduction, they are the only positive Markov processes in Laplace duality with a  deterministic process. Theorem~\ref{thm.cbs-classical} below, which should be appreciated in conjunction with Theorem~\ref{theorem:tensor-form}\ref{clement:iii} and Example~\ref{example:cb}, is well-known. It naturally arises as a specific case within our framework and will be established using our methodology.

Let $\Psi\in \pkspLp$ be a branching mechanism with L\'evy-quadruplet $(\nu,a,b,c)$:
\[\Psi(y)=\int \big(e^{-uy}-1+uy\mathbbm{1}_{(0,1]}(u)\big)\nu(\dd u)+ay^2-by-c,\quad y\in [0,\infty),\]
a piece of notation that is to remain in effect throughout the remainder of this subsection. For the purposes of the present Paragraph~\ref{para:cb-processes} only, introduce the Laplace dual symbol $\psi\in \mathsf{LDS}$ by asking that
\begin{equation*}
\psi(x,y):=\mathbbm{1}_{[0,\infty)^2}(x,y)x\Psi(y), \quad (x,y)\in [0,\infty]^2.
\end{equation*}

Consider the stochastic integral equation for the process $\tilde{X}$, $x_0\in [0,\infty)$ indicating the initial value,
\begin{equation}
\label{SDECB}
\begin{split}
\tilde{X}_t
= x_0
&+ \int_0^t \sqrt{2a\tilde{X}_s}\,\dd B_s
+ b\int_0^t \tilde{X}_s\,\dd s \\
&+ \int_0^t \int_{0}^{\tilde{X}_{s-}} \int_{0}^1
u\,\bar{\mathcal{N}}(\dd s,\dd r,\dd u)
+ \int_0^t \int_{0}^{\tilde{X}_{s-}} \int_{1}^{\infty}
u\,\mathcal{N}(\dd s,\dd r,\dd u)
\end{split}
\end{equation}
for $t\in [0,\tilde\zeta)$, with $\tilde \zeta=\inf\{t>0: \tilde{X}_{t-}\text{ or } \tilde{X}_t \notin [0,\infty)\}$ the lifetime, $B$ a Brownian motion, $\mathcal{N}(\dd s,\dd r, \dd u)$  a PRM with intensity $\dd s \dd r \nu(\dd u)$, $\bar{\mathcal{N}}(\dd s,\dd r, \dd u):=\mathcal{N}(\dd s,\dd r, \dd u)-\dd s \dd r \nu(\dd u)$ its compensated version\footnote{Here and below PRM with Lebesgue intensities are taken to live on $(0,\infty)$ in the factor corresponding to that intensity. So, to be precise, $\mathcal{N}$ is a PRM on $((0,\infty)^3,\BB_{(0,\infty)^3})$ with intensity  $ \mathscr{L}^2\times\nu$, where $\mathscr{L}$ is Lebesgue measure on $((0,\infty),\BB_{(0,\infty)})$.}. 
It admits a unique strong solution $\tilde{X}$, see e.g. Fu and Li \cite{zbMATH05676165} and Ji and Li \cite[Theorem 3.1]{zbMATH07189529}, that we set equal to $\infty$ at and after $\tilde \zeta$ (to be implicitly understood henceforth without emphasis), and which is a CB process with branching mechanism $\tilde{\Psi}:=\Psi+c$. The process $X$ that results from $\tilde{X}$ by killing (sending it to $\infty$) at the random time $\zeta:=\inf\{t\in [0,\infty): \int_{0}^{t}c\tilde X_s\dd s> \mathbbm{e}\}$, with $\mathbbm{e}$ an independent mean one exponential random variable, is a CB process with mechanism $\Psi$ and initial value $x_0$, cf. Proposition~\ref{lemma:reductions}. For $x_0=\infty$ we consider $X$ as being identically equal to $\infty$.

The deterministic process $Y$ with initial value $y_0$ is introduced for $y_0\in (0,\infty)$ as the unique locally bounded measurable  solution to the integral equation \begin{equation}\label{eq:odeY}
Y_t=y_0-\int_{0}^{t}\Psi(Y_s)\dd s, \ \ t\in [0,\infty);\end{equation} for $y_0=\infty$ we take it identically equal to $\infty$; finally, for $y_0=0$, we define it by taking the pointwise $\downarrow$ limit of the processes $Y^{y_0}$ as $y_0\downarrow 0$. The fact that there exists a unique solution to the integral equation \eqref{eq:odeY} and that the latter stays bounded away from $0$ and $\infty$ has been established in \cite[Section 4]{zbMATH03294035}; we do not repeat this here.

\begin{theorem}\label{thm.cbs-classical}
We have  $\psi_X=\psi\vert_{[0,\infty]\times (0,\infty)}$ and $\psi_Y=\hat\psi\vert_{[0,\infty]\times (0,\infty)}$. Besides, $X$ is in $(0^+\cdot \infty, \infty \cdot 0^+)$-Laplace duality with $Y$.
\end{theorem}
\begin{proof}
Work under $0^+\cdot \infty$, $\infty \cdot 0^+$.  

In order to establish Laplace duality we are going to use up Theorem~\ref{thm:generators}\ref{thm:generators:0}, checking most of the first claim en passant as well. 
Fix the initial values $x_0\in [0,\infty]$ and $y_0\in [0,\infty]$.

Because $Y$ is defined for starting value $y_0=0$ by taking limits, in verifying that 
$$\EE_{x_0}[e^{-X_t{y_0}}]=\EE^{y_0}[e^{-{x_0}Y_t}],\quad t\in [0,\infty),$$ we may assume that ${y_0}>0$. The cases when $x_0=\infty$ and/or $y_0=\infty$ are also trivial, so we may as well and do take $x_0$ and $y_0$ finite. 

We first study the deterministic process $Y$. The observation that $Y$ solves the martingale problem for $(\YY,\delta_{y_0})$  follows by a straightforward integration by parts. In particular, for $x\in (0,\infty)$, $$\lim_{t\downarrow 0}\frac{\EE^{y_0}[e^{-xY_t} ]-e^{-xy_0}}{t}=\lim_{t\downarrow 0}\frac{x}{t}\int_0^t\Psi(Y_s)e^{-xY_s}\dd s=x\Psi(y_0)e^{-xy_0}=\psi(x,y_0)e^{-xy_0}.$$ It is also clear that $Y$ has $0$ non-sticky. Besides, since $Y$ does not reach $\infty$ and since $\phi(0,\cdot)=\psi(0,\cdot)=0$ identically ($\phi$ as in \eqref{thm:generators-phi}), the process $(\mathbbm{1}_{\{Y_t<\infty\}}-\int_0^t\phi(0,Y_s)\dd s)_{t\in [0,\infty)}$ of Theorem~\ref{thm:generators}\ref{thm:generators:0} reduces to the constant $1$, which is trivially a martingale. (However, by contrast, the process $(\mathbbm{1}_{\{X_t<\infty\}}-\int_0^t\phi(X_s,0)\dd s)_{t\in [0,\infty)}=(\mathbbm{1}_{\{X_t<\infty\}}-\Psi(0)\int_0^tX_s\dd s)_{t\in [0,\infty)}$ is not always a martingale, e.g. we may have $\Psi(0)=0$ with $X$ explosive.)

Notice next that  $\mathcal{X}\ee^{y}$ is bounded for all $y\in (0,\infty)$. Since the process $\tilde X$ is a  solution to \eqref{SDECB}, by It\^o's lemma and Proposition~\ref{lemma:reductions},  we see that $X$ solves the martingale problem for $(\XX,\delta_{x_0})$. In particular, for $ y\in (0,\infty)$, $$\lim_{t\downarrow 0}\frac{\EE_{x_0}[e^{-X_ty} ]-e^{-x_0y}}{t}=\lim_{t\downarrow 0}\frac{\Psi(y)}{t}\!\int_0^t\!\!\!\EE_{x_0}\left[X_se^{-X_sy}\right]\dd s=x_0\Psi(y)e^{-x_0y}=\psi(x_0,y)e^{-x_0y}.$$

Finally we check \eqref{eq:weird-conditionLD} with $\sigma\equiv \infty\equiv \tau$. Using the bound $xye^{-xy}\leq 1$ for $(x,y)\in [0,\infty]\times (0,\infty)$, we see that
\[\int_0^T\int_0^T|X_s\Psi(Y_t)|e^{-X_sY_t}\dd s \dd t\leq T\int_0^{T}\frac{\vert\Psi(Y_t)|}{Y_t}\dd t,\quad T\in [0,\infty).\]
But $(0,\infty)\ni y\mapsto \frac{\vert \Psi(y)|}{y}$ is locally bounded and we are done.

Thus all the conditions of Theorem \ref{thm:generators}\ref{thm:generators:0} have been verified and the proof of the Laplace duality is thereby completed.

As for the statement concerning the generators, the only non-trivial computation that is yet to be performed is 
\begin{align*}
\lim_{t\downarrow 0}\frac{\EE^0[e^{-xY_t} ]-e^{-x\cdot 0}}{t}&=\lim_{t\downarrow 0}\lim_{y_0\downarrow 0}\frac{e^{-xY_t^{y_0}} -e^{-xy_0}}{t}=\lim_{t\downarrow 0}\lim_{y_0\downarrow 0}\frac{x}{t}\int_0^t\Psi(Y_s^{y_0})e^{-xY_s^{y_0}}\dd s\\
&=\lim_{t\downarrow 0}\frac{x}{t}\int_0^t\Psi(Y_s^0)e^{-xY_s^0}\dd s=x\Psi(0)=\psi(x,0)e^{-x\cdot 0},\quad x\in (0,\infty), 
\end{align*}
on noting that $Y^0$ is continuous at zero (among more elementary observations), which follows e.g. from the duality relation.
\end{proof}
\subsubsection{LCBs}\label{subsec:LCB} Continuous-state branching processes with quadratic competition, also called logistic CB (LCB) processes, were introduced by Lambert \cite{Lambert}. These processes can, for initial values $x_0\in [0,\infty)$, be seen as solutions to the SDE \eqref{SDECB} with an extra negative quadratic drift term $-\int_0^{t}\hat{a}\tilde X_s^2\dd s$, where $2\hat{a}\in (0,\infty)$ is the competition parameter. The extraneous killing, involving the parameter $c\in [0,\infty)$, is added as in the previous subsection, to get the process $X$ from $\tilde X$. It has $$\psi_X(x,y)=\mathbbm{1}_{[0,\infty)}(x)\left(x\Psi(y)+\hat{a}x^2y\right),\quad (x,y)\in [0,\infty]\times (0,\infty),$$ as its Laplace symbol. One still calls  $\Psi$  the branching mechanism of the LCB process $X$.

The dual process $Y$ is, for initial value $y_0\in [0,\infty)$, the unique strong non-explosive absorbed-at-zero solution to the stochastic equation, 
\[Y_t=y_0+\int_0^t\sqrt{2\hat{a}Y_s}\dd W_s-\int_0^t\Psi(Y_s)\dd s,\quad t\in [0,\infty),\]
where $W$ is a Brownian motion. For $y_0=\infty$ the process $Y$ is just identically equal to $\infty$.

Actually, for the Laplace dualities to emerge in their full form one is led to extend the processes $X$ and $Y$ at $\infty$ and $0$, respectively, in a suitable way. It means here also extension of paths after hitting these two points, not just changing $\PP_\infty$ and $\PP^0$. We refer to \cite[esp. Theorems~3.3-3.5]{MR3940763} for the details of these non-trivial extensions and recast below, in our framework, only the end-result findings. Noting that in  \cite{MR3940763} the competition parameter $2\hat{a}$ and the killing term $c$ are denoted by $c$ and $\lambda$ respectively, put  then  
 $$\mathcal{E}:=\int_0^{\varepsilon}\frac{1}{x}\exp\left(\int_x^{\varepsilon}\frac{\Psi(u)}{\hat{a}u}\dd u\right)\dd x$$ for an arbitrary fixed $\varepsilon\in (0,\infty)$ \cite[Theorem~3.1]{MR3940763}.  The quantity depends on the choice of $\varepsilon$, but whether or not it is finite does not, and it is only this property that we refer to in the statement to follow.   We also use up Feller's terminology on the classification of boundaries as exit, entrance etc., see e.g. Karlin and Taylor's book \cite[Chapter 15, Section 6]{zbMATH03736679}, in presenting  (without proof) the main findings of \cite[esp. Theorems~3.3-3.5]{MR3940763}:
\begin{theorem}\label{theorem:clement-lcb}
One has the following duality assertions.
 \begin{enumerate}[(i)]
     \item  Assume $\mathcal{E}=\infty$. There exists a unique Feller extension $X'$ of $X$ at $\infty$ such that $\infty$ is an entrance boundary and such that $X'$ in  $(0^+\cdot\infty,\infty^-\cdot 0)$-Laplace duality with $Y$.
 \item Assume $\mathcal{E}<\infty$ and $\frac{c}{\hat{a}}<1$. There exists a unique Feller extension $X'$ of $X$ at $\infty$ such that $\infty$ is a regular instantaneous reflecting boundary and such that $X'$ is in $(0^+\cdot\infty,\infty^-\cdot 0)$-Laplace duality with $Y$. 
\item Assume $\frac{c}{\hat{a}}\geq 1$ (automatically  $\mathcal{E}<\infty$).  There exist a unique Feller extension $X'$ of $X$ at $\infty$ with $\infty$ an exit boundary, and a unique extension $Y'$ of  $Y$ at $0$ with $0$ an entrance boundary, such that $X'$ is in $(0\cdot\infty^-,\infty\cdot 0^+)$-Laplace duality with  $Y'$. \qed
 \end{enumerate}
\end{theorem}

 More general drifts than the negative quadratic term have been considered in the literature.  We refer for instance to the works of Le and Pardoux \cite{zbMATH07553689}, Li et al. \cite{zbMATH07120715} and Rebotier \cite{rebotier} for studies of the solution to the stochastic equation \eqref{SDECB} incorporating an additional nonlinear drift term $-\int_0^tI(\tilde X_s)\dd s$. They are referred to in the literature as CB processes with \textit{drift-interaction}. When $I=\hat{\Psi}\in\pkspLp$, the CB with drift-interaction corresponds to a Laplace dual symbol $\psi\in \mathsf{LDS}$ having $$\psi(x,y)=x\Psi(y)+y\hat{\Psi}(x), \quad (x,y)\in [0,\infty)^2.$$ The LCB case is recovered with $\hat{\Psi}(x)=\hat{a}x^2$. The setting with a general $\hat{\Psi}\in \pkspLp$ will be studied in Foucart and Rebotier \cite{foucartrebotier}.  

\subsubsection{CBCs and CBCIs}\label{subsec:CBC-CBCI} We now describe a class of processes in which the quadratic competition term is generalized to a dynamics of  \textit{pairwise collisions} \cite{foucartvidmar}. 

Consider the stochastic equation \eqref{SDECB} with the following additional terms (we drop $\tilde{\phantom{x}}$ from $\tilde X$, because, for consistency with \cite{foucartvidmar}, we shall take $c=0$ here):
\begin{align}\label{eq:collision}
    \int_{0}^{t}\sqrt{2a'}X_s\dd B_s'-\int_{0}^{t}b'X_s^{2}\dd s +\int_{0}^{t}\int_{0}^{X_{s-}}\int_{0}^{X_{s-}}\int_{0}^{\infty}u\bar{\mathcal{M}}(\dd s, \dd r_1,\dd r_2, \dd u),
\end{align}
where $\{a',b'\}\subset[ 0,\infty)$, $B'$ is a Brownian motion and $\mathcal{M}(\dd s, \dd r_1,\dd r_2, \dd u)$ is a PRM with intensity $\dd s\,\dd r_1\,\dd r_2\, \nu'(\dd u)$,  $\nu'$ a L\'evy measure on $\BB_{(0,\infty)}$. The stochastic drivers $B,B',\mathcal{N},\mathcal{M}$ are assumed to be mutually independent. The process $X$, solution to the SDE ``\eqref{SDECB}$+$\eqref{eq:collision}'', is called a continuous-state branching process with \textit{collisions} (CBC), branching mechanism $\Psi$ and collision mechanism $\Sigma\in \spLpnot$, 
\[\Sigma(y):=\int \big(e^{-uy}-1+uy\big)\nu'(\dd u)+a'y^2+b'y, \quad y\in [0,\infty).\] For $x_0=\infty$ we insist that $X\equiv\infty$.
Existence and uniqueness (until the first explosion time) of the solution to \eqref{SDECB}+\eqref{eq:collision} was shown in \cite{foucartvidmar}. Furthermore, \cite[Corollary 3.1]{foucartvidmar} ensures that $\mathcal{D}=\mathrm{lin}(\{\ee^y:y\in (0,\infty)\})$ lies in the domain of the generator of $X$ and \cite[Eq.~(2.18)]{foucartvidmar} that  $\psi_X=\psi\vert_{ [0,\infty]\times (0,\infty)}$, where $\psi\in \mathsf{LDS}$ is given by
\begin{equation*}
\psi(x,y)=\mathbbm{1}_{[0,\infty)^2}(x,y)\left(x\Psi(y)+x^2\Sigma(y)\right), \quad (x,y)\in [0,\infty]^2.
\end{equation*}
The following duality result  was  established in \cite[esp. Proposition 2.18]{foucartvidmar} and will not be proved here.

\begin{proposition}\label{prop:laplacedualityCBC}
If $X$ does not explode, then $X$ is in   $(0^{+}\cdot \infty, \infty\cdot 0^+)$-Laplace duality with  the diffusion process $Y$, whose Laplace symbol is $\hat{\psi}\vert_{[0,\infty]\times (0,\infty)}$, and which is the unique non-explosive absorbed-at-zero strong solution to the SDE, 
$$Y_t=y_0+\int_0^t\sqrt{2\Sigma(Y_s)}\dd W_s-\int_0^t\Psi(Y_s)\dd s,\quad t\in [0,\infty),$$ $y_0\in [0,\infty)$ being the initial value,  $W$ a Brownian motion (for $y_0=\infty$ we take $Y\equiv \infty$).\qed
\end{proposition}

A sufficient condition for non-explosion is provided in the forthcoming Proposition~\ref{prop:nonexplosionCBCI}. 

We now add an immigration dynamics to the setting with branching and collision. Let then $\Psi \in \pkspLp$, $\Sigma \in \spLpnot$ (as above) and $\Phi \in \pksubo$ denote the branching, collision, and immigration mechanisms, respectively.  Define the Laplace dual symbol $\psi\in \mathsf{LDS}$ by asking that 
\begin{equation}\label{symbolCBCI}
\psi(x, y) := \mathbbm{1}_{[0,\infty)^2}(x,y)\left(x \Psi(y)+x^2 \Sigma(y)  - \Phi(y)\right), \quad (x,y)\in [0,\infty]^2,
\end{equation}
which supersedes the $\psi$ above for the remainder of Paragraph~\ref{subsec:CBC-CBCI}.
For ease of exposition, we focus on the setting with no killing terms: $\Psi(0)=0=\Phi(0)$. Denote by $\nu''$ and $b''$ the L\'evy measure and the drift parameter of $\Phi$ so that $$\Phi(y)=\int (1-e^{-uy})\nu''(\dd u)+b''y,\quad y\in [0,\infty),$$ and introduce the stochastic equation:
\begin{equation}
\label{SDECBCI}
\begin{split}
X_t = x_0
&+ \int_0^t \sqrt{2aX_s}\,\dd B_s
+ b\int_0^t X_s\,\dd s
+ \int_0^t \int_0^{X_{s-}} \int_0^1
u\,\bar{\mathcal{N}}(\dd s,\dd r,\dd u)
\\
&+ \int_0^t \int_0^{X_{s-}} \int_1^{\infty}
u\,\mathcal{N}(\dd s,\dd r,\dd u)
+ \int_0^t \sqrt{2a'}\,X_s\,\dd B'_s
\\
&- b'\int_0^t X_s^2\,\dd s
+ \int_0^t \int_0^{X_{s-}} \int_0^{X_{s-}} \int_0^{\infty}
u\,\bar{\mathcal{M}}(\dd s,\dd r_1,\dd r_2,\dd u)
\\
&+ \int_0^t \int_0^{\infty}
u\,\mathcal{I}(\dd s,\dd u)
+ b''\,t,
\qquad t\in[0,\zeta),
\end{split}
\end{equation}where, in addition to $B,B',\mathcal{N},\mathcal{M}$, which are as before, $\mathcal{I}(\dd s,\dd u)$ is an independent PRM with intensity  $\dd s\nu''(\dd u)$ and $\zeta$ is the explosion time. 

The addition of a subordinator-type term to the stochastic equation -- independent of the state of the process (the last two terms in \eqref{SDECBCI}) -- does not pose any significant difficulty in verifying, for example, Palau and Pardo's conditions  for strong existence and uniqueness up to the first explosion time. We refer to the resulting solution (and its law) as the $\mathrm{CBCI}(\Sigma, \Psi, \Phi)$ process. The existence and uniqueness  for more general non-local decomposable processes is addressed in the next section and partially covers the CBCI processes. Our focus below is on establishing a duality result.

Anticipating the dual process, introduce   $Y$ exactly as in Proposition~\ref{prop:laplacedualityCBC} but additionally we kill and send it to $\infty$ at rate $\Phi$, i.e. $Y$ is relegated to the cemetry $\infty$ at the time $\hat\zeta:=\inf\{t\in [0,\infty):\int_0^t\Phi(Y_s)\dd s>\hat{\mathbbm{e}}\}$, where $\hat{\mathbbm{e}}$ is a mean one exponential random variable independent of $W$. 
 
\begin{theorem}\label{thm:LaplacedualityCBCI}
The  generators of the Markov processes $X$ and $Y$ contain in their domain  $\mathcal{D}$, moreover,   $\psi_X=\psi\vert_{[0,\infty]\times (0,\infty)}$ and $\psi_Y=\hat{\psi}\vert_{[0,\infty]\times (0,\infty)}$ (where $\psi$ is that of \eqref{symbolCBCI}). If $X$ does not explode, then $X$ is in $(0^+\cdot \infty,\infty\cdot 0^+)$-Laplace duality with  $Y$. 
\end{theorem}
Theorem~\ref{thm:LaplacedualityCBCI} allows for the relatively degenerate case (from the point of view of CBCIs) of $X$ being a subordinator with Laplace exponent $\Phi$; that is to say, we may have $\Psi \equiv 0$, $\Sigma \equiv 0$. In this case, the dual process $Y$ remains at its initial value $y_0$ until a single jump from $y_0$ to $\infty$, occurring at rate $\Phi(y_0)$. Compare with the first bullet point of Example~\ref{example:subo}. 
\begin{proof}[Proof of Theorem~\ref{thm:LaplacedualityCBCI}]
Denote by $\XX$ (resp. $\YY$) the pregenerator associated  to  $\psi\vert_{[0,\infty]\times (0,\infty)}$ (resp.  to $\hat\psi\vert_{[0,\infty]\times (0,\infty)}$). By It\^o's lemma we find that each $\left(f(X_t)-\int_{0}^{t}\mathcal{X}f(X_s)\dd s\right)_{ t\in [0,\zeta)}$, for $f\in \DD$, is a local martingale (when $x_0<\infty$). Boundedness of the integrands and $\psi(\infty,\cdot)\equiv 0$  ensure by a simple argument that the processes $(f(X_t)-\int_{0}^{t}\mathcal{X}f(X_s)\dd s)_{ t\in [0,\infty)}$,   $f\in \mathcal{D}$, are in fact martingales (even under $\PP_\infty$),  and also that $\XX$ acts on $\mathcal{D}$ as the generator of $X$. Similarly, but using also Proposition~\ref{lemma:reductions} to account for the killing in $Y$, we check that that the processes
$ (f(Y_t)-\int_{0}^{t}\mathcal{Y}f(Y_s)\dd s)_{t\in [0,\infty)}$, $f\in \DD$, are martingales and that the Laplace symbol of $Y$ is $\hat{\psi}\vert_{[0,\infty]\times (0,\infty)}$.

We now seek to establish the Laplace duality and work under $0^+\cdot \infty$, $\infty \cdot 0^+$. We are going to appeal to Theorem \ref{thm:generators}\ref{thm:generators:II}, whose notation we adopt. Fix the initial values $x_0\in [0,\infty)$ of $X$ and $y_0\in [0,\infty)$ of $Y$ (the outstanding  duality identities for $x_0= \infty$ or $y_0=\infty$ are trivial). 

By the preceding discussion  the process $Y$ is a solution to the MP for ($\mathcal{Y}, \delta_{y_0}$) and $X$ is a solution to the MP for $(\mathcal{X},\delta_{x_0})$. Recall next from Proposition \ref{prop:laplacedualityCBC} that a.s. $Y$ is absorbed at its boundary $0$,  has continuous sample paths (up to $\hat\zeta$) and that there is no continuous explosion for $Y$  \cite[Lemma 6.1]{foucartvidmar}.  Furthermore, by assumption $\Psi(0)=\Phi(0)=0$; also  $\Sigma(0)=0$. Hence $\psi(\cdot,0)\equiv 0$ and therefore $\widehat{\eqref{phi:variation'}}$ is met.  Since $\psi$ is locally bounded (as is any member of $\mathsf{LDS}$) $\eqref{phi:variation}$ also follows. By assumption $X$ does not explode, a fortiori $\infty$ is non-sticky for it.   $\widehat{\eqref{phi:0}}$ holds because $\psi$, like any member of $\mathsf{LDS}$ is continuous at zero. Finally, for $(n,\epsilon)\in \mathbb{N}\times (0,\infty)$  we check  \eqref{eq:weird-conditionLD} with $\sigma\equiv \infty$, $\tau=\tau_n^+\wedge \tau_\epsilon^{-}$. By the sample path continuity of $Y$ (excepting the eventual jump to $\infty$), evidently, a.s., for all $t\in [0,\infty)$, $Y_{t\land \tau}\in [\epsilon,n]\cup \{y_0,\infty\}$. Since $\Psi$, $\Sigma$ and $\Phi$ are all locally bounded, since $\psi(\cdot,0)\equiv 0$ and since $0^+\cdot\infty$ is in force, \eqref{eq:weird-conditionLD} is then immediate (the integrand is bounded a.s.).

All the conditions for Theorem \ref{thm:generators}\ref{thm:generators:II} are thus met and the proof is complete.
\end{proof}
The next proposition complements \cite[Proposition 2.6]{foucartvidmar}. 
\begin{proposition}\label{prop:nonexplosionCBCI}
  If $\int_{0+}\frac{\dd u}{|\Psi(u)|}=\infty$ and $\limsup_{x\downarrow 0}\frac{\Phi(x)}{-\Psi(x)}<\infty$, then the CBCI process $X$ does not explode.
\end{proposition}
The sufficient condition of Proposition~\ref{prop:nonexplosionCBCI}, for which we make no claim on optimality, is automatically satisfied when $\Psi'(0)\geq 0$; also, if $\Phi'(0)<\infty$, then  $\limsup_{x\downarrow 0}\frac{\Phi(x)}{-\Psi(x)}<\infty$.
\begin{proof}[Proof of Proposition~\ref{prop:nonexplosionCBCI}]
Recall Definition~\ref{defLPsi}. Introduce 
 \[\mathcal{X}'f(x):=x^2 \mathrm{L}^{\Sigma}f(x)+x\mathrm{L}^{\Psi}f(x)+\mathrm{L}^{-\Phi}f(x),\quad x\in D_f,\, f\in  D_{\mathrm{L}^{\Sigma}}\cap D_{\mathrm{L}^{\Psi}}\cap D_{\mathrm{L}^{-\Phi}}.\]
The proof of non-explosion relies on constructing a ``Lyapunov'' function $f \in D_{\XX'}$ that is nonnegative, of class $C^2$,  $\uparrow$ and satisfies $\lim_{x \to \infty} f(x) = \infty$, along with the inequality $\XX' f \leq C f$ holding true on a neighborhood of infinity for some $C\in [0,\infty)$.  A Lyapunov-type argument -- see e.g. \cite[Section 6.3]{zbMATH05875699} for background, and \cite[Proof of Lemma 5.3]{foucartvidmar}, \cite[Theorem A]{rebotier} for results in settings similar to ours -- then entails that the process does not explode. In view of  the structure that Theorem~\ref{thm:excessivecm}  identifies for candidate  $\uparrow$ excessive functions, we endeavour to find  $f$ of Bernstein form. 

We first address the case of a (sub)critical mechanism $\Psi$, i.e. $\Psi'(0)\geq 0$.  Let $\mu$ be any infinite measure on $\BB_{(0,\infty)}$  such that $\int (z+z^2+\Phi(z)+\Sigma(z)+\Psi(z)) \mu(\dd z)<\infty$ (it exists $\because$ $\lim_{x\downarrow 0}(\Phi+\Sigma+\Psi)(x)=0$, any will do, e.g. we could select a sequence $(\alpha_n)_{n\in \mathbb{N}}$ in $(0,\infty)$  such that $(\Phi+\Sigma+\Psi)(\alpha_n)+\alpha_n^2+\alpha_n\leq n^{-2}$ for all $n\in \mathbb{N}$ and set $\mu:=\sum_{n\in \mathbb{N}}\delta_{\alpha_n}$). Define $$f(x):=\int (1-e^{-xz})\mu(\dd z), \quad x\in [0,\infty).$$
Evidently $f$ is $\uparrow$ nonnegative of class $C^2$, and  $\lim_{x\to\infty} f(x)=\infty$ due to monotone convergence and since $\mu((0,\infty))=\infty$. By Tonelli and differentiation under the integral sign, $f\in D_{\mathrm{L}^{-\Phi}}$ and 
\begin{equation*}
\mathrm{L}^{-\Phi}f(x)
=\int e^{-xz}\Phi(z)\mu(\dd z),\quad x\in [0,\infty).
\end{equation*}
By dominated convergence we see that $\lim_{x\to\infty}\mathrm{L}^{-\Phi}f(x)=0$ and thus certainly  $\mathrm{L}^{-\Phi}f\leq f$ on a neighborhood of infinity. 
By entirely analogous reasoning we find that $f\in D_{\mathrm{L}^{\Sigma}}\cap D_{\mathrm{L}^{\Psi}}$, $$\mathrm{L}^\Sigma f(x)=-\int e^{-xz}\Sigma(z)\mu(\dd z)\leq 0\text{ and }\mathrm{L}^\Psi f(x)=-\int e^{-xz}\Psi(z)\mu(\dd z)\leq 0,\,\, x\in [0,\infty).$$
Altogether it allows to conclude that $f$ has the stipulated properties of  a Lyapunov function.

We now consider the more intricate supercritical case for which $\Psi'(0) \in [-\infty, 0)$. Take any $\theta \in (0, -\Psi'(0))$, set $\rho:=\sup\{x\in (0,\infty):\Psi(x)<0\}\in (0,\infty]$ and fix any $v_0 \in (0, \rho)$. Define 
\[
f(x) := \int (1 - e^{-xv})\mu(\mathrm{d}v), \quad x\in (0,\infty),\]
 where $\mu$ is the absolutely continuous measure on $\BB_{(0,\rho)}$ given by the density  \begin{equation}\label{eq:densitymu}\frac{\mu(\mathrm{d}v)}{\dd v}:= \frac{\theta}{-\Psi(v)} \exp\left( \int_v^{v_0} \frac{\theta}{-\Psi(u)}\, \mathrm{d}u \right) ,\quad v\in (0,\rho).
\end{equation}
Thanks to having assumed $\int_{0+}\frac{\dd v}{\vert \Psi(v)\vert}=\infty$, by \cite[Lemma 5.1]{foucartvidmar} $f$ is finite, is of class $C^2$, and is $\uparrow$, nonnegative with $ \lim_{x\to \infty} f(x)=\infty$, furthermore it satisfies the identity
\[
x \mathrm{L}^{\Psi} f(x) = \theta f(x),  \quad x \in (0, \infty).
\]
Since $\int\vert  \Psi(v)\vert\mu(\dd v)<\infty$ and because $\limsup_{x\downarrow 0}\frac{\Phi(x)}{-\Psi(x)}<\infty$ we also compute, as before, that 
$$\mathrm{L}^\Sigma f(x)=-\int e^{-xz}\Sigma(z)\mu(\dd z)\leq 0\text{ and }\mathrm{L}^{-\Phi} f(x)=\int e^{-xz}\Phi(z)\mu(\dd z),\quad x\in (0,\infty).$$
It is now easy to conclude that $f$ satisfies all the properties listed in the first paragraph of this proof, as required.
\end{proof}

\subsubsection{CBREs}\label{paragraph:CBRE} Let $S$ be a L\'evy process with Laplace exponent $\kappa\in \pkspLp_1$, $$\kappa(z):=\int \big(e^{-zm}-1+zm\mathbbm{1}_{[-1,1]}(m)\big)\tilde{\nu}(\dd m)+\tilde{a}z^2-\tilde{b}z',\quad z\in [0,\infty),$$ where $\tilde{\nu}$ is a L\'evy measure on $\BB_{[-1,\infty)}$, $\tilde{a}\in [0,\infty)$ and $\tilde{b}\in \mathbb{R}$ (no killing term, however a mass of $\tilde \nu$ at $-1$ is allowed). For $x_0\in [0,\infty)$ consider the stochastic equation, for the process $X$, \begin{equation}\label{eq:doleansdade}X_t=x_0+\int_0^{t}X_{s-}\dd S_s,\quad t\in [0,\infty).\end{equation}
The unique strong solution $X$ is known as the Dol\'eans-Dade exponential associated to the semimartingale $S$, see e.g. Protter's book \cite[Theorem II.37]{zbMATH02006037}. It admits the a.s. explicit representation \begin{equation*}
X_t=x_0e^{S_t-\tilde{a}t}\prod_{s\in (0,t]}(1+\Delta S_s)e^{-\Delta S_s}, \quad t\in [0,\infty).
\end{equation*}
Additionally we consider the process $X$ issued from $x_0=\infty$ as being identically equal to $\infty$. Then it is not difficult to see that $X$ constitutes a positive Markov process with Laplace symbol $\psi\vert_{[0,\infty]\times (0,\infty)}$, where the $\psi\in \mathsf{LDS}$ is specified as $$\psi(x,y):=\mathbbm{1}_{[0,\infty)^2}(x,y)\kappa(xy),\quad (x,y)\in [0,\infty]^2.$$ Furthermore, $X$ is evidently in $(0^+ \cdot \infty, \infty \cdot 0^+)$-Laplace duality with itself.


Consider now the dynamics of a CB process \eqref{SDECB}, branching mechanism $\Psi$ as in Subsection~\ref{para:cb-processes}, but with  an extra stochastic integral term of the right-hand side of \eqref{eq:doleansdade} (with $S$ independent of $B$, $\mathcal{N}$). Such a generalization of a CB process is called a continuous-state branching process in a L\'evy random environment, CBRE for short. It has received some attention in the last decade, we refer the reader for instance to Bansaye et al. \cite{zbMATH06247275}, Palau and Pardo \cite{zbMATH06836271} and  He et al. \cite{zbMATH06963718}.  Informally speaking, at the level of the generator, the random environment is represented with an additional operator $\mathrm{L}_{\kappa}$ of Definition~\ref{eq:multiplicativelevy} corresponding to multiplicative (including downward) jumps.  
 
For simplicity and consistency with the literature that we shall appeal to  we limit our discussion here to branching dynamics with finite mean and no killing, i.e. to $\int (z\land z^2) \nu(\dd z)<\infty$ and $c=0$, so that $\Psi(0)=0$ and $\Psi'(0)\in \mathbb{R}$. The stochastic equation ``\eqref{SDECB}+\eqref{eq:doleansdade}":
\begin{align*}
    X_t=x_0+\int_0^t\!\! \sqrt{2aX_s}\dd B_s+b\int_{0}^t\!\! X_s\dd s +\int_{0}^{t}\int_{0}^{X_{s-}}\!\!\!\int_{0}^{\infty}u\bar{\mathcal{N}}(\dd s, \dd r, \dd u)&+\int_0^{t}X_{s-}\dd S_s
\end{align*}
for $ t\in [0,\infty)$, $x_0\in [0,\infty)$ being the initial value (when $x_0=\infty$ we take $X\equiv \infty$), then admits a unique strong solution \cite[Theorem 3.1]{zbMATH06963718} that is non-explosive, and we shall refer to its law as $\mathrm{CBRE}(\Psi,\kappa)$. Introduce also the associated Laplace dual symbol $\psi\in \mathsf{LDS}$,
\begin{equation}\label{symbol-for-cbre}
\psi(x,y):=\mathbbm{1}_{[0,\infty)^2}(x,y)\big(x\Psi(y)+\kappa(xy)\big), \ \ (x,y)\in [0,\infty]^2,
\end{equation}  
which supersedes the $\psi$ above for the remainder of this subsection. Then we assert
\begin{theorem}\label{thm:dualityCBRE}
The Laplace symbol of $X$ is  $\psi_X=\psi\vert_{[0,\infty]\times (0,\infty)}$. Furthermore, $X$ is in $(0^+\cdot \infty,\infty \cdot 0^+)$-Laplace duality with $Y$, the unique strong non-explosive solution to \begin{equation}\label{SDEdualCBRE}
        Y_t=y_0-\int_{0}^{t}\Psi(Y_s)\dd s+\int_{0}^{t}Y_{s-}\dd S_s,\quad t\in [0,\infty),
    \end{equation}
    $y_0\in [0,\infty)$ being the initial value (for $y_0=\infty$, again $Y\equiv \infty$). The process $Y$ has the Laplace symbol $\psi_Y=\hat\psi\vert_{[0,\infty]\times (0,\infty)}$. (In both cases the $\psi$ is that of \eqref{symbol-for-cbre}.)
\end{theorem}
The duality of Theorem~\ref{thm:dualityCBRE} can be seen  as a representation of the annealed semigroup of $X$ \cite[Theorems 3.6 and~2.4, Eq.~(3.7)]{zbMATH06963718}. The techniques used in \cite{zbMATH06963718,zbMATH06836271} in order to study these processes do not rely directly on the Laplace duality relationship, but on a quenched branching property.
\begin{proof}[Proof of Theorem~\ref{thm:dualityCBRE}]
Denote by $\XX$ the pregenerator associated to  $\psi\vert_{[0,\infty]\times (0,\infty)}$. For $f\in \DD$, by It\^o's lemma, the process
$(f(X_t)-\int_{0}^{t}\XX f(X_s)\dd s)_{t\in  [0,\infty)}$  is a local martingale for all starting values $x_0\in [0,\infty)$.
Since $\XX f$ is  bounded  we have a true martingale. We deduce that $\psi_X:=\psi\vert_{[0,\infty]\times (0,\infty)}$ is the Laplace symbol of $X$. Analogously  we can check the claim regarding the Laplace symbol of $Y$.

    Existence and uniqueness of a strong solution to  \eqref{SDEdualCBRE}, possibly explosive, follow by standard results, see e.g. Ikeda and Watanabe's book \cite[Chapter~4.9]{ikeda1989stochastic}. We refer also to Leman and Pardo \cite{zbMATH07470626} where a class of SDEs containing \eqref{SDEdualCBRE} is studied.   Let us further argue that actually $Y$ stays almost surely in $[0,\infty)$ for starting values in $y_0\in [0,\infty)$. Since $\Psi'(0)>-\infty$ there is a constant $c\in [0,\infty)$ such that $\Psi(y)\geq -cy$ for all $y\in [0,\infty)$. Define the L\'evy process $\tilde S$ by adding to $S$ a $c$-drift, i.e. $\tilde S_t=S_t+ct$ for $t\in [0,\infty)$. The Dol\'eans-Dade exponential $\tilde Y$ of $\tilde S$ with initial value $y_0$ then solves the stochastic equation
\[\tilde Y_t=y_0+\int_0^tX_{s-}\dd \tilde{S}_t,\quad t\in [0,\infty).\]    
By applying a comparison theorem, see e.g. \cite[Theorem 5.5]{zbMATH05676165} or Dawson-Li's work \cite[Theorem 2.3]{DawsonLi}, one can verify that a.s.  $Y\leq \tilde Y$. But $\tilde Y$ is finite a.s. and therefore $Y$ remains finite too a.s.. 

Thus both processes $X$ and $Y$ are non-explosive and have $0$ strongly absorbing. Then we seek to apply Theorem~\ref{thm:generators}\ref{thm:generators:II}. Appealing to Proposition \ref{prop:momentcond} we get the bound \eqref{eq:controlmeanC0}; therefore  \eqref{eq:weird-conditionLD} is verified with $ \sigma \equiv \infty$ and $\tau= \tau_n^{+}\land \tau_{\epsilon}^{-}$ (in the notation of Theorem \ref{thm:generators}). The other conditions of Theorem~\ref{thm:generators}\ref{thm:generators:II} are easily seen to be met.
\end{proof}
We conclude our study of CBREs by establishing their Feller property.
\begin{corollary} 
   Suppose $\tilde\nu(\{-1\})=0$. Then the $\mathrm{CBRE}(\Psi,\kappa)$ process $X$ is, on restriction to $[0,\infty)$, a $\mathsf{C}_0([0,\infty))$-Feller process. Furthermore, it admits a $(0\cdot \infty^-,\infty\cdot 0^+)$-Laplace dual $Y'$ for which, in particular, \begin{equation}\label{extinctionprob}\mathbb{P}_x(X_t=0)=\EE^{\infty}[e^{-xY_t'}],\quad  t\in  [0,\infty), \, x\in [0,\infty].
    \end{equation}
\end{corollary}
\begin{proof}
We have seen in (the course of the proof of) Theorem~\ref{thm:dualityCBRE} that the $(0^+\cdot \infty,\infty \cdot 0^+)$-Laplace dual $Y$ of $X$ has $\infty$ non-sticky; let us argue that this is true also of the boundary $0$. Indeed, since $\Psi(0)=0$ and $\Psi'(0)<\infty$, $\Psi(y)\leq cy$ for all $y\in [0,1]$ for some $c\in [0,\infty)$. Employing the same comparison technique as in the proof of Theorem~\ref{thm:dualityCBRE} the claim follows, the condition $\tilde\nu(\{-1\})=0$ ensuring that the Dol\'eans-Dade exponential $\tilde Y$ associated to $S$ from which is subtracted the $c$-drift stays strictly positive a.s. (it is enough to employ the comparison on the neighborhood $[0,1]$ of $0$, since $S$ never jumps down by $-1$, and hence $X$ never jumps down to $0$, a.s.). The Feller property is then a direct application of Proposition~\ref{propositon:Feller-property}\ref{propositon:Feller-property:i}.  

The second claim follows from  Proposition~\ref{prop:continuousextensionat0}\ref{extensionatinfinity} (extend continuously $Y$ restricted to $(0,\infty)$ back to the points $0$ and $\infty$ as the process $Y'$) and  \eqref{eq:laplacedualinftyabsorbing2}   (for \eqref{extinctionprob}).
\end{proof}

The extinction probability of a CBRE is
given in \cite[Section 4]{zbMATH06963718} in terms of the solution of a random differential equation with a singular terminal condition. Identity \eqref{extinctionprob} provides another representation. CB processes in random environment which also take into account  immigration (CBIREs) were studied in \cite{zbMATH06963718} and can also be explored through Laplace duality. We mention in addition \cite{zbMATH07470626} in which random environment and competition are combined.

\subsection{Stochastic equations for Markov processes with decomposable Laplace dual symbols}\label{sec:decomposable}
We are interested now in processes whose Laplace symbols are decomposable in the sense of Definition~\ref{def:LDS}.  Let then $n\in \mathbb{N}$, let $\{\Sigma_i,\hat{\Sigma}_i\}\subset \spLpnot$, $\{\Phi_i,\hat{\Phi}_i\} \subset \pksubo $ for $i\in [n]:=\{1,\ldots,n\}$, and set 
\begin{equation}\label{eq:symbolX}
\psi(x,y)=\mathbbm{1}_{[0,\infty)^2}(x,y)\sum_{i=1}^{n}\left(\hat{\Sigma}_i(x)\Sigma_i(y)-\hat{\Phi}_i(x)\Phi_i(y)\right), \ \ (x,y)\in [0,\infty]^2,
\end{equation}
The problem of constructing such decomposable processes and studying their symbol has been tackled in several works, we refer to  Hoh \cite{zbMATH00683136}, see also \cite{zbMATH02076723,zbMATH05946936}. In these articles the authors do not assume the functions $(\hat{\Sigma}_i)_{i\in [n]}$, $(\hat{\Phi}_i)_{i\in [n]}$ to be of L\'evy-Khintchine form. Moreover, they primarily consider real-valued processes whose generator domains contain $C^2$ functions with compact support.

For simplicity we shall take that there is no diffusive part, no drift, nor killing in $\psi$. More precisely, for all $i\in [n]$, we shall assume that there are  (automatically unique) measures $\mu_i$ and $\nu_i$ on $\BB_{(0,\infty)}$ satisfying $\int (u\wedge u^2)\,\mu_i(\dd u)<\infty$ and  $\int (1\wedge v) \, \nu_i(\dd v)<\infty$
for which 
\begin{equation*}
\Sigma_i(y)=\int (e^{-uy}-1+uy)\mu_i(\dd u)\text{ and }\Phi_i(y)=\int (1-e^{-vy})\nu_i(\dd v),\quad y\in [0,\infty).
\end{equation*}
Later we will impose the same condition on $\hat\psi$ too, but for now we need it only for $\psi$.

We endeavour then to provide a construction of a positive Markov process, absorbed at $\infty$, with Laplace symbol $\psi\vert_{[0,\infty]\times (0,\infty)}$, via the method of stochastic equations. To such an end let $\big(\mathcal{M}_{i}(\dd s,\dd r,\dd u)\big)_{i\in [n]}$, $\big(\mathcal{N}_{i}(\dd s,\dd r,\dd v)\big)_{i\in [n]}$ be mutually independent PRM with intensities respectively $\dd s\,\dd r\,\mu_i(\dd u)$ and $\dd s\,\dd r\,\nu_i(\dd v)$, $i\in [n]$. For all $i\in [n]$ we denote the compensated PRM associated to ${\mathcal M}_i$ by $\bar{\mathcal{M}}_i(\dd s,\dd u, \dd y):={\mathcal M}_i(\dd s,\dd u,\dd y)-\dd s\,\dd u\,\mu_i(\dd y)$.

Consider the stochastic equation for the process $X$: 
\begin{align}\nonumber
 X_t=x_0&+\sum_{i=1}^n\int_0^t\int_{0}^{\hat{\Sigma}_i(X_{s-})}\!\!\int_{0}^{\infty}u\,\bar{\mathcal{M}}_{i}(\dd s,\dd r, \dd u)\\
  &\qquad+\sum_{i=1}^{n}\int_0^t\int_{0}^{\hat{\Phi}_i(X_{s-})}\!\!\int_{0}^{\infty}v\,\mathcal{N}_{i}(\dd s,\dd r, \dd v),\quad t\in [0,\zeta), \label{SDE}
\end{align}
where $x_0\in [0,\infty)$ is the initial value  of $X$, $\zeta=\inf\{t>0: X_{t-} \text{ or } X_t \notin [0,\infty)\}$ the lifetime. For the initial value $x_0=\infty$ we take $X\equiv \infty$. 

If one were to allow the functions $(\Sigma_i)_{i\in [n]}$, $(\Phi_i)_{i\in [n]}$ in \eqref{eq:symbolX} to have drift parameters, say $(b_i)_{i\in [n]}\in [0,\infty)^n$ and $(d_i)_{i\in [n]}\in [0,\infty)^n$ respectively, then $\psi(x,y)$ in \eqref{eq:symbolX} would have a drift term $\hat{\Psi}(x)y$ with $\hat{\Psi}:=\sum_{i=1}^{n}(b_i\hat\Sigma_i-d_i\hat\Phi_i)\in \pkspLp$, whilst at the level of the stochastic equation \eqref{SDE} an additional term of the form $-\int_{0}^{t}\hat{\Psi}(X_s)\dd s$ would emerge. 

The next theorem provides a sufficient condition for having existence and uniqueness of a strong $[0,\infty]$-valued solution to \eqref{SDE}. Define  $$\hat{\Phi}:=\sum_{i=1}^{n}\hat{\Phi}_i\in \pksubo\text{ and }\Phi:=\sum_{i=1}^n\Phi_i\in\pksubo.$$

\begin{theorem}\label{thm:existencedecomposable}
Assume   \begin{equation}\label{conduniqueness} \int_{0+}\frac
{\dd x}{\hat{\Phi}(x)-\hat{\Phi}(0)}=\infty.
\end{equation} 
Then, for all initial values $x_0\in [0,\infty)$, there exists a unique strong solution $X$ to \eqref{SDE}. Moreover, $X$ is a positive Markov process  whose Laplace symbol is $\psi_X=\psi\vert_{[0,\infty]\times (0,\infty)}$.
\end{theorem}
\begin{proof}
The uniqueness of a strong solution will follow from \cite[Proposition 1]{zbMATH06836271}; see also the references therein. We rewrite \eqref{SDE} as \cite[Eq.~(5)]{zbMATH06836271}
\begin{equation*}
X_t=x_0+\int_0^t\int_Vh(X_{s-},\underline{v})\mathcal{N}(\dd s,\dd \underline{v})+\int_0^t\int_Ug(X_{s-},\underline{u})\bar{\mathcal{M}}(\dd s,\dd \underline{u}), \quad t\in [0,\zeta), 
\end{equation*}
where,  $c_{[n]}$ denoting counting measure on $[n]$:
\[V:=[n]\times (0,\infty)^{2},\ \ \underline{v}:=(i,r,v), \ \ h(x,\underline{v}):=v\mathbbm{1}_{[0,\hat{\Phi}_i(x)]}(r), \ \ \underline{\nu}(\dd \underline{v}):= c_{[n]}(\dd i)\nu_i(\dd v)\dd r,\]
$\mathcal{N}(\dd s,\dd \underline{v}):=\mathcal{N}_i(\dd s,\dd r,\dd v)$ having intensity  $\dd s\underline{\nu}(\dd \underline{v})$; 
\[U:=[n]\times (0,\infty)^{2},\ \ \underline{u}:=(i,r,u), \ \ g(x,\underline{u}):=u\mathbbm{1}_{[0,\hat{\Sigma}_i(x)]}(r), \ \ \underline{\mu}(\dd \underline{u}):=c_{[n]}(\dd i)\mu_i(\dd u)\dd r,\] $\mathcal{M}(\dd s,\dd \underline{u}):=\mathcal{M}_i(\dd s,\dd r,\dd u)$ having intensity  $\dd s\underline{\mu}(\dd \underline{u})$.  

We now check  \cite[pp.~60-61, (a)-(c)]{zbMATH06836271}. Written for our setting and with some simplifications as the SDE  \eqref{SDE} has only Poisson terms, these conditions are as follows.
\begin{enumerate}[(a)]
\item\label{palau-calderon:a}  For each $m\in (0,\infty)$ there is $A_m<\infty$, such that $  \int \big(|h(x,\underline{v})|\wedge 1\big) \underline{\mu}(\dd \underline{v})\leq A_m(1+x)$ for every $x\in [0,m]$.
\item\label{palau-calderon:b} For each $m\in (0,\infty)$ there is a $\uparrow$ concave $r_m:[0,\infty)\to [0,\infty)$ for which $\int_{0+}\frac{\dd u}{r_m(u)}=\infty$, such that $\int \big |h(x,\underline{v})\wedge m-h(y,\underline{v})\wedge m\big | \underline{\mu}(\dd \underline{v})\leq r_m(x-y)$ for $\{x,y\}\subset [0,m]$.
\item\label{palau-calderon:c} For all $\underline{u}=(i,u,r)\in U$ the map $([0,\infty)\ni x\mapsto x+g(x,\underline{u}))$ is $\uparrow$. Furthemore, setting  \[l(x,y,\underline{u}):=g(x,\underline{u})-g(y,\underline{u})=u\mathbbm{1}_{[0,\hat{\Sigma}_i(x)]}(r)-u\mathbbm{1}_{[0,\hat{\Sigma}_i(y)]}(r),\]  then for all $m\in (0,\infty)$ there is $B_m<\infty$, such that 
  \[\int \big |l(x,y,\underline{u})|\wedge l(x,y,\underline{u})^2 \underline{\nu}(\dd u)\leq B_m|x-y|,\quad \{x,y\}\subset [0,m].\]
\end{enumerate}
Note that the  notation is  transposed here relative to that of \cite[p.~60]{zbMATH06836271}: the compensation is on $\mathcal{M}$ to which is associated $g$, not on $\mathcal{N}$ to which is associated $h$.  

\textbf{Ad} \ref{palau-calderon:a}. Since the functions $(\hat{\Phi}_i)_{i\in [n]}$ are all $\uparrow$, therefore locally bounded,  the expression $  \int \big(|h(x,\underline{v})|\wedge 1\big) \underline{\mu}(\dd \underline{v})=\sum_{i=1}^{n}\hat{\Phi}_i(x)\int (1\wedge v)\nu_i(\dd v)$ is also locally bounded in $x\in [0,\infty)$.

\textbf{Ad } \ref{palau-calderon:b}. Let $m\in (0,\infty)$. For $\{x,y\}\subset [0,m]$, $y\leq x$,  \begin{align*}
   &\int \big |h(x,\underline{v})\wedge m-h(y,\underline{v})\wedge m\big | \underline{\mu}(\dd \underline{v})
   \leq\sum_{i=1}^{n}\int(v \wedge m)\big(\hat{\Phi}_i(x)-\hat{\Phi}_i(y)\big)\nu_i(\dd v)\\
   &\leq \sum_{i=1}^{n}\int  (v\wedge m) \nu_i(\dd v)\big(\hat{\Phi}_i(x-y)-\hat{\Phi}_i(0)\big)\leq C_m(\hat{\Phi}(x-y)-\hat\Phi(0)\big)=:r_m(x-y),
   \end{align*}
   where in the second inequality we used the fact that the Bernstein functions $(\hat{\Phi}_i-\hat\Phi_i(0))_{i\in [n]}$ are subadditive, while in the third inequality we set $C_m:=\max_{i\in [n]}\int_{0}^{\infty} (u\wedge m) \nu_i(\dd u)\in [0,\infty)$.  The assumption $\int_{0+} \frac{\dd u}{\hat{\Phi}(u)-\hat{\Phi}(0)}=\infty$ ensures $\int_{0+}\frac{\dd u}{r_m(u)}=\infty$;  $r_m:[0,\infty)\to [0,\infty)$ is $\uparrow$ concave since members of $\pksubo$ are so.
   
\textbf{Ad} \ref{palau-calderon:c}. Let $m\in (0,\infty)$. Since the functions $(\hat{\Sigma}_i)_{i\in [n]}$ are all $\uparrow$, plainly for any $\underline{u}=(i,u,r)\in U$, $([0,\infty)\ni x\mapsto x+g(x,\underline{u}))$ is $\uparrow$. 
To see the inequality involving $B_m$ assume with no loss of generality  $x\geq y$. Then $l(x,y,\underline{u})=u\mathbbm{1}_{(\hat{\Sigma}_i(y),\hat{\Sigma}_i(x)]}(r)$ and we have
   \begin{equation*}
\int \big |l(x,y,\underline{u})|\wedge l(x,y,\underline{u})^2 \underline{\nu}(\dd u)   =\sum_{i=1}^{n}\big(\hat{\Sigma}_i(x)-\hat{\Sigma}_i(y)\big) \int  (u\wedge u^2 )\mu_i(\dd u).
   \end{equation*}
But the members of $\spLpnot$  are locally Lipschitz and we are done.

The other  standing assumptions of \cite[p.~60]{zbMATH06836271} are easily seen to be met so that \cite[Proposition 1]{zbMATH06836271} gives a unique strong solution  $X$ to \eqref{SDE}. Let now $\XX$ be the pregenerator associated to the Laplace symbol $\psi\vert_{[0,\infty]\times (0,\infty)}$:
\begin{equation*}
\mathcal{X}f(x)=\sum_{i=1}^{n}\left(\hat{\Sigma}_i(x) \mathrm{L}^{\Sigma_i}f(x)+\hat{\Phi}_i(x)\mathrm{L}^{-\Phi_i}f(x)\right), \quad x\in [0,\infty),\, f\in \mathcal{D}.
\end{equation*}
By It\^o's lemma, for any $f\in \DD$, $M_f:=(f(X_t)-\int_{0}^{t}\mathcal{X}f(X_s)\dd s)_{t\in[ 0,\zeta)}$   
is a local martingale in any filtration $\FF$ relative to which $\mathcal{M}$ and $\mathcal{N}$ are PRM. By Proposition~\ref{lem:extendedpregen} or directly we see that $(\mathcal{X}f)\vert_{[0,\infty)}\in \mathsf{C}_0([0,\infty))$. So, if $(T_n)_{n\in \mathbb{N}}\uparrow \zeta$ is a  sequence of $\FF$-stopping times such that the stopped process $(M_f)^{T_n}$ is a martingale for each $n\in \mathbb{N}$ (a localizing sequence for $M_f$), then for all $s\leq t$ from $[0,\infty)$, in $$\EE_{x_0}\left[f(X_{t\land T_n})-\int_{0}^{t\land T_n}\mathcal{X}f(X_u)\dd u\Big\vert \FF_s\right]=f(X_{s\land T_n})-\int_{0}^{s\land T_n}\mathcal{X}f(X_u)\dd u\quad \text{a.s.-$\PP_{x_0}$}$$ we may pass to the limit $n\to\infty$ by bounded convergence to obtain $\EE_{x_0}[M_f(t)\vert \FF_s]=M_f(s)$ a.s.-$\PP_{x_0}$. It follows that  in fact $(f(X_t)-\int_{0}^{t}\mathcal{X}f(X_s)\dd s)_{t\in[ 0,\infty)}$ is a martingale under $\PP_{x_0}$ for all $x_0\in [0,\infty]$ (it is trivial for $x_0=\infty$).  In turn, for all $x\in [0,\infty]$ and $y\in (0,\infty)$, again by bounded convergence, we may compute $\lim_{t\downarrow 0}\frac{\EE_x[e^{-X_ty}]-e^{-xy}}{t}=\lim_{t\downarrow 0}\frac{\EE_x[\int_0^t(\XX\ee^y)(X_s)\dd s]}{t}=(\XX\ee^y)(x)$ entailing that $\psi_X=\psi\vert_{[0,\infty]\times (0,\infty)}$.
\end{proof}
The next proposition provides a sufficient condition for  non-explosion of \eqref{SDE}. 

\begin{proposition}\label{prop:nonexplosionDCM}
In addition to \eqref{conduniqueness}, suppose \begin{equation}\label{suffcondnonexplosion}\int_{0+}\frac{\dd y}{\Phi(y)}=\infty.
 \end{equation}    
Then the process $X$, solution to \eqref{SDE}, does not explode.
\end{proposition}
\noindent
Note the symmetry between the condition \eqref{conduniqueness} for the unique existence of $X$   and the condition \eqref{suffcondnonexplosion} for non-explosion  of $X$.
\begin{proof}[Proof of Proposition~\ref{prop:nonexplosionDCM}]
In the case $\Phi\equiv 0$ the solution to \eqref{SDE}, $x_0<\infty$, is a local martingale up to $\zeta$ in any filtration $\FF$ in which $\mathcal{M}$ is a PRM. In such case let then  $(T_n)_{n\in \mathbb{N}}\uparrow \zeta$ be a  sequence of $\FF$-stopping times that is localizing for $X$. By Fatou's lemma and optional sampling, for all $t\in [0,\infty)$, $$\infty\PP_{x_0}(\zeta\leq t)\leq \EE_{x_0}[\liminf_{n\to \infty}X_{t\land T_n}]\leq \liminf_{n\to\infty}\EE_{x_0}[X_{t\land T_n}]=x_0;$$ therefore $\PP_{x_0}(\zeta\leq t)=0$ and, on letting $t\uparrow \infty$, $\PP_{x_0}(\zeta<\infty)=0$, which means that $X$ is non-explosive. 

Suppose now that $\Phi$ is not identically null. Introduce
\begin{equation*}\mathcal{X}'f(x)=\sum_{i=1}^{n}\left(\hat{\Sigma}_i(x) \mathrm{L}^{\Sigma_i}f(x)+\hat{\Phi}_i(x)\mathrm{L}^{-\Phi_i}f(x)\right),\quad x\in D_f,\, f\in \cap_{i\in [n]}(D_{\mathrm{L}^{\Sigma_i}}\cap D_{\mathrm{L}^{-\Phi_i}}).
\end{equation*}
As in the proof of Proposition \ref{prop:nonexplosionCBCI}  we will find a function $f\in D_{\XX'}$ that is nonnegative, of class $C^2$, $\uparrow$ and satisfies $\lim_{x \to \infty} f(x) = \infty$, along with the inequality $\XX' f \leq C f$ holding true on a neighborhood of infinity for some $C\in [0,\infty)$.  Heuristically, to find such a function, we will compare the term $\sum_{i=1}^{n}\hat{\Phi}_i \mathrm{L}^{-\Phi_i}$ to the generator of an immortal CB with   branching mechanism $-\Phi$. Fix any $l\in (0,\infty)$. By applying \cite[Lemma 5.1]{foucartvidmar} with (in the notation therein) $\rho=\infty$ and $\Psi=-\Phi$, for any $\theta\in (0, \Phi'(0))$ the function 
    $$f(x):= \int (1-e^{-xv})\mu(\dd v),\quad x\in [l,\infty),$$
where $\mu$ is given by \eqref{eq:densitymu} (with $\rho=\infty$, $\Psi=-\Phi$) is of class $C^2$ and has limit $\infty$ at infinity. Furthermore, $x\mathrm{L}^{-\Phi}f(x)=\theta f(x)$ for all $x\in [l,\infty)$, and   since $[l,\infty)\ni x\mapsto \Phi(x)/x$ is $\downarrow$ we deduce that
    \begin{align*}\sum_{i=1}^{n}\hat{\Phi}_i(x)\mathrm{L}^{-\Phi_i}f(x)=\sum_{i=1}^{n}\frac{\hat{\Phi}_i(x)}{x}x\mathrm{L}^{-\Phi_i}f(x)&\leq \sum_{i=1}^{n}\frac{\hat{\Phi}_i(l)}{l} x\mathrm{L}^{-\Phi_i}f(x)\\ 
    &\leq \hat{C_{0}}x\mathrm{L}^{-\Phi}f(x)
    =\hat{C_{0}}\theta f(x)=:Cf(x)\end{align*} for all $x\in [l,\infty)$, where $\hat{C}_0:=\max_{i\in [n]}\frac{\hat{\Phi}_i(l)}{l}$. 
 Besides, $$\sum_{i=1}^n\hat{\Sigma}_i(x)\mathrm{L}^{\Sigma_i}f(x)
 =-\sum_{i=1}^n\hat{\Sigma}_i(x)\int \Sigma_i(v)e ^{-xv} \mu(\dd v)\leq 0, \quad x\in [l,\infty).$$
Therefore $\mathcal{X}'f(x) \leq C f(x)$ for $x\in [l,\infty)$, which completes the proof.\end{proof}
    
We now establish a Laplace duality relationship. Again for simplicity, we assume henceforth that the functions $(\hat{\Sigma}_i)_{i\in [n]}, (\hat{\Phi}_i)_{i\in [n]}$ take the ``pure-jump'' form 
\[\hat{\Sigma}_i(x)=\int_{0}^{\infty}\big(e^{-xu}-1+ux\big)\hat{\mu}_i(\dd u),\ \ \hat{\Phi}_i(x)=\int_{0}^{\infty}(1-e^{-xv})\hat{\nu}_i(\dd v),\ \  x\in [0,\infty), \ i\in [n],\] 
with $(\hat{\mu}_i)_{i\in [n]}, (\hat{\nu}_i)_{i\in [n]}$ L\'evy measures on $\BB_{(0,\infty)}$ such that \[\int (u\wedge u^2) \hat{\mu}_i(\dd u)<\infty \text{ and } \int (1\wedge v) \hat{\nu}_i(\dd v)<\infty,\quad  i\in [n].\]
Introduce  the  ``dual" stochastic equation for the process $Y$,
\begin{align}
Y_t=y_0+\sum_{i=1}^{n}\int_0^t\int_{0}^{\Sigma_i(Y_{s-})}\!\!\int_{0}^{\infty}u\,\bar{\hat{\mathcal{M}}}_i(\dd s,\dd r, \dd u)+\sum_{i=1}^{n}\int_0^t\int_{0}^{\Phi_i(Y_{s-})}&\!\!\int_{0}^{\infty}v\hat{\mathcal{N}}_i(\dd s,\dd r, \dd v),\label{SDEY}
\end{align}
with $t\in [0,\hat\zeta)$,  $y_0\in [0,\infty)$ the initial value, $\hat\zeta$ the lifetime, $\hat{\mathcal{M}}_{i}(\dd s,\dd r,\dd u)$ and  $\hat{\mathcal{N}}_{i}(\dd s,\dd r,\dd v)$ mutually independent PRM with respective intensities $\dd s\,\dd r\, \hat{\mu}_i(\dd u)$ and $\dd s\,\dd r\, \hat{\nu}_i(\dd v)$, $\bar{\hat{\mathcal{M}}}_{i}$ the compensated random measure associated to $\hat{\mathcal{M}}_{i}$ ($i\in [n]$). The usual provisos for the value $\infty$ of $Y$ under $y_0=\infty$ apply.

\begin{theorem}\label{thm:laplacedualspLDS+}
Assume  \begin{equation}\label{cond:decomposablesymbol}\int_{0+}\frac{\dd x}{\hat{\Phi}(x)}=\infty \text{ and } \int_{0+}\frac{\dd y}{\Phi(y)}=\infty.\end{equation} 
Then $X$, the unique non-explosive strong solution to \eqref{SDE},  is in   $(0^{+}\cdot \infty, \infty\cdot 0^+)$-Laplace duality with  the process $Y$, the unique non-explosive strong solution to \eqref{SDEY}.
\end{theorem}
Theorem~\ref{thm:laplacedualspLDS+} applies to a broader class of processes with decomposable non-local Laplace symbols than Theorem~\ref{thm:LaplacedualitysemigroupLDS}, as condition~\eqref{cond:decomposablesymbol} is strictly weaker than~\eqref{eq:controlmeanC0}. In particular, \eqref{eq:controlmeanC0} requires that $\hat{\Phi}'(0) < \infty$ and $\Phi'(0) < \infty$, whereas \eqref{cond:decomposablesymbol} imposes no such restriction.
\begin{proof}[Proof of Theorem~\ref{thm:laplacedualspLDS+}]
 Theorem~\ref{thm:existencedecomposable} and  Proposition~\ref{prop:nonexplosionDCM} guarantee unique strong existence as well as the non-explosivity of the solutions. As for the duality, let us apply Theorem \ref{thm:generators}\ref{thm:generators:III}.  The fact that $X$ and $Y$ solve the relevant MPs was mentioned in the proof of  Theorem~\ref{thm:existencedecomposable} and follows by It\^o's lemma (as it did in said proof). For all $(\delta,\epsilon) \in (0,\infty)^2$, by polynomial boundedness of the L\'evy-Khintchine functions,
\[\sup_{(x,y)\in [\delta,\infty)\times [\epsilon,\infty)}e^{-xy}\vert \psi(x,y)\vert<\infty;\]  together with the absence of negative jumps \eqref{eq:weird-conditionLD}-\eqref{eq:weird-condition++LD} follow  with $\sigma=\sigma_\delta^{-}$ and $\tau=\tau_{\epsilon}^{-}$. 

By assumption, for all $i\in [n]$, $\hat{\Phi}_i(0)=\Phi_i(0)=0$, and one always has $\hat{\Sigma}_i(0)=\Sigma_i(0)=0$, $i\in [n]$. By continuity of the functions $(\hat{\Sigma}_i, \Sigma_i, \hat{\Phi}_i, \Phi_i)_{i\in [n]}$, the symbol $\psi$, and then the $\phi$ of \eqref{thm:generators-phi} are continuous on $[0,\infty)^2$. Thus \eqref{phi:0} $\widehat{\eqref{phi:0}}$,  \eqref{phi:variation'}, $\widehat{\eqref{phi:variation'}}$ and \eqref{phi:4} hold. It remains to check that the processes $X$ and $Y$ are $\mathbb{D}_{[0,\infty]}^{\mathrm{m}0}$-valued. But this is a consequence of the pathwise uniqueness of the solutions and the fact that the constant $0$ is always a solution (with initial value $0$) since the coefficients $(\Sigma_i)_{i\in [n]}$, $(\Phi_i)_{i\in [n]}$ all vanish at $0$. 
\end{proof}
%
%
In closing, let us mention that  completely monotone Markov processes with the Laplace symbol of~\eqref{symbolgencb} can be obtained as scaling limits of discrete population models. A prominent example are the CBI processes, which emerge as limits of Bienaym\'e-Galton-Watson processes with immigration; see, e.g., Grimvall~\cite{MR0362529} and Kawazu and Watanabe \cite{KAW}. We refer to Bansaye and Simatos~\cite{zbMATH06471553} for the setting of branching processes in random environments, and to Bansaye et al.~\cite{zbMATH07055657} for a more general framework that includes competition effects. This raises the natural question of whether processes with other symbols in the class $\mathsf{LDS}$ can similarly be derived as scaling limits of suitable discrete models with a meaningful physical or biological interpretation.

\begin{appendix}
\section{Proof of Theorem \ref{thm:courrege}: Courr\`ege form}\label{appendix}
Uniqueness is direct from \eqref{eq:courrege} and  Proposition~\ref{lemma:laplace-symbol-unique}. We turn to existence.

Fix $x\in [0,\infty]$. Let us prove that
\begin{equation}\label{courrege-ab-initio}
(\XX\ee^y)(x)=\psi_{l_x}(y)e^{-xy},\quad y\in (0,\infty),
\end{equation}
 for a suitable $l_x$.
 
$(\bullet)$ Suppose first that $x\in (0,\infty)$. Set $$\tilde{p}_t(x,\dd u):=p_t(x,x+\dd u),\quad u\in [-x,\infty],$$ for the push-forward  of $p_t(x,\cdot)$ along the map $[0,\infty]\ni u\mapsto u-x\in [-x,\infty]$.  

Picking and fixing $y_0\in (0,\infty)$, $v_0\in (0,y_0)$ arbitrary, 
\begin{align}
\label{lhs}&\frac{v_0}{y_0}e^{xy_0}\frac{(P_t\ee^{y_0})(x)-\ee^{y_0}(x)}{t}-e^{xv_0}\frac{(P_t\ee^{v_0})(x)-\ee^{v_0}(x)}{t}\\ 
&\quad=\frac{v_0}{y_0}\int\left(e^{-uy_0}-1\right)\frac{1}{t}\tilde{p}_t(x,\dd u)-\int\left(e^{-uv_0}-1\right)\frac{1}{t}\tilde{p}_t(x,\dd u)\nonumber \\
&\quad=\int\Big(\underbrace{\frac{v_0}{y_0}e^{-uy_0}-\frac{v_0}{y_0}-e^{-uv_0}+1}_{:=\mathsf{c}(u)}\Big)\frac{1}{t}\tilde{p}_t(x,\dd u),\quad t\in (0,\infty).\label{rus}
\end{align}\normalsize

Let us expose some basic properties of $\mathsf{c}$. For $u\in [0,\infty)$, $\mathsf{c}'(u)=v_0\left(e^{-uv_0}-e^{-uy_0}\right)\geq 0$ and $\mathsf{c}(u)\underset{u\rightarrow \infty}{\longrightarrow} 1-\frac{v_0}{y_0}=\mathsf{c}(\infty)$. Thus $\mathsf{c}$ is $\uparrow$ on $[0,\infty]$ and (so) $0\leq \mathsf{c}(u)\leq  1-\frac{v_0}{y_0}=\mathsf{c}(\infty)$ for $u\in [0,\infty]$. For $u\in [-x, 0]$: $\mathsf{c}'(u)=v_0\left(e^{-uv_0}-e^{-uy_0}\right)\leq 0$; thus $\mathsf{c}$ is $\downarrow$ on $[-x,0]$ and $0\leq \mathsf{c}(u)\leq \mathsf{c}(-x)$. Besides, a Taylor expansion yields
\[\mathsf{c}(u)=\frac{v_0}{y_0}\left(-uy_0+\frac{y_0^2 u^2}{2}\right)+uv_0-\frac{v_0^2u^2}{2}+o(u^2)=\frac{v_0(y_0-v_0)}{2}u^2+o(u^2)\text{ as }u\to 0;\]
therefore there exists $d>0$ (only depending on $y_0,v_0$), such that \begin{equation}\label{bound1wedgeu2} d\leq \frac{\mathsf{c}(u)}{1\wedge u^2},\quad u\in [-x,\infty],
\end{equation}
the expression being understood in the limiting sense when $u=0$. 

Now, the l.h.s. of \eqref{lhs} converges in $\mathbb{R}$ as $t\downarrow 0$ because by assumption $\ee^{y_0}$ and $\ee^{v_0}$ belong to $D_\XX$. The integral of  \eqref{rus} is therefore bounded uniformly in $t\in (0,\infty)$ and we obtain with the help of  \eqref{bound1wedgeu2} that so too is the family of  measures $\gamma:=(\frac{1\wedge u^2}{t}\tilde{p}_t(x,\dd u))_{t\in (0,\infty)}$ on $\BB_{[-x,\infty]}$  bounded. 

Pick  any sequence $(t_k)_{k\in \mathbb{N}}$ in $(0,\infty)$ that is $\downarrow\downarrow$ to $0$. Passing to a subsequence if necessary, we may and do assume $(\gamma_{t_k}([-x,\infty])_{k\in \mathbb{N}}$ converges in $[0,\infty)$. Since $[-x,\infty]$ is compact, by Prokhorov's theorem \cite[p.~104, Theorem 2.2]{EthierKurtz}, passing to a further subsequence if necessary,
\[\frac{1\wedge u^2}{t_k}\tilde{p}_{t_k}(x,\dd u) \Longrightarrow \tilde{\nu}_x(\dd u)\text{ as $k\to\infty$}\]
(weak convergence) for some finite measure $\tilde{\nu}_x$ on $\BB_{[-x,\infty]}$. 


 Let us pick and fix now an $r\in (0,\infty)$ for which $\{-r,r\}\cap [-x,\infty]$ is not charged by $\tilde{\nu}_x$ (it exists because $\tilde{\nu}_x$ has at most countably many atoms). Returning to our aim of getting \eqref{eq:courrege}, we write, for $y\in (0,\infty)$, $k\in \mathbb{N}$,
\begin{align}
\frac{(P_{t_k}\ee^y)(x)-\ee^y(x)}{t_k}&=e^{-xy}\int \frac{e^{-uy}-1+uy\mathbbm{1}_{[-r,r]}(u)}{1\wedge u^2}\frac{1\wedge u^2}{t_k}\tilde{p}_{t_k}(x,\dd u)\label{noise}\\* &\qquad\qquad \qquad\qquad \quad \quad \qquad\qquad-e^{-xy}\int uy\mathbbm{1}_{[-r,r]}(u)\frac{1}{t_k}\tilde{p}_{t_k}(x,\dd u) \label{drift}.
\end{align}
The l.h.s. in \eqref{noise}  converges as $k \to\infty$, since by assumption $\ee^y\in D_\XX$, moreover, the limit is $(\XX \ee^y)(x)$. 
The r.h.s. of \eqref{noise} must therefore also converge as $k\to\infty$. Its first term does so by the portmanteau theorem (indeed $([-x,\infty]\ni u\mapsto \frac{e^{-uy}-1+uy\mathbbm{1}_{[-r,r]}(u)}{1\wedge u^2})$ [understood in the limiting sense at $0$] is bounded,  continuous a.e.-$\tilde\nu_x$ by the choice of $r$), the limit being $$I_1:=e^{-xy}\int \frac{e^{-uy}-1+uy\mathbbm{1}_{[-r,r]}(u)}{1\wedge u^2}\tilde\nu_x(\dd u).$$  In turn the expression of \eqref{drift} converges as $k\to\infty$ and is of the form
\begin{equation}\label{combine1}
I_2:=-b_xy\ee^y(x)
\end{equation}
for 
$$b_x:=\underset{k\rightarrow \infty}{\lim} \int u\mathbbm{1}_{[-r,r]}(u)\frac{1}{t_k}\tilde{p}_{t_k}(x,\dd u)\in \mathbb{R}.$$  
Parametrizing now the measure $\tilde{\nu}_x$  in the form $$\tilde{\nu}_x(\dd u)=:2a_x\delta_0(\dd u)+\mathbbm{1}_{[-x,\infty)}(u)(1\wedge u^2) \nu_x(\dd u)+c_x\delta_{\infty}(\dd u),\quad u\in [-x,\infty],$$ with $a_x\in [0,\infty)$, $c_x\in[ 0,\infty)$ and $\nu_x$ a L\'evy  measure on $\BB_{[-x,\infty)}$, we express
\begin{equation*}
I_1=e^{-xy}\int\left(\frac{e^{-uy}-1+uy\mathbbm{1}_{[-r,r]}(u)}{1\wedge u^2}\right)\left(2a_x\delta_0(\dd u)+(1\wedge u^2)\nu_x(\dd u)+c_x\delta_{\infty}(\dd u)\right).
\end{equation*}
Integration against $\delta_0$, respectively $\delta_\infty$, $(1\wedge u^2)\nu_x(\dd u)$, yields 
\begin{equation}\label{combine2}
e^{-xy}\int\left(\frac{e^{-uy}-1+uy\mathbbm{1}_{[-r,r]}(u)}{1\wedge u^2}\right)2a_x\delta_0(\dd u)=e^{-xy}\frac{y^2}{2}2a_x
=a_xy^2\ee^y(x),
\end{equation}
respectively
\begin{equation}\label{combine3}
e^{-xy}\int\left(\frac{e^{-uy}-1+uy\mathbbm{1}_{[-r,r]}(u)}{1\wedge u^2}\right)c_x\delta_\infty(\dd u)=-c_x\ee^y(x),
\end{equation}
\begin{equation}\label{combine4}
\begin{split}
&e^{-xy}\int\left(\frac{e^{-uy}-1+uy\mathbbm{1}_{[-r,r]}(u)}{1\wedge u^2}\right)(1\wedge u^2)\nu_x(\dd u)\\
&\qquad=\int\left(e^{-uy}-1+uy\mathbbm{1}_{[-r,r]}(u)\right)\nu_x(\dd u)\ee^y(x).
\end{split}
\end{equation}
Combining \eqref{combine1}-\eqref{combine4} we obtain \eqref{courrege-ab-initio} for $l_x:=(\nu_x,a_x,b_x,c_x)$ except that maybe $r\ne 1$, which does not matter (by readjustment of $b_x$).

$(\bullet)$ For $x=0$ we follow a similar but simplified approach. First, picking and fixing arbitrary $y_0\in (0,\infty)$, we have 
\begin{equation}\label{eq:gen-at-zero}
\frac{1-(P_t\ee^{y_0})(0)}{t}=\int\Big(\underbrace{1-e^{-uy_0}}_{=:\mathsf{d}(u)}\Big)\frac{1}{t}p_t(0,\dd u),\quad t\in (0,\infty),
\end{equation} and we note that there is $f>0$ (only depending on $y_0$) such that  $$f\leq \frac{\mathsf{d}(u)}{1\land u},\quad u\in [0,\infty],$$ the expression being understood in the limiting sense when $u=0$. Therefore, since the l.h.s. of \eqref{eq:gen-at-zero} is converging as $t\downarrow 0$ due to $\ee^{y_0}\in D_\XX$,  the family of measures $(\frac{1\wedge u}{t}p_t(0,\dd u))_{t\in (0,\infty)}$  on $\BB_{[0,\infty]}$ is bounded.  We deduce existence of a sequence 
 $(t_k)_{k\in \mathbb{N}}$ in $(0,\infty)$ that is $\downarrow\downarrow$ to $0$ such that 
\[\frac{1\wedge u}{t_k}p_{t_k}(0,\dd u) \Longrightarrow \tilde{\nu}_0(\dd u)\text{ as $k\to\infty$}\] for some finite measure $\tilde{\nu}_0$ on $\BB_{[0,\infty]}$. Writing then, for $y\in (0,\infty)$, $k\in \mathbb{N}$,
\begin{equation*}\label{expressx0}
\frac{1-(P_{t_k}\ee^y)(0)}{t_k}=\int\frac{1-e^{-uy}}{1\land u}\frac{1\land u}{t_k}p_{t_k}(0,\dd u),
\end{equation*}
we may pass to the limit $k\to\infty$ and deduce via portmanteau that \begin{equation}\label{equation:x0}
 -(\XX\ee^y)(0)=\int \frac{1-e^{-uy}}{1\land u}\tilde\nu_0(\dd u)
\end{equation} (the integrand being understood in the limiting sense for $u=0$). Parametrizing the measure $\tilde{\nu}_0$  in the form $$\tilde{\nu}_0(\dd u)=:d_0\delta_0(\dd u)+\mathbbm{1}_{(0,\infty)}(u)(1\wedge u) \nu_0(\dd u)+c_0\delta_{\infty}(\dd u),\quad u\in [0,\infty],$$ with $d_0\in [0,\infty)$, $c_0\in[ 0,\infty)$ and $\nu_0$ a  measure on $\BB_{(0,\infty)}$   satisfying $\int (1\land u)\nu_0(\dd u)<\infty$, we obtain from \eqref{equation:x0} at once \eqref{courrege-ab-initio}  for  $l_0=(\nu_0,d_0,c_0)$.

$(\bullet)$ We consider now the case $x=\infty$. Let $(y_n)_{n\in \mathbb{N}}$ be a sequence in $(0,\infty)$ that is $\downarrow\downarrow$ to $0$.  For $n\in \mathbb{N}$ we have 
\begin{equation*}
\frac{(P_t\ee^{y_n})(\infty)}{t}=\int e^{-uy_n}\frac{p_t(\infty,\dd u)}{t},\quad t\in (0,\infty).
\end{equation*}
As in the previous two instances, inductively in $n\in \mathbb{N}$ and then via diagonalization, we find a sequence $(t_k)_{k\in \mathbb{N}}$ in $(0,\infty)$ that is $\downarrow \downarrow 0$ and such that \[\frac{e^{-uy_n}}{t_k}p_{t_k}(\infty,\dd u) \Longrightarrow \tilde{\nu}_n(\dd u)\text{ as $k\to\infty$}\]
for some finite measure $\tilde{\nu}_n$ on $\BB_{[0,\infty]}$ for all $n\in \mathbb{N}$. Notice that (by portmanteau and uniqueness of weak limits on $[0,\infty]$) $$\tilde\nu_n(\dd u)=e^{-u(y_n-y_m)}\tilde\nu_m(\dd u),\quad u\in [0,\infty],$$ for all $m\geq n$ from $\mathbb{N}$; in particular $\tilde \nu_n$ is actually carried by $[0,\infty)$ for all $n\in \mathbb{N}$. Consequently, there is a single measure $\tilde\nu_\infty$ on $\BB_{[0,\infty)}$ for which $$\tilde\nu_\infty(\dd u)=e^{uy_n}\tilde\nu_n(\dd u),\quad u\in [0,\infty),\, n\in \mathbb{N}.$$ Then, for any given $y\in (0,\infty)$, there is $n\in \mathbb{N}$ such 
that $y>y_n$, and in writing, for $k\in \mathbb{N}$,
$$\frac{(P_{t_k}\ee^y)(\infty)}{t_k}=\int e^{-u(y-y_n)}\frac{e^{-uy_n}}{t_k}p_{t_k}(\infty,\dd u),$$ we may pass to the limit $k\to \infty$ via portmanteau to deduce that 
$$(\XX \ee^y)(\infty)=\int e^{-uy}\tilde\nu_\infty(\dd u).$$
Parametrizing now the measure $\tilde{\nu}_\infty$ in the form 
$$\tilde{\nu}_\infty(\dd u)=:k_\infty\delta_0(\dd u)+\mathbbm{1}_{(0,\infty)}(u)\nu_\infty(\dd u),\quad u\in [0,\infty),$$ with $k_\infty\in [0,\infty)$ and $\nu_\infty$ a measure on $\BB_{(0,\infty)}$ having finite Laplace transform we obtain \eqref{courrege-ab-initio} for $l_\infty:=(\nu_\infty,k_\infty)$.

Lastly, no longer holding $x$ fixed, it is straightforward to check that measurability of $(0,\infty)\ni x\mapsto \psi_{l_x}(y)$ for all $y\in (0,\infty)$ entails measurability of $(0,\infty)\ni x\mapsto  l_x$, for instance by employing the approach of Proposition~\ref{lemma:laplace-symbol-unique} and monotone class arguments.\qed

\section{Proof of Proposition~\ref{lemma:reductions}: killed processes and martingale problems}\label{appendix:reductions}
Recall $\zeta$ and $X'$ from \eqref{reductions:one}-\eqref{reductions:two}. For $(f,g)\in A$ with $f(\triangle)=0$ (a restriction that shall be without loss of generality due to the linearity of $A$ and $(1,0)\in A$) and $fk\in \mathsf{B}(E^\triangle)$, for $\{s,t\}\subset [0,\infty)$ with $s\leq t$, for  $W\in \sigma(X_u:u\in [0,s])$ and for $r\in [0,s]$  we compute:\small
\begin{align*}
&\EE\left[f(X'_t)-f(X'_s)-\int_s^tg(X'_u)\dd u+\int_s^tk(X'_u)f(X'_u)\dd u;W\cap \{ r<\zeta\}\right]\\
&=\EE\left[f(X_t)\mathbbm{1}_{\{t<\zeta\}}-f(X_s)\mathbbm{1}_{\{s<\zeta\}}-\int_s^tg(X_u)\mathbbm{1}_{\{u<\zeta\}}\dd u+\int_s^tk(X_u)f(X_u)\mathbbm{1}_{\{u<\zeta\}}\dd u;W\right]\\
&=\EE\left[\left(f(X_t)e^{-\int_s^{t}k(X_u)\dd u}-f(X_s)-\int_s^{t}g(X_u)e^{-\int_s^uk(X_v)\dd v}\dd u\right.\right.\\
&\qquad \qquad \qquad \qquad \qquad \qquad \left.\left.+\int_s^{t}k(X_u)f(X_u)e^{-\int_s^uk(X_v)\dd v}\dd u\right)e^{-\int_0^{s}k(X_u)\dd u};W\right].
\end{align*}\normalsize
By an application of Dynkin's lemma we see now that actually we are to check the (equivalently, by boundedness, local) martingale property of 
\small
$$\left(f(X_T)e^{-\int_0^{T}\!\!k(X_u)\dd u}-\int_0^{T}\!\!g(X_u)e^{-\int_0^uk(X_v)\dd v}\dd u+\int_0^{T}\!\!k(X_u)f(X_u)e^{-\int_0^uk(X_v)\dd v}\dd u\right)_{T\in [0,\infty)}$$ 
\normalsize
in the natural filtration of $X$. But performing an integration by parts we have, by associativity of integration, a.s.-$\PP$, 
\begin{align*}
&e^{-\int_0^tk(X_u)\dd u}\left(f(X_t)-\int_0^tg(X_u)\dd u\right)\\
&=f(X_0)+\int_0^te^{-\int_0^{u}k(X_v)\dd v}\dd_u\left(f(X_u)-\int_0^ug(X_s)\dd s\right)\\
&\qquad +\int_0^t\left(f(X_u)-\int_0^ug(X_v)\dd v\right)\dd_u\left(e^{-\int_0^uk(X_v)\dd v}\right)\\
&=f(X_0)+\int_0^te^{-\int_0^{u}k(X_v)\dd v}\dd_u\left(f(X_u)-\int_0^ug(X_s)\dd s\right)\\
&\qquad  -\int_0^tk(X_u)f(X_u)e^{-\int_0^uk(X_v)\dd v}\dd u-\int_0^t\int_0^ug(X_v)\dd v\dd_u\left(e^{-\int_0^uk(X_v)\dd v}\right) \\
&=f(X_0)+\int_0^te^{-\int_0^{u}k(X_v)\dd v}\dd_u\left(f(X_u)-\int_0^ug(X_s)\dd s\right)-\int_0^tk(X_u)f(X_u)e^{-\int_0^uk(X_v)\dd v}\dd u\\
&\qquad -e^{-\int_0^tk(X_u)\dd u}\int_0^tg(X_u)\dd u+\int_0^tg(X_u)e^{-\int_0^uk(X_v)\dd v}\dd u,
\end{align*}
where the last equality is another integration by parts. Since the second term on the r.h.s. is (as a process in $t\in [0,\infty)$) indeed a c\`adl\`ag  local martingale in the usual $\PP$-augmentation of the natural filtration of $X$ by the properties of stochastic integration, we conclude at once.\qed

\section{Proof of Theorem~\ref{theorem:duality-from-mtgs}: refined Ethier-Kurtz's duality criterion}\label{sec:proofoftheorem:duality-from-mtgs}
By replacing $\phi,H$ with $-\phi,-H$ if necessary we may and do assume that  \eqref{eq:weird-condition} holds with a ``$-$'' for ``$\pm$'' (integrability of the negative part). 

Put $$F(s,t):=\mathbb{E}[H(X_{s\land \sigma},Y_{t\land\tau})],\quad (s,t)\in [0,\infty)^2; $$ since $H$ is bounded, so is $F$. 

The martingales \eqref{eq:mtg-one}  ensure that, for all $s\in [0,\infty)$, all $y\in F$, one has
$$\EE[H(X_{s\land \sigma},y)-H(X_0,y)]=\EE\left[\int_0^{s\land\sigma}\phi(X_r,y)\dd r\right],$$ 
therefore,   for all $t\in [0,\infty)$, setting $s=T-t$ in, and then integrating the preceding display against $\QQ({Y_{t\land \tau}}\in \dd y)\mathbbm{1}_{[0,T]}(t)\dd t$,  via Tonelli:
\begin{align}
\int_0^T F(T-t,t)-F(0,t)\dd t&=\int_0^T\int_0^{T-t}\mathbb{E}[\phi(X_{s},Y_{t\land\tau});s<\sigma]\dd s\dd t\nonumber\\
&=\int_0^T\int_t^T\mathbb{E}[\phi(X_{s-t},Y_{t\land\tau});s-t<\sigma]\dd s\dd t\nonumber\\
&=\int_0^T\int_0^s\mathbb{E}[\phi(X_{s-t},Y_{t\land\tau});s-t<\sigma]\dd t\dd s,\quad T\in [0,\infty),\label{eq:shit-one}
\end{align} the triple integral on the r.h.s. being necessarily absolutely convergent (because, appealing  to \eqref{eq:weird-condition}, it is so under $\phi\rightsquigarrow \phi_-$, while the l.h.s. is finite). 

By analogy (and some changes of variables)
\begin{align} \nonumber
\int_0^T F(T-s,s)-F(s,0)\dd s&=\int_0^T F(s,T-s)-F(s,0)\dd s\\\nonumber
&=\int_0^T\int_0^s\mathbb{E}[\phi(X_{t\land\sigma},Y_{s-t});s-t<\tau]\dd t\dd s\\
&=\int_0^T\int_0^s\mathbb{E}[\phi(X_{(s-t)\land \sigma},Y_t);t<\tau]\dd t\dd s,\quad T\in [0,\infty),\label{eq:shit-two}
\end{align}
the triple integral on the r.h.s. being again automatically absolutely convergent. Subtracting \eqref{eq:shit-one} from \eqref{eq:shit-two} and rearranging delivers\small
\begin{align*}
&\int_0^T F(s,0)\dd s+\int_0^T\int_0^s\mathbb{E}[\phi(X_{(s-t)\land \sigma},Y_t);t<\tau]\dd t\dd s\\
&\qquad \qquad =\int_0^T F(0,t)\dd t+\int_0^T\int_0^s\mathbb{E}[\phi(X_{s-t},Y_{t\land\tau});s-t<\sigma]\dd t\dd s
\end{align*} 
\normalsize for all $T\in [0,\infty)$. Hence,  for  a.e. $t\in [0,\infty)$, \begin{equation}\label{thmD:general-}
\begin{split}
\mathbb{E}[H(X_{t\land \sigma},Y_0)]&-\mathbb{E}[H(X_0,Y_{t\land\tau})]\\
&=\mathbb{E}\left[\int_0^t\phi(X_{s},Y_{(t-s)\land \tau})\mathbbm {1}_{\{s<\sigma\}} -\phi(X_{s\land\sigma},Y_{t-s})\mathbbm{1}_{\{t-s<\tau\}}\dd s\right],
\end{split}
\end{equation}
where in fact  $$\int_0^t\mathbb{E}[\vert \phi(X_{s},Y_{(t-s)\land \tau})\vert;s<\sigma]\dd s<\infty \text{ and } \int_0^t\mathbb{E}[\vert \phi(X_{s\land\sigma},Y_{t-s}) \vert;t-s<\tau]\dd s<\infty.$$ 

 Notice now that \eqref{eq:weird-condition}  can be rewritten as $$\EE\left[\int_0^T\int_0^t\phi_{-} (X_{s\land \sigma},Y_{(t-s)\land \tau})\dd s\dd t\right]<\infty,\quad T\in [0,\infty),$$ so that
\begin{equation*}
\EE\left[\int_0^t\phi_{-} (X_{s\land \sigma},Y_{(t-s)\land \tau})\dd s\right]<\infty\text{ for a.e. $t\in [0,\infty)$}.
\end{equation*}

Besides, for all $s\leq t$ from $[0,\infty)$,
\begin{align}
&\phi(X_{s},Y_{(t-s)\land \tau})\mathbbm {1}_{\{s<\sigma\}} -\phi(X_{s\land\sigma},Y_{t-s})\mathbbm{1}_{\{t-s<\tau\}}\nonumber \\
&\qquad =\phi(X_{s\land \sigma},Y_{(t-s)\land \tau})(\mathbbm {1}_{\{s<\sigma\}} -\mathbbm{1}_{\{t-s<\tau\}}) \nonumber \\
&\qquad =\phi(X_{s\land \sigma},Y_{(t-s)\land \tau})(\mathbbm {1}_{\{t-s\geq \tau\}} -\mathbbm{1}_{\{s\geq \sigma\}})\nonumber\\
&\qquad =\phi(X_{s\land \sigma},Y_{\tau})\mathbbm {1}_{\{t-s\geq \tau\}}-\phi(X_{\sigma},Y_{(t-s)\land \tau})\mathbbm{1}_{\{s\geq \sigma\}}.\label{eq:aux-tonelli}
\end{align}
Then we compute, for a.e. $t\in [0,\infty)$,
\begin{align}\nonumber
&\EE\left[\int_0^t\phi(X_{s\land \sigma},Y_{\tau})\mathbbm {1}_{\{t-s\geq \tau\}}\dd s\right]\\\nonumber
&\quad=\EE\left[\int_0^{(t-\tau)_+\land\sigma}\phi(X_{s},Y_{\tau})\dd s\right]+\EE[\phi(X_\sigma,X_\tau)((t-\tau)_+-\sigma) _+]\\
&\quad=\EE[H(X_{(t-\tau)\land \sigma},Y_\tau)-H(X_0,Y_\tau);\tau\leq t]+\EE[\phi(X_\sigma,Y_\tau)((t-\tau)_+-\sigma) _+],\label{referee:query-answer}
\end{align}
in particular $\EE\left[\int_0^t\vert \phi(X_{s\land \sigma},Y_{\tau})\vert\mathbbm {1}_{\{t-s\geq \tau\}}\dd s\right]<\infty$ on taking into account \eqref{eq:weird-condition++}. Here, in passing from the second to the third line of \eqref{referee:query-answer}, we used  that the martingale of \eqref{eq:mtg-one} has a constant expectation (combined with Tonelli's theorem). By analogy,
\begin{align}\nonumber
\EE\left[\int_0^t\phi(X_{ \sigma},Y_{(t-s)\land \tau})\mathbbm {1}_{\{s\geq \sigma\}}\dd s\right]&=\EE\left[\int_0^t\phi(X_{ \sigma},Y_{s\land \tau})\mathbbm {1}_{\{t-s\geq \sigma\}}\dd s\right]\\\nonumber
&=\EE[H(X_\sigma,Y_{(t-\sigma)\land \tau})-H(X_\sigma,Y_0);\sigma\leq t]\\\label{ethier-kurtz-last}
&\qquad \qquad  \qquad +\EE[\phi(X_\sigma,Y_\tau)((t-\sigma)_+-\tau) _+],
\end{align}
and again we find a posteriori that actually $\EE\left[\int_0^t\vert\phi(X_{ \sigma},Y_{(t-s)\land \tau})\vert\mathbbm {1}_{\{s\geq \sigma\}}\dd s\right]<\infty$. Combining  \eqref{thmD:general-}-\eqref{ethier-kurtz-last} we conclude at once. \qed

\end{appendix}

\bibliographystyle{abbrv}
\bibliography{doku}

@book{conway,
  title     = {A Course in Functional Analysis},
  author    = {Conway, John B.},
  year      = {1997},
  edition   = {2nd},
  publisher = {Springer},
  address   = {New York},
  series    = {Graduate Texts in Mathematics},
  volume    = {96}
}

@article{blumenthal-getoor,
 author = {Blumenthal, R. M. and Getoor, R. K.},
 title = {Dual processes and potential theory},
 year = {1970},
  pages = {137--156},
 journal = {{Proc}. 12th Biennal {Seminar}, {Canad.} Math. {Congr}}
}

@article{duffie-etal,
 ISSN = {10505164},
 URL = {http://www.jstor.org/stable/1193233},
 author = {D. Duffie and D. Filipović and W. Schachermayer},
 journal = {The Annals of Applied Probability},
 number = {3},
 pages = {984--1053},
 publisher = {Institute of Mathematical Statistics},
 title = {Affine Processes and Applications in Finance},
 urldate = {2026-04-07},
 volume = {13},
 year = {2003}
}

@article{watanabe-bivariate,
 author = {Watanabe, Shinzo},
 journal = {Transactions of the American Mathematical Society},
 pages = {447--466},
 publisher = {American Mathematical Society},
 title = {On Two Dimensional {M}arkov Processes with Branching Property},
 volume = {136},
 year = {1969}
}

@book{dudley,
  title={Real Analysis and Probability},
  author={Dudley, Richard M.},
  series={Cambridge Studies in Advanced Mathematics},
  year={2004},
  publisher={Cambridge University Press}
}

@book{billingsley,
  title={Probability and Measure},
  author={Billingsley, Patrick},
  series={Wiley Series in Probability and Statistics},
  year={2012},
  publisher={Wiley}
}

@book{pollard,
  title={A User's Guide to Measure Theoretic Probability},
  author={Pollard, David},
  series={Cambridge Series in Statistical and Probabilistic Mathematics},
  year={2002},
  publisher={Cambridge University Press}
}

@article{kuznetsov-wang,
author = {Kuznetsov, Alexey and Wang, Yizao},
title = {{On the dual representations of Laplace transforms of Markov processes}},
volume = {29},
journal = {Electronic Journal of Probability},
publisher = {Institute of Mathematical Statistics and Bernoulli Society},
pages = {no. 161, 1--22},
year = {2024}
}

@article{curtiss,
 ISSN = {00034851},
 URL = {http://www.jstor.org/stable/2235846},
 author = {John H. Curtiss},
 journal = {The Annals of Mathematical Statistics},
 number = {4},
 pages = {430--433},
 publisher = {Institute of Mathematical Statistics},
 title = {A Note on the Theory of Moment Generating Functions},
 urldate = {2025-06-10},
 volume = {13},
 year = {1942}
}

@book{feller,
  title={An Introduction to Probability Theory and Its Applications},
  author={Feller, William},
  volume={2},
  year={1971},
  publisher={Wiley}
}

@article{zbMATH05676165,
 author = {Fu, Zongfei and Li, Zenghu},
 title = {Stochastic equations of non-negative processes with jumps},
 fjournal = {Stochastic Processes and their Applications},
 journal = {Stochastic Processes Appl.},
 issn = {0304-4149},
 volume = {120},
 number = {3},
 pages = {306--330},
 year = {2010},
 
 doi = {10.1016/j.spa.2009.11.005},
 zbMATH = {5676165},
 Zbl = {1184.60022}
}

@article{zbMATH06225378,
 author = {M{\"o}hle, Martin},
 title = {Duality and cones of {Markov} processes and their semigroups},
 fjournal = {Markov Processes and Related Fields},
 journal = {Markov Process. Relat. Fields},
 issn = {1024-2953},
 volume = {19},
 number = {1},
 pages = {149--162},
 year = {2013},
 zbMATH = {6225378},
 Zbl = {1298.60079}
}

@book{bernstein,
  title={Bernstein Functions: Theory and Applications},
  author={Schilling, Ren\'e  and Song, Renming and Vondra\v{c}ek, Zoran},
  isbn={9783110269338},
  lccn={2012006137},
  series={De Gruyter Studies in Mathematics},
  url={https://books.google.si/books?id=QVmHfZZQ0k0C},
  year={2012},
  publisher={De Gruyter}
}

@book{zbMATH05875699,
 title = {Markov processes, semigroups and generators.},
 author = {Kolokoltsov, Vassili N.},
 issn = {0179-0986},
 isbn = {978-3-11-025010-7; 978-3-11-025011-4},
 series = {De Gruyter Studies in Mathematics},
 year = {2011},
 publisher = {de Gruyter}
 }

@article{zbMATH07553689,
 author = {Le, Vi and Pardoux, Etienne},
 title = {Extinction time and the total mass of the continuous-state branching processes with competition},
 fjournal = {Stochastics},
 journal = {Stochastics},
 issn = {1744-2508},
 volume = {92},
 number = {6},
 pages = {852--875},
 year = {2020},
 
 doi = {10.1080/17442508.2019.1677661},
 zbMATH = {7553689},
 Zbl = {1490.60231}
}

@article{zbMATH08044135,
 author = {Depperschmidt, Andrej and Greven, Andreas and Pfaffelhuber, Peter},
 title = {Duality and the well-posedness of a martingale problem},
 fjournal = {Theoretical Population Biology},
 journal = {Theor. Popul. Biol.},
 issn = {0040-5809},
 volume = {159},
 pages = {59--73},
 year = {2024},
 doi = {10.1016/j.tpb.2024.07.003},
}

@misc{greven2015multitypespatialbranchingmodels,
      title={Multi-type spatial branching models for local self-regulation {I}: Construction and an exponential duality}, 
      author={Andreas Greven and Anja Sturm and Anita Winter and Iljana Zähle},
      year={2015},
      howpublished = {preprint, {arXiv}:1509.04023},
      url={https://arxiv.org/abs/1509.04023}
}

@article{FoRivWi2024,
 author = {Foucart, Cl{\'e}ment and Rivero, V{\'{\i}}ctor and Winter, Anita},
 title = {Conditioning the logistic continuous-state branching process on non-extinction via its total progeny},
 fjournal = {The Annals of Applied Probability},
 journal = {Ann. Appl. Probab.},
 issn = {1050-5164},
 volume = {35},
 number = {6},
 pages = {4172--4212},
 year = {2025},
 doi = {10.1214/25-AAP2216}
}

@article{zbMATH07734715,
 author = {Foucart, Cl{\'e}ment and Zhou, Xiaowen},
 title = {On the boundary classification of {{\(\Lambda\)}}-{Wright}-{Fisher} processes with frequency-dependent selection},
 fjournal = {Annales Henri Lebesgue},
 journal = {Ann. Henri Lebesgue},
 issn = {2644-9463},
 volume = {6},
 pages = {493--539},
 year = {2023},
 
 doi = {10.5802/ahl.170},
 zbMATH = {7734715},
 Zbl = {1521.60053}
}

@Article{zbMATH07470626,
 Author = {H\'el\`ene {Leman} and Juan Carlos {Pardo}},
 Title = {{Extinction time of logistic branching processes in a Brownian environment}},
 FJournal = {{ALEA. Latin American Journal of Probability and Mathematical Statistics}},
 Journal = {{ALEA, Lat. Am. J. Probab. Math. Stat.}},
 ISSN = {1980-0436},
 Volume = {18},
 Number = {2},
 Pages = {1859--1890},
 Year = {2021},
 Publisher = {Instituto de Matem\'atica Pura e Aplicada (IMPA), Rio de Janeiro},
 MSC2010 = {60J80 60J70 60J85}
}

@Article{zbMATH06963718,
 Author = {Hui {He} and Zenghu {Li} and Wei {Xu}},
 Title = {{Continuous-state branching processes in {L\'e}vy random environments}},
 FJournal = {{Journal of Theoretical Probability}},
 Journal = {{J. Theor. Probab.}},
 ISSN = {0894-9840},
 Volume = {31},
 Number = {4},
 Pages = {1952--1974},
 Year = {2018},
 Publisher = {Springer US, New York, NY},

 DOI = {10.1007/s10959-017-0765-1},
 MSC2010 = {60J80 60K37 60G51},
 Zbl = {1428.60125}
}

@Article{zbMATH07189529,
 Author = {Ji, Li-na and Li,  Zeng-hu},
 Title = {{Moments of continuous-state branching processes with or without immigration}},
 FJournal = {{Acta Mathematicae Applicatae Sinica. English Series}},
 Journal = {{Acta Math. Appl. Sin., Engl. Ser.}},
 ISSN = {0168-9673},
 Volume = {36},
 Number = {2},
 Pages = {361--373},
 Year = {2020},
 Publisher = {Springer, Berlin/Heidelberg; Chinese Mathematical Society, Beijing; Chinese Academy of Sciences, Academy of Mathematics \& Systems Science, Beijing},

 DOI = {10.1007/s10255-020-0935-2},
 MSC2010 = {60J80 60J85 60H20},
 Zbl = {1456.60227}
}

@article{zbMATH02119574,
 author = {Horridge, Paul and Tribe, Roger},
 title = {On stationary distributions for the {KPP} equation with branching noise},
 fjournal = {Annales de l'Institut Henri Poincar{\'e}. Probabilit{\'e}s et Statistiques},
 journal = {Ann. Inst. Henri Poincar{\'e}, Probab. Stat.},
 issn = {0246-0203},
 volume = {40},
 number = {6},
 pages = {759--770},
 year = {2004},
 
 doi = {10.1016/j.anihpb.2004.01.005},
 url = {https://eudml.org/doc/77832},
 zbMATH = {2119574},
 Zbl = {1058.60049}
}

@INCOLLECTION{jacob2001pseudo,
  booktitle={Volume {I}: {F}ourier {A}nalysis and {S}emigroups},
  author={Jacob, Niels},
  title={Pseudo Differential Operators \& {M}arkov Processes},
  year={2001},
  publisher={Imperial College Press}
}

@article{zbMATH07581733,
 author = {Cordero, Fernando and Hummel, Sebastian and Schertzer, Emmanuel},
 title = {General selection models: {Bernstein} duality and minimal ancestral structures},
 fjournal = {The Annals of Applied Probability},
 journal = {Ann. Appl. Probab.},
 issn = {1050-5164},
 volume = {32},
 number = {3},
 pages = {1499--1556},
 year = {2022},
 
 doi = {10.1214/21-AAP1683},
 zbMATH = {7581733},
 Zbl = {1503.60143}
}

@book {Swart,
    AUTHOR = {Swart, Jan M.},
     TITLE = {Duality and Intertwining of {M}arkov Chains},
    SERIES = {Lecture notes for the ALEA in Europe School October 21-25 2013, Luminy (Marseille)},
      YEAR = {2013}, 
       URL = {http://staff.utia.cas.cz/swart/tea index.html }}

@book{EthierKurtz,
    AUTHOR = {Ethier, Steven N. and Kurtz, Thomas G.},
     TITLE = {Markov {P}rocesses. {C}haracterization and  {C}onvergence},
    SERIES = {Wiley Series in Probability and Mathematical Statistics:
              Probability and Mathematical Statistics},
 PUBLISHER = {John Wiley \& Sons Inc.},
   ADDRESS = {New York},
      YEAR = {1986}
      }

@Article{zbMATH06836271,
 Author = {{Palau}, Sandra  and  {Pardo}, Juan Carlos},
 Title = {{Branching processes in a L\'evy random environment}},
 FJournal = {{Acta Applicandae Mathematicae}},
 Journal = {{Acta Appl. Math.}},
 ISSN = {0167-8019},
 Volume = {153},
 Number = {1},
 Pages = {55--79},
 Year = {2018},
 Publisher = {Springer Netherlands, Dordrecht},

 DOI = {10.1007/s10440-017-0120-7},
 MSC2010 = {60G17 60G51 60J80},
 Zbl = {1380.60043}
}

@article{duality,
author = {Sabine Jansen and Noemi Kurt},
title = {{On the notion(s) of duality for Markov processes}},
volume = {11},
journal = {Probability Surveys},
publisher = {Institute of Mathematical Statistics and Bernoulli Society},
pages = {59--120},
year = {2014},
doi = {10.1214/12-PS206},
URL = {https://doi.org/10.1214/12-PS206}
}

@article{zbMATH03418463,
 author = {Benveniste, Albert and Jacod, Jean},
 title ={Syst{\`e}mes de {L}\'evy des processus de {M}arkov},
 fjournal = {Inventiones Mathematicae},
 journal = {Invent. Math.},
 issn = {0020-9910},
 volume = {21},
 pages = {183--198},
 year = {1973},
 
 doi = {10.1007/BF01390195},
 url = {https://eudml.org/doc/142226},
 zbMATH = {3418463},
 Zbl = {0265.60074}
}

@book{jacod1987limit,
  title={Limit Theorems for Stochastic Processes},
  author={Jacod, Jean and Shiryaev, Albert N.},
  isbn={9783540178828},
  lccn={lc87009865},
  series={Grundlehren der mathematischen Wissenschaften},
  url={https://books.google.si/books?id=sUgXKpUIdHwC},
  year={1987},
  publisher={Springer Berlin Heidelberg}
}

@book{ikeda1989stochastic,
  title={Stochastic Differential Equations and Diffusion Processes},
  author={Ikeda, Nobuyuki and Watanabe, Shinzo},
  isbn={9780444873781},
  lccn={89212869},
  series={Kodansha scientific books},
  url={https://books.google.si/books?id=9RzvAAAAMAAJ},
  year={1989},
  publisher={North-Holland}
}

@article {MR0362529,
    AUTHOR = {Grimvall, Anders},
     TITLE = {On the convergence of sequences of branching processes},
   JOURNAL = {Ann. Probab.},
    VOLUME = {2},
      YEAR = {1974},
     PAGES = {1027--1045},
   MRCLASS = {60J80},
  MRNUMBER = {0362529},
MRREVIEWER = {D. R. Grey},
}

@article{FoMaMa2019,
 author = {Foucart, Cl{\'e}ment and Ma, Chunhua and Mallein, Bastien},
 title = {Coalescences in continuous-state branching processes},
 fjournal = {Electronic Journal of Probability},
 journal = {Electron. J. Probab.},
 issn = {1083-6489},
 volume = {24},
 pages = {no. 103, 1-52},
 year = {2019},
 doi = {10.1214/19-EJP358},
 zbMATH = {7142897},
 Zbl = {1427.60177}
}

@article{zbMATH07055657,
 author = {Bansaye, Vincent and Caballero, Maria-Emilia and M{\'e}l{\'e}ard, Sylvie},
 title = {Scaling limits of population and evolution processes in random environment},
 fjournal = {Electronic Journal of Probability},
 journal = {Electron. J. Probab.},
 issn = {1083-6489},
 volume = {24},
 pages = {38},
 year = {2019},
 
 doi = {10.1214/19-EJP262},
 zbMATH = {7055657},
 Zbl = {1466.60150}
}

@article{zbMATH06471553,
 author = {Bansaye, Vincent and Simatos, Florian},
 title = {On the scaling limits of {Galton}-{Watson} processes in varying environments},
 fjournal = {Electronic Journal of Probability},
 journal = {Electron. J. Probab.},
 issn = {1083-6489},
 volume = {20},
 pages = {36},
 year = {2015},
 
 doi = {10.1214/EJP.v20-3812},
 zbMATH = {6471553},
 Zbl = {1321.60171}
}

@misc{arXiv:2009.11820,
 author = {Berzunza Ojeda, Gabriel  and Pardo, Juan Carlos},
 title = {Branching processes with pairwise interactions},
 year = {2020},
 howpublished = {preprint, {arXiv}:2009.11820},
 url = {https://arxiv.org/abs/2009.11820},
 arXiv = {arXiv:2009.11820}
}

@article{zbMATH07458586,
 author = {Gonz{\'a}lez Casanova, Adri{\'a}n and Pardo, Juan Carlos and P{\'e}rez, Jos{\'e} Luis},
 title = {Branching processes with interactions: subcritical cooperative regime},
 fjournal = {Advances in Applied Probability},
 journal = {Adv. Appl. Probab.},
 issn = {0001-8678},
 volume = {53},
 number = {1},
 pages = {251--278},
 year = {2021},
 
 doi = {10.1017/apr.2020.59},
 zbMATH = {7458586},
 Zbl = {1490.60228}
}

@article {MR3940763,
    AUTHOR = {Foucart, Cl\'{e}ment},
     TITLE = {Continuous-state branching processes with competition: duality
              and reflection at infinity},
   JOURNAL = {Electron. J. Probab.},
  FJOURNAL = {Electronic Journal of Probability},
    VOLUME = {24},
      YEAR = {2019},
     PAGES = {no. 33, 1--38},
      ISSN = {1083-6489},
   MRCLASS = {60J80 (60J70 92D25)},
  MRNUMBER = {3940763},
MRREVIEWER = {B. L. Granovsky},
       DOI = {10.1214/19-EJP299},
       URL = {https://doi.org/10.1214/19-EJP299},
}

@Misc{zbMATH03736679,
 Author = {Samuel {Karlin} and Howard M. {Taylor}},
 Title = {{A second course in stochastic processes}},
 Year = {1981},

 HowPublished = {{Academic Press, New York}},
 MSC2010 = {60-01 60J25 60J10 60K25},
 Zbl = {0469.60001}
}

@article{MR0431386,
    AUTHOR = {Siegmund, David},
     TITLE = {The equivalence of absorbing and reflecting barrier problems
              for stochastically monotone {M}arkov processes},
   JOURNAL = {Ann. Probability},
    VOLUME = {4},
      YEAR = {1976},
    NUMBER = {6},
     PAGES = {914--924},
   MRCLASS = {60J05 (60J25)},
MRREVIEWER = {D. J. Daley},
}

@article {MR724061,
    AUTHOR = {Cox, J. Theodore and R\"osler, Uwe},
     TITLE = {A duality relation for entrance and exit laws for {M}arkov
              processes},
   JOURNAL = {Stochastic Process. Appl.},
  FJOURNAL = {Stochastic Processes and their Applications},
    VOLUME = {16},
      YEAR = {1984},
    NUMBER = {2},
     PAGES = {141--156},
      ISSN = {0304-4149},
   MRCLASS = {60J50 (60J10 60J60)},
  MRNUMBER = {724061},
MRREVIEWER = {Yves Le Jan},
       DOI = {10.1016/0304-4149(84)90015-2},
       URL = {http://dx.doi.org/10.1016/0304-4149(84)90015-2},
}

@book {Kallenberg,
    AUTHOR = {Kallenberg, Olav},
     TITLE = {Foundations of {M}odern {P}robability},
    SERIES = {Probability and its Applications},
   EDITION = {Second},
 PUBLISHER = {Springer-Verlag},
   ADDRESS = {New York},
      YEAR = {2002}
}

@article {DawsonLi,
    AUTHOR = {Dawson, Donald A. and Li, Zenghu},
     TITLE = {Stochastic equations, flows and measure-valued processes},
   JOURNAL = {Ann. Probab.},
  FJOURNAL = {The Annals of Probability},
    VOLUME = {40},
      YEAR = {2012},
    NUMBER = {2},
     PAGES = {813--857},
      ISSN = {0091-1798},
   MRCLASS = {60J68 (60H20 60J25 60J80)},
  MRNUMBER = {2952093},
MRREVIEWER = {Jos{\'e} Villa Morales},
       DOI = {10.1214/10-AOP629},
       URL = {http://dx.doi.org/10.1214/10-AOP629},
}

@article{zbMATH06098183,
 author = {Li, Zenghu and Pu, Fei},
 title = {Strong solutions of jump-type stochastic equations},
 fjournal = {Electronic Communications in Probability},
 journal = {Electron. Commun. Probab.},
 issn = {1083-589X},
 volume = {17},
 pages = {13},
 year = {2012},
 
 doi = {10.1214/ECP.v17-1915},
}

@book{zbMATH06256582,
 author = {B{\"o}ttcher, Bj{\"o}rn and Schilling, Ren{\'e} and Wang, Jian},
 title = {L{\'e}vy {M}atters {III}. {L{\'e}vy}-{T}ype {P}rocesses: {C}onstruction, {A}pproximation and {S}ample {P}ath {P}roperties},
 fseries = {Lecture Notes in Mathematics},
 series = {Lect. Notes Math.},
 issn = {0075-8434},
 volume = {2099},
 isbn = {978-3-319-02683-1; 978-3-319-02684-8},
 year = {2013},
 publisher = {Cham: Springer},
 
 doi = {10.1007/978-3-319-02684-8},
 zbMATH = {6256582},
 Zbl = {1384.60004}
}

@article{zbMATH06403246,
 author = {Carinci, Gioia and Giardin{\`a}, Cristian and Giberti, Claudio and Redig, Frank},
 title = {Dualities in population genetics: a fresh look with new dualities},
 fjournal = {Stochastic Processes and their Applications},
 journal = {Stochastic Processes Appl.},
 issn = {0304-4149},
 volume = {125},
 number = {3},
 pages = {941--969},
 year = {2015},
 
 doi = {10.1016/j.spa.2014.10.009},
 zbMATH = {6403246},
 Zbl = {1326.92056}
}

@incollection{zbMATH07279967,
 author = {Sturm, Anja and Swart, Jan M. and V{\"o}llering, Florian},
 title = {The algebraic approach to duality: an introduction},
 booktitle = {Genealogies of interacting particle systems. Papers based on lectures and turorials of the National University of Singapore, Singapore, July 17 -- Aug 18, 2017},
 isbn = {978-981-12-0608-5; 978-981-12-0610-8},
 pages = {81--150},
 year = {2020},
 publisher = {Hackensack, NJ: World Scientific},
 
 doi = {10.1142/9789811206092_0003},
 zbMATH = {7279967},
 Zbl = {1453.82059}
}

@book{dualityliebook,
 author = {Giardin{\`a}, Cristian and Redig, Frank},
 title = {Duality for {M}arkov processes:
a {L}ie algebraic approach},
 year = {2025},
 fseries = {Grundlehren der Mathematischen Wissenschaften},
 series = {Grundlehren Math. Wiss.},
 publisher = {Springer Cham},
 volume={365}
}

@book{zbMATH03892344,
 author = {Liggett, Thomas M.},
 title = {Interacting particle systems},
 fseries = {Grundlehren der Mathematischen Wissenschaften},
 series = {Grundlehren Math. Wiss.},
 issn = {0072-7830},
 volume = {276},
 year = {1985},
 publisher = {Springer, Cham},
 zbMATH = {3892344},
 Zbl = {0559.60078}
}

@article{zbMATH03573037,
 author = {Schwartz, Diane L.},
 title = {Applications of duality to a class of {Markov} processes},
 fjournal = {The Annals of Probability},
 journal = {Ann. Probab.},
 issn = {0091-1798},
 volume = {5},
 number={4},
 pages = {522--532},
 year = {1977},
 
 doi = {10.1214/aop/1176995758},
 zbMATH = {3573037},
 Zbl = {0367.60111}
}

@article{zbMATH01396203,
 author = {M{\"o}hle, Martin},
 title = {The concept of duality and applications to {Markov} processes arising in neutral population genetics models},
 fjournal = {Bernoulli},
 journal = {Bernoulli},
 issn = {1350-7265},
 volume = {5},
 number = {5},
 pages = {761--777},
 year = {1999},
 
 doi = {10.2307/3318443},
 zbMATH = {1396203},
 Zbl = {0942.92020}
}

@article {Lamperti2,
    AUTHOR = {Lamperti, John},
     TITLE = {Continuous state branching processes},
   JOURNAL = {Bull. Amer. Math. Soc.},
  FJOURNAL = {Bulletin of the American Mathematical Society},
    VOLUME = {73},
      YEAR = {1967},
     PAGES = {382--386},
      ISSN = {0002-9904},
   MRCLASS = {60.67},
  MRNUMBER = {0208685 (34 \#8494)},
MRREVIEWER = {P. E. Ney},
}

@incollection {subordinators,
    AUTHOR = {Bertoin, Jean},
     TITLE = {Subordinators: examples and applications},
 BOOKTITLE = {Lectures on probability theory and statistics
              ({S}aint-{F}lour, 1997)},
    SERIES = {Lecture Notes in Math.},
    VOLUME = {1717},
     PAGES = {1--91},
 PUBLISHER = {Springer},
   ADDRESS = {Berlin},
      YEAR = {1999},
   MRCLASS = {60-02 (60G17 60G51 60H10 60H15 60J55 60J65)},
  MRNUMBER = {1746300 (2002a:60001)},
}

@book{MR3185174,
    AUTHOR = {Sato, Ken-iti},
     TITLE = {L\'evy processes and infinitely divisible distributions},
    SERIES = {Cambridge Studies in Advanced Mathematics},
    VOLUME = {68},
 PUBLISHER = {Cambridge University Press, Cambridge},
      YEAR = {2013},
     PAGES = {xiv+521},
      ISBN = {978-1-107-65649-9},
   MRCLASS = {60G51 (60E07 60G18 60G52 60J45)},
  MRNUMBER = {3185174},
}

@article{zbMATH07188057,
 author = {Hermann, Felix and Pfaffelhuber, Peter},
 title = {Markov branching processes with disasters: extinction, survival and duality to {{\(p\)}}-jump processes},
 fjournal = {Stochastic Processes and their Applications},
 journal = {Stochastic Processes Appl.},
 issn = {0304-4149},
 volume = {130},
 number = {4},
 pages = {2488--2518},
 year = {2020},
 
 doi = {10.1016/j.spa.2019.07.011},
 zbMATH = {7188057},
 Zbl = {1434.60241}
}

@book{Ber96,
	Address = {Cambridge},
	Author = {Bertoin, Jean},
	Publisher = {Cambridge University Press},
	Series = {Cambridge Tracts in Mathematics},
	Title = {L\'evy processes},
	Volume = {121},
	Year = {1996}}

@article{KAW,
	Author = {Kawazu, Kiyoshi and Watanabe, Shinzo},
	journal = {Akademija Nauk SSSR. Teorija Verojatnoste\u\i\ i ee Primenenija},
	Pages = {34-51},
	Title = {Branching processes with immigration and related limit theorems},
	Volume = {16},
	number={1},
	Year = {1971}}

@article{Lambert,
	Author = {Lambert, Amaury},
	Coden = {PTRFEU},
	Doi = {10.1007/s004400100155},
	Fjournal = {Probability Theory and Related Fields},
	Issn = {0178-8051},
	Journal = {Probab. Theory Related Fields},
	Mrclass = {60J80},
	Mrnumber = {1883717 (2003a:60130)},
	Mrreviewer = {Fima Klebaner},
	Number = {1},
	Pages = {42--70},
	Title = {The genealogy of continuous-state branching processes with immigration},
	Url = {http://dx.doi.org/10.1007/s004400100155},
	Volume = {122},
	Year = {2002}}

@article{foucart2021local,
 author = {Foucart, Cl{\'e}ment},
 title = {Local explosions and extinction in continuous-state branching processes with logistic competition},
 fjournal = {Potential Analysis},
 journal = {Potential Anal.},
 issn = {0926-2601},
 volume = {64},
 number = {1},
 pages = {1--38},
 note = {article no. 30},
 year = {2026},
 doi = {10.1007/s11118-025-10237-w}
}

@article{foucartvidmar,
title = {Continuous-state branching processes with collisions: First passage times and duality},
 journal = {Stochastic Processes Appl.},
 fjournal = {Stochastic Processes and their Applications},
volume = {167},
pages = {no. 104230, 1-35},
year = {2024},
author = {Foucart, Cl\'ement and Vidmar, Matija},
}

@book {Kyprianoubook,
    AUTHOR = {Kyprianou, Andreas E.},
     TITLE = {Fluctuations of {L}\'evy {P}rocesses with {A}pplications. {I}ntroductory {L}ectures},
    SERIES = {Universitext},
   EDITION = {Second},
 PUBLISHER = {Springer, Heidelberg},
      YEAR = {2014},
     PAGES = {xviii+455},
      ISBN = {978-3-642-37631-3; 978-3-642-37632-0},
   MRCLASS = {60-01 (60E07 60E10 60Gxx)},
  MRNUMBER = {3155252},
MRREVIEWER = {Ren\'e L. Schilling},
       DOI = {10.1007/978-3-642-37632-0},
       URL = {http://dx.doi.org/10.1007/978-3-642-37632-0}
}

@book{zbMATH02006037,
 author = {Protter, Philip E.},
 title = {Stochastic integration and differential equations},
 edition = {second},
 fseries = {Applications of Mathematics},
 series = {Appl. Math. (N. Y.)},
 volume = {21},
 isbn = {3-540-00313-4; 978-3-642-05560-7},
 year = {2004},
 publisher = {Berlin: Springer},
 zbMATH = {2006037},
 Zbl = {1041.60005}
}

@article{zbMATH06247275,
 author = {Bansaye, Vincent and Pardo Millan, Juan Carlos and Smadi, Charline},
 title = {On the extinction of continuous state branching processes with catastrophes},
 fjournal = {Electronic Journal of Probability},
 journal = {Electron. J. Probab.},
 issn = {1083-6489},
 volume = {18},
 pages = {no. 103, 1-31},
 year = {2013},
 
 doi = {10.1214/EJP.v18-2774},
 zbMATH = {6247275},
 Zbl = {1286.60083}
}

@article{zbMATH03294035,
 author = {Silverstein, Martin L.},
 title = {A new approach to local times},
 fjournal = {Journal of Mathematics and Mechanics},
 journal = {J. Math. Mech.},
 issn = {0095-9057},
 volume = {17},
 number={11},
 pages = {1023--1054},
 year = {1968},
 
 zbMATH = {3294035},
 Zbl = {0184.41101}
}

@misc{zbMATH03249220,
 author = {Courr\`ege, Philippe},
 title = {Sur la forme int{\'e}gro-diff{\'e}rentielle des op{\'e}rateurs de {{\(C_ k^ \infty\)}} dans {{\(C\)}} satisfaisant au principe du maximum},
 year = {1967},
 pages={1-38},
 
 howpublished = {Semin. {Theorie} {Potentiel} {M}. {Brelot}, {G}. {Choquet} et {J}. {Deny} 10 (1965/66), {No}. 2},
 url = {https://eudml.org/doc/109338},
 zbMATH = {3249220},
 Zbl = {0155.17402}
}

@article{zbMATH07120715,
 author = {Li, Pei-Sen and Yang, Xu and Zhou, Xiaowen},
 title = {A general continuous-state nonlinear branching process},
 fjournal = {The Annals of Applied Probability},
 journal = {Ann. Appl. Probab.},
 issn = {1050-5164},
 volume = {29},
 number = {4},
 pages = {2523--2555},
 year = {2019},
 
 doi = {10.1214/18-AAP1459},
 zbMATH = {7120715},
 Zbl = {1466.60179}
}

@article{Rebotier,
author= {Rebotier, F\'elix},
title= {On the coming down from infinity  of
continuous-state branching processes with drift-interaction},
note = {arXiv:2510.05958, 2025}
}

@article{foucartrebotier,
author= {Foucart, Cl\'ement and Rebotier, F\'elix},
title= {Continuous-state branching processes with {L\'e}vy-{K}hintchine drift-interaction: Laplace duality and {F}ellerian extensions},
note = {In preparation, 2025}
}

@article{zbMATH05212987,
 author = {Alkemper, Roland and Hutzenthaler, Martin},
 title = {Graphical representation of some duality relations in stochastic population models},
 fjournal = {Electronic Communications in Probability},
 journal = {Electron. Commun. Probab.},
 issn = {1083-589X},
 volume = {12},
 pages = {206--220},
 year = {2007},
 
 doi = {10.1214/ECP.v12-1283},
 url = {https://eudml.org/doc/128368},
 zbMATH = {5212987},
 Zbl = {1129.60093}
}

@article {MR2308333,
    AUTHOR = {Hutzenthaler, Martin and Wakolbinger, Anton},
     TITLE = {Ergodic behavior of locally regulated branching populations},
   JOURNAL = {Ann. Appl. Probab.},
  FJOURNAL = {The Annals of Applied Probability},
    VOLUME = {17},
      YEAR = {2007},
    NUMBER = {2},
     PAGES = {474--501},
      ISSN = {1050-5164},
   MRCLASS = {60K35 (60J60 60J80 92D25)},
  MRNUMBER = {2308333},
MRREVIEWER = {Jan M. Swart},
       URL = {https://doi.org/10.1214/105051606000000745}
}

@article{zbMATH00496250,
 author = {Pakes, Anthony G.},
 title = {Explosive {Markov} branching processes: {Entrance} laws and limiting behaviour},
 fjournal = {Advances in Applied Probability},
 journal = {Adv. Appl. Probab.},
 issn = {0001-8678},
 volume = {25},
 number = {4},
 pages = {737--756},
 year = {1993},
 
 doi = {10.2307/1427789},
 zbMATH = {496250},
 Zbl = {0796.60089}
}

@article{zbMATH05946936,
 author = {Schilling, Ren\'e L. and Schnurr, Alexander},
 title = {The symbol associated with the solution of a stochastic differential equation},
 fjournal = {Electronic Journal of Probability},
 journal = {Electron. J. Probab.},
 issn = {1083-6489},
 volume = {15},
 pages = {no. 43, 1369--1393},
 year = {2010},
 
 doi = {10.1214/EJP.v15-807},
 zbMATH = {5946936},
 Zbl = {1226.60116}
}

@article{zbMATH00683136,
 author = {Hoh, Walter},
 title = {The martingale problem for a class of pseudo differential operators},
 fjournal = {Mathematische Annalen},
 journal = {Math. Ann.},
 issn = {0025-5831},
 volume = {300},
 number = {1},
 pages = {121--147},
 year = {1994},
 
 doi = {10.1007/BF01450479},
 url = {https://eudml.org/doc/165245},
 zbMATH = {683136},
 Zbl = {0805.47045}
}

@article{zbMATH07544388,
 author = {Foucart, Cl{\'e}ment and M{\"o}hle, Martin},
 title = {Asymptotic behaviour of ancestral lineages in subcritical continuous-state branching populations},
 fjournal = {Stochastic Processes and their Applications},
 journal = {Stochastic Processes Appl.},
 issn = {0304-4149},
 volume = {150},
 pages = {510--531},
 year = {2022},
 
 doi = {10.1016/j.spa.2022.05.001},
 zbMATH = {7544388},
 Zbl = {1491.60149}
}

@article{zbMATH02076723,
 author = {Kolokoltsov, Vassili N.},
 title = {On {Markov} processes with decomposable pseudo-differential generators.},
 fjournal = {Stochastics and Stochastics Reports},
 journal = {Stochastics Stochastics Rep.},
 issn = {1045-1129},
 volume = {76},
 number = {1},
 pages = {1--44},
 year = {2004},
 
 doi = {10.1080/10451120410001661250},
 url = {wrap.warwick.ac.uk/39211/1/WRAP_Kolokoltsov_KolokMarkovdecomposable.pdf},
 zbMATH = {2076723},
 Zbl = {1091.60014}
}

\noindent \textbf{Acknowledgements.}   MV acknowledges funding from the Slovenian Research and Innovation Agency (ARIS) under the research programme No. P1-0448 and is grateful to Alex Watson for the suggestion of using the Esscher transform in the  proof of Proposition~\ref{proposition:closure-under-limits}. This paper was completed while CF was affiliated with the Centre de  Math\'ematiques Appliqu\'ees (CMAP) at \'Ecole Polytechnique. He is grateful to Sylvie M\'el\'eard for  the opportunity to join her team and to Vincent Bansaye, Carl Graham and Anita Winter for several insightful discussions related to this project. CF is supported  by the European Union (ERC, SINGER, 101054787). Views and opinions expressed are however those of the authors only and do not necessarily reflect those of the European Union or the European Research Council. Neither the European Union nor the granting authority can be held responsible for them. 

\end{document}